\documentclass[11pt,reqno]{amsart}

\usepackage{xifthen}
\usepackage{pkgfile}
\usepackage{multirow,rotating,lscape}
\usepackage{mathrsfs}
\usepackage{subfig}

\newcommand{\rrvert}{\vert}
\newcommand{\rrVert}{\Vert}
\newcommand{\llvert}{\vert}
\newcommand{\llVert}{\Vert}

\allowdisplaybreaks

\title[Two-Sample Tests Based on Geometric Graphs]{Asymptotic Distribution and Detection Thresholds for Two-Sample Tests Based on Geometric Graphs}

\author{Bhaswar B. Bhattacharya}
\address{Department of Statistics, University of Pennsylvania, Philadelphia, USA,
{\tt bhaswar@wharton.upenn.edu}}

\begin{document}

\begin{abstract}
In this paper, we consider the problem of testing the equality of two
multivariate distributions based on geometric graphs 
constructed using the interpoint distances between the observations.
These include the tests based on the minimum spanning tree
and the $K$-nearest neighbor (NN) graphs, among others. These tests are
asymptotically distribution-free, universally
consistent and computationally efficient, making them particularly
useful in modern applications. However, very
little is known about the power properties of these tests. In this
paper, using the theory of stabilizing
geometric graphs, we derive the asymptotic distribution of these tests
under general alternatives,
in the Poissonized setting. Using this, the detection threshold and the
limiting local power of the test based on the
$K$-NN graph are obtained, where interesting exponents depending on
dimension emerge. This provides a way to
compare and justify the performance of these tests in different examples.
\end{abstract}

%
%
%
%

\maketitle

\section{Introduction}
\label{sec1}

Given independent and identically distributed samples
%
\begin{align}
\label{2} \sX_{N_{1}}=\{X_{1}, X_{2},
\ldots, X_{N_{1}}\} \quad\text{and} \quad\sY _{N_{2}}=
\{Y_{1}, Y_{2}, \ldots, Y_{N_{2}}\},
\end{align}
from two unknown densities $f$ and $g$ (with respect to the Lebesgue
measure) in $\R^{d}$, respectively, the two-sample problem is to test
the hypotheses
%
\begin{equation}
H_{0}: f=g\quad \text{versus}\quad H_{1}: f\ne g.
\label{test}
\end{equation}
In this paper, we will derive asymptotic properties of two-sample tests
based on geometric graphs in the usual limiting regime where the
dimension $d$ is fixed and the sample size $N_{1}+N_{2}:=N\rightarrow
\infty$, such that
%
\begin{align}
\label{pq} \frac{N_{1}}{N_{1}+N_{2}}\rightarrow p\in(0, 1),\qquad
\frac
{N_{2}}{N_{1}+N_{2}}\rightarrow q:=1-p.
\end{align}

For univariate data, there are several well-known nonparametric tests
such as the two-sample Kolmogorov--Smirnoff maximum deviation test
\cite{smirnoff}, the Wald--Wolfowitz runs test \cite{ww} and the
Mann--Whitney rank test \cite{mann_whitney}.

The nonparametric two-sample problem for multivariate data has been extensively
studied, beginning with the work of Weiss \cite{weiss} and Bickel
\cite{bickel}. Friedman and Rafsky \cite{fr} generalized the
Wald--Wolfowitz runs test \cite{ww} to higher dimensions using the
Euclidean minimal spanning tree (MST) of the pooled data. Thereafter,
many other two-sample tests based on geometric graphs have been
proposed. Schilling \cite{schilling} and Henze \cite{henzenn} considered 
tests based on the $K$-nearest neighbor ($K$-NN) graph of the pooled
sample. Later, Rosenbaum \cite{rosenbaum} developed a test based on
matchings, and, more recently, Biswas et al. \cite{tsp} proposed a
test based on the Hamiltonian cycle, both of which are exactly
distribution-free under the null. Recently, Chen and Friedman \cite
{chenfriedman} proposed new modifications of these tests for
high-dimensional and object data. Maa et al. \cite{interpoint}
provided certain theoretical motivations for using tests based on
interpoint distances.

Another class of multivariate two-sample tests is the Liu--Singh rank
sum statistics \cite{liu_singh}, which generalize the Mann--Whitney
rank test using the notion of data depth \cite{tukey,liu_singh}.
For other popular two-sample tests, refer to \cite
{aslan_zech,bfranz,gretton,hall_tajvidi,rousson} and the references
therein. The problem of testing the equality of two discrete
distributions has also been extensively studied in recent years \cite
{batu,CDVV}.

\subsection{Graph-based two-sample tests}
\label{graph2tests}

Many of the tests mentioned above can be studied in the general
framework of graph-based two-sample tests \cite{BBB}, which include
the tests based on geometric graphs, as well as those based on data
depth. To this end, we have the following definition: A \emph{graph
functional} $\sG$ in $\R^{d}$ defines a graph for all finite subsets
of $\R^{d}$, that is, given $S\subset\R^{d}$ finite, $\sG(S)$ is a
graph with vertex set $S$. A graph functional is said to be \emph
{undirected/directed} if the graph $\sG(S)$ is an undirected/directed
graph with vertex set $S$. We assume that $\sG(S)$ has no self loops
and multiple edges, that is, no edge is repeated more than once in the
undirected case, and no edge in the same direction is repeated more
than once in the directed case. The set of edges in the graph $\sG(S)$
will be denoted by $E(\sG(S))$.

%
\begin{defn}[Bhattacharya \cite{BBB}]\label{graph2} Let $\sX
_{N_{1}}$ and $\sY_{N_{2}}$ be i.i.d. samples of size $N_{1}$ and
$N_{2}$ from densities $f$ and $g$, respectively, as in (\ref{2}). The
\emph{2-sample test statistic based on the graph functional $\sG$} is
defined as
%
\begin{align}
\label{graph2test} T\bigl(\sG(\sX_{N_{1}}\cup\sY_{N_{2}})\bigr):=
\sum_{i=1}^{N_{1}} \sum
_{j=1}^{N_{2}} \pmb1\bigl\{(X_{i},
Y_{j})\in E\bigl(\sG(\sX_{N_{1}}\cup\sY _{N_{2}})
\bigr)\bigr\}.
\end{align}
\end{defn}
If $\sG$ is an undirected graph functional, then the statistic \eqref
{graph2test} counts the number of edges in the graph $\sG(\sX
_{N_{1}}\cup\sY_{N_{2}})$ with one end point in $\sX_{N_{1}}$ and
the other in $\sY_{N_{2}}$. If $\sG$ is a directed graph functional,
then \eqref{graph2test} is the number of directed edges with the
outward end in $\sX_{N_{1}}$ and the inward end in $\sY_{N_{2}}$. The
null hypothesis is generally rejected for ``small'' values of the
statistic \eqref{graph2test}. This includes the Friedman--Rafsky (FR)
test  \cite{fr} (based on the MST), the test based on the $K$-NN graph
\cite{henzenn,schilling}, the cross match test \cite{rosenbaum}
(based on minimum non-bipartite matching), among others. These tests
are asymptotically distribution-free, universally consistent and
computationally efficient (both in sample size and in dimension),
making them particularly attractive for modern statistical applications.

\subsection{Poissonization}
\label{sec:poissonization}

In the Poissonized setting, instead of taking $N_{1}$ samples from the
density $f$ and $N_{2}$ from the density $g$, we have $\dPois(N_{1})$
from $f$ and $\dPois(N_{2})$ samples from $g$. To this end, suppose
$\sX=\{X_{1}, X_{2}, \ldots\}$ and $\sY=\{Y_{1}, Y_{2}, \ldots\}$
be i.i.d. samples from $f$ and $g$, respectively, and
%
\begin{align}
\sX_{N_{1}}'=\{X_{1}, X_{2},
\ldots, X_{L_{N_{1}}}\}\quad \text{and}\quad \sY _{N_{2}}'=
\{Y_{1}, Y_{2}, \ldots, Y_{L_{N_{2}}}\},
\label{2poisson}
\end{align}
where $L_{N_{1}}\sim\dPois(N_{1})$ and $L_{N_{2}} \sim\dPois
(N_{2})$ are independent of each other, and of $\sX$ and $\sY$.
Poissonization is a common assumption in geometric probability, which
simplifies calculations, due to the spatial independence of the Poisson
process, and yields cleaner formulas for the asymptotic variances. One
can expect to de-Poissionize the limit theorems derived below, using
well-known de-Poissonization methods \cite{penrosebook,penroseclt}.
However, de-Poissonization does not affect the rates of convergence,
and the detection thresholds obtained below would remain unchanged (see
Remark~\ref{rem:poisson} for more on de-Poissonization).

Given a graph functional $\sG$, the Poissonized two-sample statistic is
defined as
%
\begin{align}
\label{graph2testp} T\bigl(\sG\bigl(\sX_{N_{1}}'\cup
\sY_{N_{2}}'\bigr)\bigr):=\sum
_{i=1}^{L_{N_{1}}} \sum_{j=1}^{L_{N_{2}}}
\pmb1\bigl\{(X_{i}, Y_{j})\in E\bigl(\sG\bigl(
\sX'_{N_{1}}\cup \sY'_{N_{2}}\bigr)
\bigr)\bigr\}.
\end{align}
The distribution of this statistic can be described as follows: Let
$\phi_{N}(x):=\frac{N_{1}}{N} f(x)+ \frac{N_{2}}{N} g(x)$ and
$Z_{1}, Z_{2}, \ldots,$ be independent random variables with common
density $\phi_{N}(\cdot)$. Let $L_{N}$ be an independent Poisson
variable with mean $N_{1}+N_{2}$. Then $\cZ_{N}'=\{Z_{1}, Z_{2},
\ldots, Z_{L_{N}}\}$ is a nonhomogeneous Poisson process in $\R^{d}$
with rate function $N\phi_{N}=N_{1} f+N_{2} g$. Label each point of
$z\in\cZ_{N}'$ independently with
%
\begin{align}
\label{pooledlabelalternative} c_{z}= %
\begin{cases} 1 &
\text{with probability } \frac{N_{1} f(z)}{N_{1} f(z)+N_{2}
g(z)},
\\
2 & \text{with probability } \frac{N_{2} g(z)}{N_{1} f(z)+N_{2} g(z)}. \end{cases}
\end{align}
Then the sets of points assigned labels 1 and 2 have the same
distribution as $\sX_{N_{1}}'$ and $\sY_{N_{2}}'$ (as in \eqref
{2poisson}), respectively. This implies that for a directed graph
functional $\sG$, the Poissonized 2-sample test statistic \eqref
{graph2testp} is equal in distribution to
%
\begin{align}
\label{Tpoisson} T\bigl(\sG\bigl(\cZ_{N}'\bigr)
\bigr) = \sum_{x, y \in\cZ_{N}'} \psi(c_{x},
c_{y}) \bm 1\bigl\{(x, y)\in E\bigl(\sG\bigl(\cZ_{N}'
\bigr)\bigr)\bigr\},
\end{align}
where $\psi(c_{x}, c_{y})= \bm 1\{c_{x}=1, c_{y}=2\}$. (Note that every
undirected graph functional $\sG$ can be modified to a directed graph
functional $\sG_{+}$ in a natural way: For $S\subset\R^{d}$ finite,
$\sG_{+}(S)$ is obtained by replacing every edge in $\sG(S)$ with two
directed edges, one in each direction. Thus, without loss of
generality, it suffices to consider directed graph functionals.)

Denote by $\E_{H_{0}}$ and $\E_{H_{1}}$ the expectation under the
null and the alternative, respectively. For a directed graph functional
$\sG$,
\[
\E_{H_{0}}\bigl(T\bigl(\sG\bigl(\cZ_{N}'
\bigr)\bigr)\bigr)=\frac{N_{1}N_{2}}{(N_{1}+N_{2})^{2}}\E \bigl( \bigl\llvert E\bigl(\sG\bigl(
\cZ_{N}'\bigr)\bigr) \bigr\rrvert \bigr),
\]
where $|E(\sG(\cZ_{N}'))|$ denotes the number of edges in the graph
$\sG(\cZ_{N}')$. For example, in the MST functional, $\E(|E(\cT
(\cZ_{N}'))|)=N-1$, and in the directed $K$-NN graph functional $\E
(|E(\cN_{K}(\cZ_{N}'))|)=KN$, respectively. (Formal definitions of
these graph-functionals are given in Section~\ref{sec:stabilizing_graphs} below.) We will see later in Section~\ref{distribution} that for many geometric graphs, such as the MST and the
$K$-NN graph, the statistic $T(\sG(\cZ_{N}'))$ is asymptotically
normal and distribution-free under the null $H_{0}$, that is,
$N^{-\frac{1}{2}}\{ T(\sG(\cZ_{N}'))- \E_{H_{0}}(T(\sG(\cZ
_{N}')))\} \dto N(0, \sigma_\sG^{2})$,
where $\sigma_\sG$ depends on the graph functional $\sG$, but not on
the unknown null distribution. For such a graph functional $\sG$, the
asymptotically level $\alpha$-test rejects $H_{0}$ when
%
\begin{align}
\label{rejregionN} \frac{1}{\sqrt N} \bigl\{ T\bigl(\sG\bigl(
\cZ_{N}'\bigr)\bigr)- \E_{H_{0}}\bigl(T\bigl(
\sG\bigl(\cZ _{N}'\bigr)\bigr)\bigr) \bigr\} \leq
\sigma_\sG z_{\alpha},
\end{align}
where $z_{\alpha}$ is the standard normal quantile of level $\alpha$.

\subsection{Stabilizing graphs}
\label{sec:stabilizing_graphs}

Many geometric graphs such as the MST and the $K$-NN graph, have local
dependence, that is, addition/deletion of a point only effects the
edges incident on the neighborhood of that point. This was formalized
by Penrose and Yukich \cite{py}, using the notion of stabilization. To
describe this, a few definitions are needed: A subset $S\subset\R
^{d}$ is said to be \emph{locally finite}, if $S\cap C$ is finite, for all
compact subsets $C\subset\R^{d}$. A locally finite set $S\subset\R
^{d}$ is \emph{nice} if all the interpoint distances among elements of
$S$ are distinct. If $S$ is a set of $N$ i.i.d. points $W_{1}, W_{2},
\ldots, W_{N}$ from some continuous distribution function $F$, then
the distribution of $ \llVert  W_{1}-W_{2} \rrVert  $ does not have any point mass, and
$S$ is nice almost surely.

Let $\sG$ be a graph functional defined for all locally finite subsets
of $\R^{d}$. For $S\subset\R^{d}$ nice and $x\in\R^{d}$, let $E(x,
\sG(S))$ be the set edges incident on $x$ in $\sG(S\cup\{x\})$. Note
that $|E(x, \sG(S))|:=d(x, \sG(S))$, the (total) degree of the vertex
$x$ in $\sG(S\cup\{x\})$. Finally, note that two graphs $H_{1},
H_{2}$ are said to be isomorphic if there is a bijection $\phi$ from the
vertex set of $H_{1}$ to the vertex set of $H_{2}$ such that any two
vertices $u$ and $v$ of $H_{1}$ are adjacent in $H_{1}$ if and only if
$\phi(u)$ and $\phi(v)$ are adjacent in $H_{2}$.

%
\begin{defn}\label{ts}Given $S\subset\R^{d}$, $y\in\R^{d}$, and
$a\in\R$, denote by $y+S=\{y+z: z\in S\}$ and $aS=\{az: z\in S\}$. A
graph functional $\sG$ is said to be \emph{translation invariant} if
the graphs $\sG(x+S)$ and $\sG(S)$ are isomorphic for all points
$x\in\R^{d}$ and all locally finite $S\subset\R^{d}$. A graph
functional $\sG$ is \emph{scale invariant} if $\sG(aS)$ and $\sG
(S)$ are isomorphic for all points $a\in\R$ and and all locally
finite $S\subset\R^{d}$.
\end{defn}

For $\lambda\geq0$, denote by $\cP_\lambda$ the homogeneous Poisson
process of intensity $\lambda$ in $\R^{d}$, and $\cP_{\lambda}^{x}
:= \cP_{\lambda} \cup\{x\}$, for $x\in\R^{d}$. Penrose and Yukich
\cite{py} defined stabilization of graph functionals over homogeneous
Poisson processes as follows.

%
\begin{defn}[Penrose and Yukich \cite{py}] \label{stabilization} A
translation and scale invariant graph functional $\sG$ \emph
{stabilizes} on $\cP_\lambda$ if, for almost all realizations $\cP
_\lambda$, there exists $R:=R(\cP_\lambda)<\infty$ such that
%
\begin{equation}
E\bigl(0, \sG\bigl(\cP_\lambda^{0}\bigr)\bigr)
\stackrel{a.s.}=E\bigl(0, \sG\bigl(\cP_\lambda ^{0}\cap
B(0, R)\cup\sA\bigr)\bigr),
\end{equation}
for all finite $\sA\subset\R^{d} \setminus B(0, R)$, where $B(0, R)$
is the (Euclidean) ball of radius $R$ centered at the origin $0\in\R^{d}$.
\end{defn}

Informally, a graph functional is stabilizing if addition of finitely
many points outside a ball of finite radius centered at the origin,
does not effect the set of edges incident at the origin. The $K$-NN
graph and the minimum spanning tree are known to be stabilizing (\cite{py}, Lemma~2.1). We discuss the two-sample tests associated with these
graphs below.

\subsubsection{Friedman--Rafsky (FR) test}
\label{sec1.3.1}

Friedman and Rafsky \cite{fr} generalized the Wald and Wolfowitz runs
test to higher dimensions by using the Euclidean minimal spanning tree
of the pooled sample.

%
\begin{defn}\label{mst}
Given a nice finite set $S \subset\R^{d}$, a \emph{spanning tree} of
$S$ is a connected graph $\cT$ with vertex-set $S$ and no cycles. The
\emph{length} $w(\cT)$ of $\cT$ is the sum of the Euclidean lengths
of the edges of $\cT$. A \emph{minimum spanning tree} (MST) of $S$,
denoted by $\cT(S)$, is a spanning tree with the smallest length, that
is, $w(\cT(S))\leq w(\cT)$ for all spanning trees $\cT$ of $S$.
\end{defn}

Thus, $\cT$ defines a graph functional in $\R^{d}$, and given $\sX
_{N_{1}}'$ and $\sY_{N_{2}}'$ as in \eqref{2poisson}, the FR-test
rejects $H_{0}$ for small values of
%
\begin{align}
\begin{split}
T\bigl(\cT\bigl(\cZ_{N}'\bigr)\bigr)={}&\sum
_{x, y \in\cZ_{N}'}\pmb1\{c_{x}\ne c_{y}\} \bm1
\bigl\{(x, y)\in E\bigl(\cT\bigl(\cZ_{N}'\bigr)\bigr)
\bigr\},
\\
={}&\sum_{x, y \in\cZ_{N}'} \psi(c_{x},
c_{y}) \pmb1\bigl\{(x, y)\in E\bigl(\cT_{+}\bigl(
\cZ_{N}'\bigr)\bigr)\bigr\}, \label{Tfr}
\end{split}
\end{align}
where $\cZ_{N}'=\sX_{N_{1}}'\cup\sY_{N_{2}}'$ and $\cT_{+}(\cZ
_{N}')$ is obtained by replacing every (undirected) edge in $\cT(\cZ
_{N}')$ with two directed edges, one in each direction. Note that this
counts the number of edges in the MST of the pooled sample with one
end-point in sample 1 and the other end-point in sample 2, which is
expected to be small when the two distributions are different. Note
that this reduces to the well-known Wald--Wolfowitz runs test when
dimension $d=1$, where the MST is the path through the data.

Friedman and Rafsky \cite{fr} calibrated \eqref{Tfr} as a permutation
test, and showed that it has good power in finite sample simulations.
Later, Henze and Penrose \cite{henzepenrose} proved that the statistic
$T(\cT(\cZ_{N}'))$ is asymptotically normal under $H_{0}$ and is
consistent under all fixed alternatives.

\subsubsection{Test based on the $K$-nearest neighbor ($K$-NN) graph}
\label{sec1.3.2}
As in \eqref{Tfr}, a multivariate two-sample test can be constructed
using the $K$-nearest neighbor graph of $\cZ_{N}'$. This was
originally suggested by Friedman and Rafsky \cite{fr}, and later
studied by Schilling \cite{schilling} and Henze \cite{henzenn}.

%
\begin{defn}\label{knn}
Given a nice finite set $S \subset\R^{d}$, the (directed) $K$-nearest
neighbor graph ($K$-NN) is a graph with vertex set $S$ with a directed
edge $(a, b)$, for $a, b \in S$, if the Euclidean distance between $a$
and $b$ is among the $K$-th smallest distances from $a$ to any other
point in $S$. Denote the directed $K$-NN of $S$ by $\cN_{K}(S)$.
\end{defn}

Given $\cZ_{N}'=\sX_{N_{1}}' \cup\sY_{N_{2}}'$ as in \eqref
{2poisson}, the $K$-NN statistic is
%
\begin{align}
\label{Tknn} T\bigl(\cN_{K}\bigl(\cZ_{N}'
\bigr)\bigr)=\sum_{x, y \in\cZ_{N}'} \psi(c_{x},
c_{y}) \pmb1\bigl\{(x, y)\in E\bigl(\cN_{K}\bigl(
\cZ_{N}'\bigr)\bigr)\bigr\}.
\end{align}

As before, when the two distributions are different, the number of
directed edges starting from sample 1 and ending in sample 2 will be
small (see Figure~\ref{3nn}), so the $K$-NN test rejects $H_{0}$ for
small values of \eqref{Tknn}. This will be our main running example
throughout the paper.

\begin{figure*}[h] 
\centering
\begin{minipage}[l]{0.495\textwidth}
\centering
\includegraphics[width=2.0in]
    {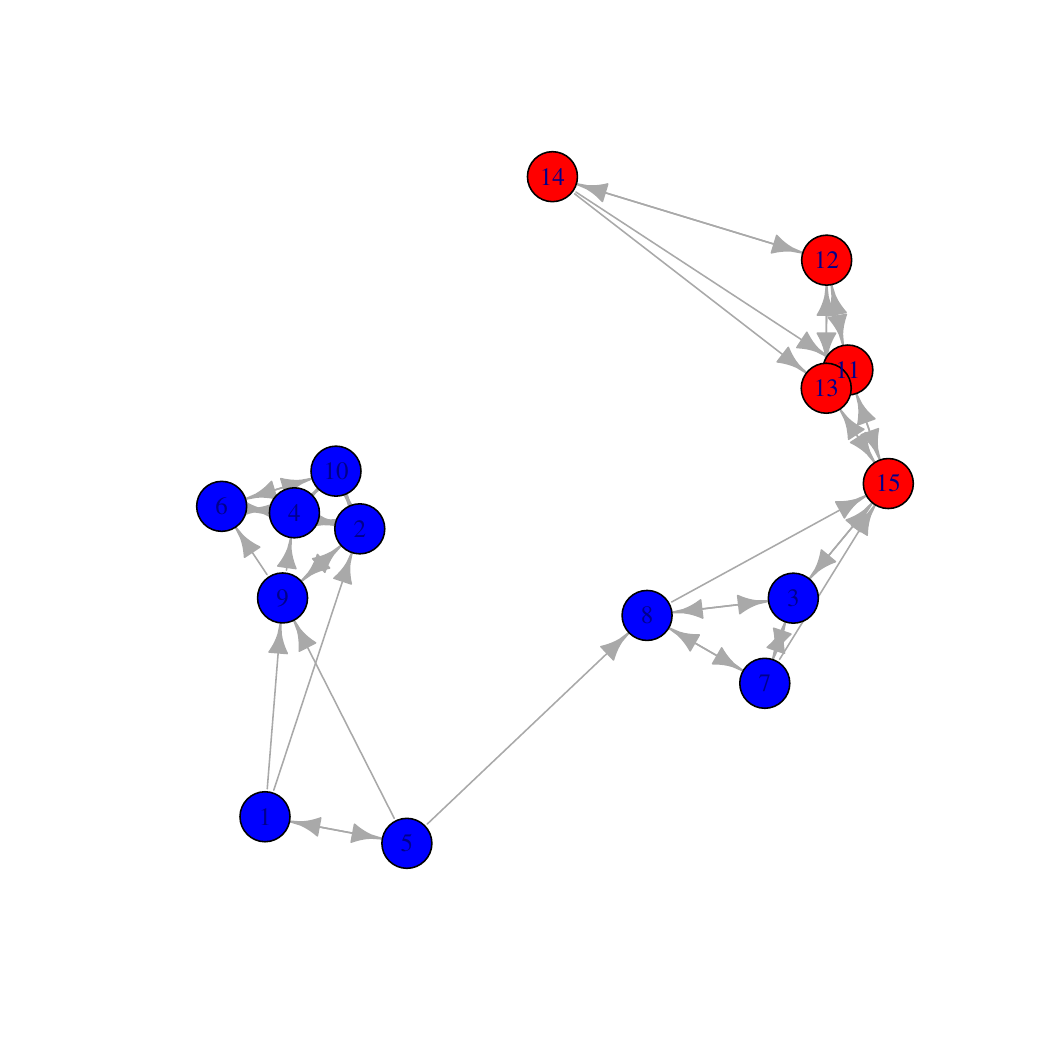}\\
\vspace{-0.2in}\scriptsize{(a)}
\end{minipage}
\begin{minipage}[c]{0.495\textwidth}
\centering
\includegraphics[width=2.0in]
    {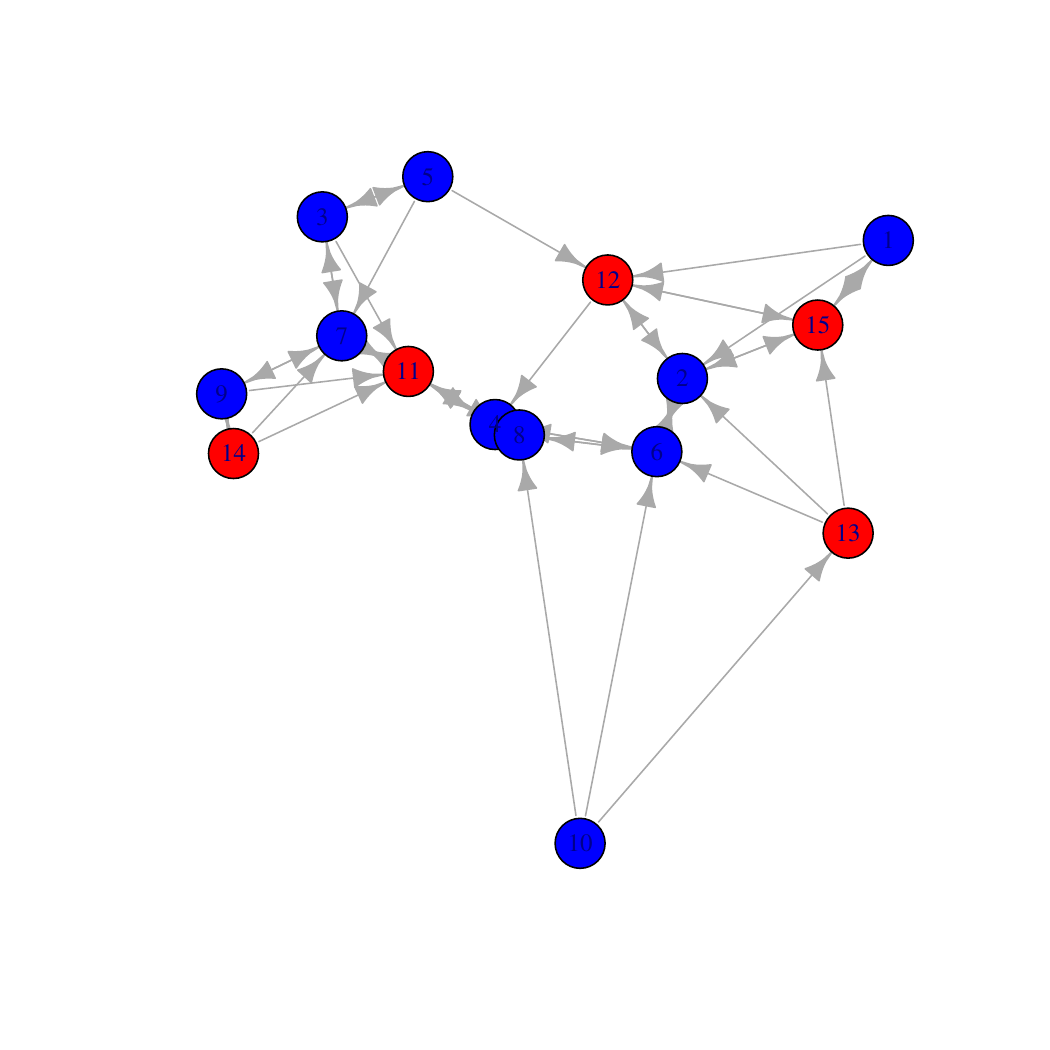}\\
\vspace{-0.2in}\scriptsize{(b)}
\end{minipage}
\caption{\small{The directed 3-NN graph on a pooled sample of size 15 in $\R^2$ with 10 i.i.d. points from $N(0, \mathrm I_2)$ (colored blue) and 5 i.i.d. points from $N(\Delta\cdot \pmb 1, \mathrm I_2)$ (colored red). For (a) $\Delta=2$, there are 3 directed edges starting from sample 1 and ending in sample 2,  and  for (b) $\Delta=0.05$, there are 8 such edges. }}
\label{3nn}
\end{figure*}

Another variant is the \emph{symmetrized $K$-NN test statistic} \cite
{schilling}:
%
\begin{align}
\label{Tknnsymmetry} T_{S}\bigl(\cN_{K}\bigl(
\cZ_{N}'\bigr)\bigr)=\sum
_{x, y \in\cZ_{N}'} \psi_{S}(c_{x},
c_{y}) \pmb1\bigl\{(x, y)\in E\bigl(\cN_{K}\bigl(
\cZ_{N}'\bigr)\bigr)\bigr\},
\end{align}
where $\psi_{S}(c_{x}, c_{y})= \bm1\{c_{x}\ne c_{y}\}$, which counts
the number of (directed) edges with the end-points in the different
samples. This can be rewritten as a graph-based test (\ref{Tpoisson})
by considering the underlying undirected multigraph (which allows for
multiple edges between two vertices).

\subsection{Summary of results}
\label{sec1.4}

The asymptotic null distribution and consistency of the tests described
above are well known (see \cite{henzepenrose} for the FR test and
\cite{henzenn,schilling} for the $K$-NN test). However, a mathematical
treatment of the power properties of these tests, which requires
understanding the limiting distribution of the test statistics under
the alternative, remained unavailable. In this paper, we address this
problem by deriving the asymptotic distribution of \eqref{Tpoisson},
for stabilizing geometric graph functionals, under general
alternatives, in the Poissonized setting described above. As a
consequence, the exact detection threshold and the limiting local power
of these tests can be derived.

We begin with a few notations: For a vector $x\in\R^{p}$, $ \llVert  x \rrVert  $ and
$ \llVert  x \rrVert  _{1}$ will denote the $L_{2}$ and $L_{1}$ norms of $x$,
respectively. For two nonnegative sequences, $(a_{n})_{n \geq1}$ and
$(b_{n})_{n \geq1}$, $a_{n}=\Theta(b_{n})$ means that there exist
positive constants $C_{1}, C_{2} $, such that $C_{1} b_{n} \leq a_{n}
\leq C_{2} b_{n}$, for all $n$ large enough. Finally, for two positive
sequences $(a_{n})_{n \geq1}$ and $(b_{n})_{n \geq1}$, we write
$a_{n} \ll b_{n}$ or $a_{n} \gg b_{n}$, if $a_{n}/b_{n} \rightarrow0$
or $a_{n}/b_{n} \rightarrow\infty$, respectively. The results
obtained in this paper are summarized below.
\begin{enumerate}
\item[1.] The limiting distribution of graph-based two-sample tests
under general alternatives is derived. The proof of this general result
has two main steps: To begin with we show that for tests based on
stabilizing geometric graphs, such as the Friedman--Rafsky test \eqref
{Tfr} and the test based on the $K$-nearest-neighbor ($K$-NN)
graph \eqref{Tknn}, the statistic \eqref{Tpoisson} has a limiting
normal distribution, after centering by the conditional mean and
scaling by $N^{-\frac{1}{2}}$ (Theorem~\ref{TH:CLT_R1}). This result
is of independent interest, as it leads to a new conditional test, and
can be used for approximate power calculations (Remark~\ref
{rem:approxpower}). Next, under the stronger assumption of \emph
{exponential stabilization} \cite{penroseclt}, the conditional CLT can
be strengthened to obtain the (unconditional) central limit theorem of
\eqref{Tpoisson} (Theorem~\ref{TH:CLT_R}).

\item[2.] The CLT proved above can be used to determine the detection
threshold of the $K$-NN test, that is, the rate at which the
alternatives shrink toward the null, such that the limiting power of
the test transitions from 0 to 1. More precisely, suppose $\{\P_\theta
\}_{\theta\in\Theta}$ is a parametric family of distributions in $\R
^{d}$, indexed by $\theta \in \Theta\subseteq\R^{p}$. Given samples $\sX
_{N_{1}}'$ and $\sY_{N_{2}}'$ from $\P_{\theta_{1}}$ and $\P
_{\theta_{2}}$ as in \eqref{2poisson}, respectively, consider the
testing problem
\[
H_{0}: \theta_{2}-\theta_{1}=0\quad
\text{versus}\quad H_{1}: \theta _{2}-
\theta_{1}=\varepsilon_{N},
\]
for a sequence $(\varepsilon_{N})_{N\geq1}$ in $\R^{p}$, such that
$ \llVert  \varepsilon_{N} \rrVert  \rightarrow0$. The detection threshold for the
$K$-NN test is the magnitude of the sequence $\varepsilon_{N}$ below which the test is
powerless and above which the test has power going to 1. The parametric
rate of detection is $O(N^{-\frac{1}{2}})$; however, results in \cite
{BBB} imply that tests based on geometric graphs, have no power in this
scale, that is, they have zero Pitman efficiency, which makes the
problem of determining the detection threshold of such tests
particularly interesting. In Theorem~\ref{EFFSECOND}, we determine the
precise detection threshold of the $K$-NN test, which undergoes a
remarkable phase transition at dimension $d \geq9$, and compute the
exact limiting power at the threshold. The result is pictorially
represented in Figure~\ref{fig:effsecond} and summarized below:
\begin{itemize}
\item[$\bullet$] For dimension $d \leq8$, the detection threshold of
the test based on the $K$-NN graph \eqref{rejregionKNN} is at $\Theta
(N^{-\frac{1}{4}})$, that is, the limiting power of the test undergoes
a phase transition from the level $\alpha$ to 1, depending on whether
$ \llVert  N^{\frac{1}{4}} \varepsilon_{N} \rrVert   \rightarrow0$ or $ \llVert  N^{\frac
{1}{4}} \varepsilon_{N} \rrVert   \rightarrow\infty$, respectively.
Moreover, using the CLT above, we can derive the exact local power at
the threshold $N^{\frac{1}{4}}\varepsilon_{N} \rightarrow h $.

\item[$\bullet$] The detection threshold changes for dimension $d
\geq9$, where the situation becomes more delicate: Here, the $K$-NN
test has power going to $\alpha$ or 1, depending on whether
$ \llVert  N^{\frac{1}{2}-\frac{2}{d}}\varepsilon_{N} \rrVert   \rightarrow0 $ or
$ \llVert   N^{\frac{2}{d}} \varepsilon_{N} \rrVert   \rightarrow\infty$,
respectively. This shows that the detection threshold is somewhere
between these two bounds, however, unlike in $d \leq8$, the exact
location of the detection threshold has no universality: it depends on
the distribution of the data under the null and the direction along
which $\varepsilon_{N}$ goes to zero. Note that the exponent in the
lower bound $\frac{1}{2}-\frac{2}{d}$ increases to $\frac{1}{2}$
(the parametric detection threshold), and the exponent in the upper
bound decreases to the 0 (which gives consistent fixed alternatives).
We show that both these thresholds are tight in a truncated spherical
normal problem, depending on the sign of the alternative. This is an
example where the $K$-NN test exhibit a surprising \emph{blessing of
dimensionality}, that is, it becomes easier to detect local changes
along certain directions as dimension increases (see Section~\ref{sec:snormal} for details). The reason behind the phase transition of
the detection threshold at dimension 9 is explained in Section~\ref{sec:pfoutline}, and the details of the proof are given in
Appendix \ref{sec:pfeffsecond}.
\end{itemize} 
\end{enumerate}

\begin{figure}[h]
\centering
\begin{minipage}[c]{1.00\textwidth}
\centering
\includegraphics[width=4.5in]
    {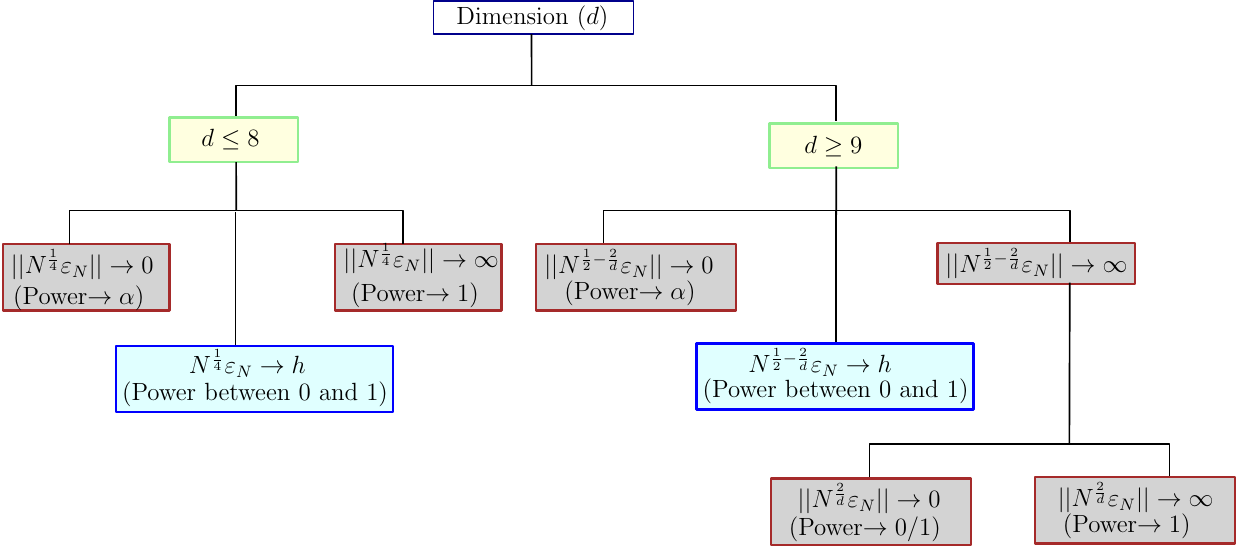}\\
\end{minipage}
\caption{\small{Detection threshold of the test based on the $K$-NN graph: Illustration of Theorem \ref{EFFSECOND}.}}
\label{fig:effsecond}\vspace{-0.2in}
\end{figure}

\subsection{Organization}
\label{sec1.5}
The rest of the paper is organized as follows: The general consistency
result is stated in Section~\ref{consistent}. The central limit
theorems for the statistic \eqref{Tpoisson} are described in
Section~\ref{distribution}. The detection threshold and local power of
the $K$-NN test are given in Section~\ref{sec:secondorder}, and the
performance of the different tests are compared in simulations in
Section~\ref{sec:examples}. The proofs of the results are given in the
Appendix.

\section{Consistency}
\label{consistent}

In this section, we prove consistency against all fixed alternatives of
the test \eqref{rejregionN} for stabilizing graphs functionals. This
unifies the proof of consistency of the test based on the $K$-NN
graph \cite{schilling,henzenn}, and the FR test \cite{henzepenrose},
generalizing the result to any stabilizing graph. We begin by recalling
that $d(x, \sG(S))$ is the total degree of the vertex $x$ in $\sG
(S\cup\{x\})$, for $S\subset\R^{d}$ nice and $x\in\R^{d}$.
Moreover, for a function $\psi: \R^{d} \rightarrow\R$, we denote by
$\cP_\psi$ the inhomogeneous Poisson process with intensity function
$\psi$. (In particular, this means for any measurable set $A \subset
\R^{d}$, the number of points in $A$ is distributed as $\dPois
(\int_{A} \psi(x) \,\mathrm{d}x )$.)

%
\begin{assumption}[Degree moment condition]
A translation and scale invariant graph functional $\sG$ is said to
satisfy the $\beta$-\emph{degree moment condition} if it stabilizes on 
$\cP_{\lambda}$, for all $\lambda\in(0, \infty)$, and
%
\begin{equation}
\sup_{N \in\N}\sup_{\substack{z \in\R^{d},\\ \cA\subset\R
^{d}}}\E \bigl
\{d^\beta\bigl(z, \sG(\cP_{N\phi_{N}}\cup\cA)\bigr) \bigr\} <
\infty, \label{degmoment}
\end{equation}
where $\cA$ ranges over all finite subsets of $\R^{d}$, and $\phi
_{N}=\frac{N_{1}}{N} f + \frac{N_{2}}{N} g$.
\label{pdegree}
\end{assumption}

This condition ensures that the $\beta$-th moment of the degree
function at a point $z$ is uniformly bounded over $z$ and over the
addition of finitely many points to the data. Note that this is
trivially satisfied for bounded degree graphs, such as the $K$-NN and
the MST. Under this assumption, the weak limit of the statistic $\frac
{1}{N} T(\sG(\cZ_{N}'))$ can be derived, which is given in terms of
the Henze--Penrose dissimilarity measure between the two density functions.

%
\begin{defn}\label{hpdiss}Given $p\in(0, 1)$ and densities $f$ and
$g$ in $\R^{d}$, the \emph{Henze--Penrose dissimilarity} measure is
defined as
%
\begin{equation}
\delta(f, g, p)=1- 2pq
\int\frac{f(x)g(x)}{p f(x)+q g(x)}\,\mathrm{d}x. \label{hpdis}
\end{equation}
This belongs to a general class of separation measures between
probability distributions \cite{gyorfinemetz1}.
\end{defn}

The following proposition gives the weak-limit of $\frac{1}{N} T(\sG
(\cZ_{N}'))$ for stabilizing graph functionals satisfying the degree
moment condition. The proof of the proposition closely mimics
\cite{henzepenrose}, Theorem~2, and is detailed in Section~A.3.

%
\begin{ppn} \label{CONSISTENT} Let $\sG$ be a translation and scale
invariant directed graph functional which stabilizes on $\cP_{\lambda
}$ for all $\lambda\in(0, \infty)$. If $\sG$ satisfies the $\beta
$-degree moment condition for some $\beta>4$, then
%
\begin{equation}
\label{eq:weaklimit} \frac{1}{N} T\bigl(\sG\bigl(\cZ_{N}'
\bigr)\bigr)\pto\frac{\E\Delta_{0}^\uparrow}{2} \bigl(1-\delta(f, g, p)\bigr),
\end{equation}
where $\Delta_{0}^\uparrow=d^\uparrow(0, \sG(\cP_{1}))$ is the
out-degree of the origin in the graph $\sG(\cP_{1} \cup\{0\})$.
\end{ppn}

Using the fact $\delta(f, g, p)\geq\delta(f, f, p)=p^{2}+q^{2}$ and
that the inequality is strict for densities $f$ and $g$ differing on a
set of positive measure (see \cite{gyorfinemetz1}, Theorem~1 and Corollary~1), it can be shown that various tests based on
stabilizing graph functionals, which includes the MST and the $K$-NN
graphs, are consistent for all fixed alternatives (\ref{test}) (refer
to Remark~A.2 for details).

%
\begin{remark}Recently, Arias-Castro and Pelletier \cite
{cm_consistency} showed that Rosenbaum's cross match test \cite
{rosenbaum} based on non-bipartite matching (NBM), has the same limit
as in \eqref{eq:weaklimit}, thus it is also consistent for general
alternatives. Note that this does not follow from Proposition~\ref
{CONSISTENT}, because it is unknown whether the NBM graph functional is
stabilizing. They show that the properties of stabilizing graphs
required in the proof of consistency also hold for the NBM graph
functional and, therefore, \eqref{eq:weaklimit} holds for the cross
match test as well.
\end{remark}

\section{Distribution under general alternatives}
\label{distribution}

This section describes the central limit theorems of the Poissonized
two sample statistic $T(\sG(\cZ_{N}'))$ (recall \eqref{Tpoisson})
for stabilizing graph functionals. Let $\sX_{N_{1}}'$ and $\sY
_{N_{2}}'$ be Poissonized samples from densities $f$ and $g$ in $\R
^{d}$ as in \eqref{2poisson}. Define
%
\begin{align}
\phi_{N}(x)=\frac{N_{1}}{N}f(x)+\frac{N_{2}}{N}g(x)
\quad\text{and}\quad \phi (x)=pf(x)+qg(x). \label{phi}
\end{align}
Recall from Section~\ref{sec:poissonization} that the joint
distribution of $\sX_{N_{1}}'$ and $\sY_{N_{2}}'$ can be described as
follows: Let $\cZ'=\{Z_{1}, Z_{2}, \ldots\}$ be independent random
variables with common density $\phi_{N}$. Then $\cZ_{N}'=\{Z_{1},
Z_{2}, \ldots, Z_{L_{N}}\}$, where $L_{N}\sim\dPois(N)$ is
independent of $\cZ'$, and each point of $\cZ_{N}'$ is labeled 1 or 2
as in \eqref{pooledlabelalternative}. Then the sets of points assigned
labels 1 and 2 have the same distribution as $\sX_{N_{1}}'$ and $\sY
_{N_{2}}'$. In this section, we derive the limiting distribution of the
test statistic
%
\begin{align}
\cR\bigl(\sG\bigl(\cZ_{N}'\bigr)\bigr)&=
\frac{1}{\sqrt N} \bigl\{ T\bigl(\sG\bigl(\cZ_{N}'
\bigr)\bigr)- \E_{H_{1}}\bigl(T\bigl(\sG\bigl(\cZ_{N}'
\bigr)\bigr)\bigr) \bigr\} \label{R}
\end{align}
for stabilizing graph functionals. This involves the following two steps:
\begin{enumerate}
\item[(1)] The first step is to derive the CLT of the \emph{test
statistic centered by the conditional mean} $\E_{H_{1}}(T(\sG(\cZ
_{N}'))|\cF)$, where $\cF:=\sigma(\cZ', L_{N})$ is the
sigma-algebra generated by $\cZ'$ and the Poisson random variable
$L_{N}$, that is,
%
\begin{equation}
\cR_{1}\bigl(\sG\bigl(\cZ_{N}'\bigr)
\bigr)=\frac{1}{\sqrt N} \bigl\{ T\bigl(\sG\bigl(\cZ _{N}'
\bigr)\bigr)- \E_{H_{1}}\bigl(T\bigl(\sG\bigl(\cZ_{N}'
\bigr)\bigr)|\cF\bigr) \bigr\}, \label{R1}
\end{equation}
for a stabilizing graph functional $\sG$. Note that conditional on
$\cF$, the randomness comes from the labeling \eqref
{pooledlabelalternative}. As the labeling is independent across the
vertices of the graph, the dependence in \eqref{R1} is local, and the
CLT can be proved using Stein's method based on dependency graphs
(Theorem~\ref{TH:CLT_R1}). This can be used to devise and calibrate a
conditional test (see Remark~\ref{rem:approxpower}), which might be of
independent interest.

\item[(2)] The second step is to derive the CLT of the \emph
{conditional mean}
%
\begin{equation}
\cR_{2}\bigl(\sG\bigl(\cZ_{N}'\bigr)
\bigr)= \frac{1}{\sqrt N} \bigl\{ \E_{H_{1}}\bigl(T\bigl(\sG \bigl(
\cZ_{N}'\bigr)\bigr)|\cF\bigr) - \E_{H_{1}}
\bigl(T\bigl(\sG\bigl(\cZ_{N}'\bigr)\bigr)\bigr) \bigr
\}. \label{R2}
\end{equation}
This requires the additional assumption of exponential stabilization
(Definition~\ref{expstabilize}), and is proved in Proposition~\ref{CLT_R2}.
\end{enumerate}

The above results can be combined to obtain the CLT of \eqref{R},
since $\cR(\sG(\cZ_{N}'))=\cR_{1}(\sG(\cZ_{N}'))+\cR_{2}(\sG
(\cZ_{N}'))$ (see Theorem~\ref{TH:CLT_R} below for details).

\subsection{The conditional CLT}
\label{sec3.1}

For a directed graph functional $\sG$, $S\subset\R^{d}$ finite and a
point $x\in\R^{d}$, let $d^\uparrow(x, \sG(S))$ be the out-degree
of the vertex $x$ in the graph $\sG(S\cup\{x\})$, that is, the number
of outgoing edges $(x, y)$, where $ y \in S\cup\{x\}$, in the graph
$\sG(S\cup\{x\})$. Similarly, let $d^\downarrow(x, \sG(S))$ be the
in-degree of the vertex $x$ in the graph $\sG(S\cup\{x\})$, that is,
the number of incoming edges $(y, x)$, where $y \in S\cup\{x\}$, in
the graph $\sG(S\cup\{x\})$. Note that $d(x, \sG(S))=d^\downarrow
(x, \sG(S))+d^\uparrow(x, \sG(S))$ is the total degree of the vertex
$x$ in the graph $\sG(S\cup\{x\})$.

Moreover, let
%
\begin{equation}
T_{2}^{\uparrow}\bigl(x, \sG(S)\bigr)= {d^\uparrow\bigl(x, \sG(S)\bigr)\choose2},\qquad
T_{2}^{\downarrow}\bigl(x, \sG(S)\bigr)= {d^\downarrow\bigl(x, \sG(S)\bigr)\choose2}
\label{2star}
\end{equation}
be the number of outward 2-stars and inward 2-stars incident on $x$ in
$\sG(S)$, respectively. Finally, let $T_{2}^{+}(x, \sG(S))$ be the
number of 2-stars incident on $x$ in $\sG(S)$ with different
directions on the two edges. For notational brevity, denote
%
\begin{align}
\label{Delta0} \Delta^\uparrow_{0}=d^\uparrow
\bigl(0, \sG(\cP_{1})\bigr), \qquad \Delta ^\downarrow_{0}=d^\downarrow
\bigl(0, \sG(\cP_{1})\bigr),
\end{align}
and $\Delta^{+}_{0}:=|\{z\in\cP_{1}: (0, z), (z, 0)\in E(\sG(\cP
_{1}^{0}))\}|$. (Note that $\Delta^\uparrow_{0}$ was already defined
in the statement of Proposition~\ref{CONSISTENT}.) Similarly, let
%
\begin{align}
\label{2star0} T_{2}^\uparrow=T_{2}^\uparrow
\bigl(0, \sG(\cP_{1})\bigr),\qquad  T_{2}^\downarrow
=T_{2}^\downarrow\bigl(0, \sG(\cP_{1})\bigr),
\end{align}
and $T_{2}^{+}:=T_{2}^{+}(0, \sG(\cP_{1}))$.

To derive the CLT of \eqref{R1}, we need some control on the maximum
degree of the graph functional $\sG$. The natural assumption of
bounded maximum degree includes most of the natural graphs, such as the
MST and the $K$-NN graph. The slightly weaker polynomial upper bound
given below includes other stabilizing geometric graphs, like the
Delaunay graph \cite{py}.

%
\begin{assumption}[Maximum degree condition]\label{assummaxdeg}
A graph functional $\sG$ is said to satisfy the \emph{maximum degree
condition} if
%
\begin{equation}
\sup_{z\in\cP_{N\phi_{N}}} d\bigl(z, \sG(\cP_{N\phi
_{N}})
\bigr)=o_{P}\bigl(N^{\frac{1}{40}}\bigr). \label{maxdegree}
\end{equation}
\end{assumption}

The following theorem gives the CLT of the test statistic centered by
the conditional mean, as in \eqref{R1}, for stabilizing graph functionals. Recall $\phi(x)=p f(x) + q g(x)$. 

%
\begin{thm}\label{TH:CLT_R1} Let $\sG$ be a translation and scale
invariant directed graph functional which stabilizes on $\cP_{\lambda
}$, for all $\lambda\in(0, \infty)$. If $\sG$ satisfies the $\beta
$-degree moment condition for $\beta>4$ and the maximum degree
condition \eqref{maxdegree}, then
\begin{align*}
\mathcal R_{1}\bigl(\sG\bigl(\cZ_{N}'
\bigr)\bigr) \dto N\bigl(0, \kappa_{\sG}^{2}\bigr),
\end{align*}
where
%
\begin{align}
\label{sigma1} \kappa_{\sG}^{2}=\frac{r}{4}
\int\frac{f(x)g(x)}{\phi^{3}(x)} L(x) \,\mathrm{d}x,
\end{align}
with $r:=2pq$, and
\[
L(x):=2\E\Delta_{0}^\uparrow\phi^{2}(x) + 4
\phi(x) \bigl(q\E T_{2}^\uparrow g(x)+ p\E
T_{2}^\downarrow f(x) \bigr) - 4pq \E\Gamma _{0}
f(x)g(x),
\]
where $\Gamma_{0}:= T_{2}^\uparrow+ T_{2}^\downarrow+ T_{2}^{+} +
\frac{\Delta^{+}_{0}}{2}+ \frac{\Delta_{0}^\uparrow}{2}$, with $\Delta_{0}^\uparrow$, $\Delta_{0}^\downarrow$, $T_{2}^\uparrow$,
and $T_{2}^\downarrow$ as defined in \eqref{Delta0} and \eqref{2star0}.
\end{thm}

The proof of theorem is given in Section~A.4. The limit
of the conditional variance \eqref{sigma1} is derived using properties
of stabilizing graphs, and the CLT is proved using Stein's method based
on dependency graphs. In fact, our proof suggests that it is possible
to extend the CLT in Theorem~\ref{TH:CLT_R1} to other distance
functions in $\R^{d}$, whenever the maximum degree condition
(Assumption~\ref{assummaxdeg}) holds, and the conditional variance of
$\mathcal R_{1}(\sG(\cZ_{N}'))$ has a limit in probability, in the
graph $\sG(\cZ_{N}')$ constructed using that metric. This is because
our proof technique proceeds by conditioning on the randomness of the
graph and, therefore, as long as the associated graph quantities that
arise in $\Var(\cR_{1}(\sG(\cZ_{N}'))|\cF)$ converge in
probability (as in Lemma~A.5), and the dependence is local
(which is ensured by Assumption~\ref{assummaxdeg}), the Stein's method
argument applies and the asymptotic normality in Theorem~\ref
{TH:CLT_R1} would hold.

%
\begin{remark}[Null distribution] Given the graph functional $\sG$,
the limit of the conditional variance $\kappa_{\sG}$ depends on the
densities $f$ and $g$ and the limiting proportion $p$ of the samples.
Under the null ($f=g$) this simplifies to
%
\begin{align}
\kappa_{\sG, H_{0}}^{2}=\frac{r}{4} \bigl\{2\E
\Delta_{0}^\uparrow+ 4 \bigl(q\E T_{2}^\uparrow+
p\E T_{2}^\downarrow \bigr) -2 r\Gamma _{0} \bigr
\}. \label{sigmanulldirected}
\end{align}
\begin{itemize}
\item$\sG=\cN_{K}$ is the $K$-NN nearest neighbor graph functional:
In this case, $\Delta_{0}^\uparrow=K$, $T_{2}^\uparrow=\frac
{K(K-1)}{2}$, $T_{2}^{+}=\Delta_{0}^\uparrow\Delta_{0}^\downarrow
-\Delta_{0}^{+}=K\Delta_{0}^\downarrow-\Delta_{0}^{+}$, $\E\Delta
^\downarrow=K$, and \eqref{sigmanulldirected} simplifies to
%
\begin{align}
\label{knnnulldirected} \kappa_{\cN_{K}, H_{0}}^{2} &=\frac{r}{2}
\bigl\{Kpq + (p-q)^{2}K^{2}+ p^{2}\Var\bigl(
\Delta_{0}^\downarrow\bigr) + pq \E\Delta _{0}^{+}
\bigr\}.
\end{align}

\item$\sG$ is an undirected graph functional: In this case \eqref
{sigmanulldirected} simplifies to
%
\begin{align}
\kappa_{\sG, H_{0}}^{2}=\frac{r}{2} \bigl\{r \E
\Delta_{0} + \E \bigl(\Delta_{0}^{2}\bigr)
(1-2r) \bigr\}, \label{sigmanullundirected}
\end{align}
since $\E T_{2}^\uparrow=\E T_{2}^\downarrow=\E\frac{\Delta
_{0}(\Delta_{0}-1)}{2}$, $\E T_{2}^{+}=\E\Delta_{0}(\Delta_{0}-1)$
and $\E\Delta_{0}^{+}=\E\Delta_{0}$. For example, when $\sG=\cT$
is the MST graph functional as in the Friedman--Rafsky test \eqref
{Tfr}, $\Delta_{0}=2$ (\cite{aldous_steele}, Lemma~7), and \eqref
{sigmanullundirected} becomes $\kappa_{\cT, H_{0}}^{2}= r  \{r +
\frac{1}{2}\E(\Delta_{0}^{2})(1-2r) \}$.
\end{itemize}
\end{remark}

The above discussion suggests that the CLT in Theorem~\ref{TH:CLT_R1}
can be used to derive a conditional test for \eqref{test}.

%
\begin{remark}[A conditional test and its
power]\label{rem:approxpower} For concreteness, suppose $\sG=\cT$ is the MST. Then under the
null $\E_{H_{0}}(T(\cT(\cZ_{N}'))|\cF)=\frac
{N_{1}N_{2}}{(N_{1}+N_{2})^{2}}$ $|E(\cT(\cZ_{N}'))|=\frac
{N_{1}N_{2}}{(N_{1}+N_{2})^{2}}(L_{N}-1)$, and given the data, we
reject $H_{0}$ whenever
\[
\frac{1}{\sqrt N} \biggl\{ T\bigl(\cT\bigl(\cZ_{N}'
\bigr)\bigr)- \frac
{N_{1}N_{2}}{(N_{1}+N_{2})^{2}} L_{N} \biggr\} \leq
\kappa_{\sG,
H_{0}} z_{\alpha}.
\]
By Theorem~\ref{TH:CLT_R1}, this test has asymptotically level $\alpha
$. Moreover, it can be shown that (see Section~A.3
for details) that
\begin{align*}
\E_{H_{1}}\bigl(T\bigl(\sG\bigl(\cZ_{N}'
\bigr)\bigr)|\cF\bigr)=\sum_{1\leq i\ne
j\leq L_{N}}
\frac{N_{1}N_{2} f(Z_{i})g(Z_{j}) \bm1\{(Z_{i},
Z_{j})\in E(\sG(\cZ_{N}'))\}
}{(N_{1}f(Z_{i})+N_{2}g(Z_{j}))(N_{1}f(Z_{i})+N_{2}g(Z_{j}))}.
\end{align*}
The proof of Proposition~\ref{CONSISTENT} reveals that this test is
consistent against all fixed alternatives, and using Theorem~\ref
{TH:CLT_R1} we can compute the approximate power of this test as
%
\begin{align}
\label{eq:approxpower} \Phi \biggl( \frac{\kappa_{\sG, H_{0}} z_{\alpha} -\Xi(\cZ
_{N}')}{\kappa_{\sG}} \biggr),
\end{align}
where $\Xi(\cZ_{N}')= \frac{1}{\sqrt N}( \E_{H_{1}}(T(\sG(\cZ
_{N}'))|\cF)-\E_{H_{0}}(T(\sG(\cZ_{N}'))|\cF))$ is the difference
of the conditional means under the alternative and the null, which can
be calculated from the data. The approximation in \eqref
{eq:approxpower} can be justified because Stein's method gives uniform
control on the corresponding distribution functions (see Proposition~A.2 in the  Appendix). (Note that the argument
above holds for any stabilizing graph functional, as long as the number
of edges $|E(\sG(\cZ_{N}'))|$ does not depend on the unknown null
distribution, as is the case for the Friedman--Rafsky test and the test
based on the $K$-NN graph.)
\end{remark}

\subsection{CLT of the test statistic under general alternatives}
\label{sec3.2}

In this section, the (unconditional) CLT of the test statistic \eqref
{R} is derived. This involves finding the CLT of the conditional mean
\eqref{R2}, which requires the stronger notion of exponential
stabilization \cite{penroseclt}. For any locally finite point set $\cH
\subset\R^{d}$ and $x\in\R^{d}$, define the \emph{out-degree
measure} of a graph functional $\sG$ as follows: For all Borel sets
$A\subset\R^{d}$,
%
\begin{equation}
\label{outdegmeasure} d_\sG^\uparrow(x, \cH, A)=\sum
_{y\in\cH^{x} \cap A}\bm1\bigl\{(x, y)\in E\bigl(\sG\bigl(
\cH^{x}\bigr)\bigr)\bigr\},
\end{equation}
where $\cH^{x}=\cH\cup\{x\}$. In other words, the out-degree measure
of a set $A$, with respect to $\cH$ and $x$ is the number of edges
incident on $x$ with the other end point in $\cH^{x} \cap A$ in the
graph $\sG(\cH^{x})$. The following definition formalizes the notion
of ``radius of stabilization'' of a point, which is the smallest radius
outside which addition of finitely many points does not affect the
degree measure at the point.

%
\begin{defn}Fix a locally finite point set $\cH$, a point $x\in\R^{d}$, and a Borel set $A\subseteq\R^{d}$. The \emph{radius of
stabilization} of the degree measure (\ref{outdegmeasure}) at $x$ with
respect to $\cH$ and $A$ (to be denoted by $R(x, \cH, A)$) is the
smallest $R\geq0$ such that
%
\begin{equation}
d_\sG^\uparrow\bigl(x, x+\bigl\{\cH\cap B(0, R)\cup\cY
\bigr\}, x+B\bigr)=d_\sG ^\uparrow\bigl(x, x+\bigl\{\cH\cap
B(0, R)\bigr\}, x+B\bigr),
\end{equation}
for all finite $\cY\subseteq A\setminus B(0, R)$ and all Borel subsets
$B\subseteq A$, where $B(0, R)$ is the (Euclidean) ball of radius $R$
with center at the point $0\in\R^{d}$. If no such $R$ exists, then
set $R(x, \cH, A)=\infty$.
\end{defn}

Throughout this section, we will assume that $f$ and $g$ have a common
support $S$, which is compact and convex, and $N \rightarrow\infty$
such that
%
\begin{align}
\label{eq:N12} \sqrt N \biggl(\frac{N_{1}}{N_{1}+N_{2}}-p \biggr)\rightarrow0\quad
\text{and}\quad \sqrt N \biggl(\frac{N_{1}}{N_{1}+N_{2}}-q \biggr)\rightarrow0.
\end{align}

%
\begin{defn}\label{expstabilize}Let $R_{N}(x):=R(x, \cP_{N\phi_{N}},
S)$ be the radius of stabilization of out-degree measure $d_\sG
^\uparrow$ at $x$ with respect to the Poisson process $\cP_{N\phi
_{N}}$ and $S$. Define
%
\begin{equation}
\label{tau} \tau(s):=\sup_{N \in\mathbb N}\sup
_{x\in\mathbb R^{d}}\P\bigl(R_{N\phi
_{N}}(x)>N^{-\frac{1}{d}}s\bigr).
\end{equation}
The out-degree measure $d_\sG^\uparrow$ is said to be 
\begin{itemize}
\item{\emph{power law stabilizing of order}} $q$ if $\sup_{s\geq1}
s^{q}\tau(s)<\infty$,
\item{\emph{exponentially stabilizing}} if $\lim\sup_{s\rightarrow
\infty} \frac{1}{s} \log\tau(s)<0$.
\end{itemize}
\end{defn}

Conditions on the decay of the tail of the radius of stabilization,
similar to \eqref{tau} above, is a standard requirement for proving
limit theorems of functionals of random geometric graphs \cite
{penroseclt,yukichclt}. Using this machinery, we prove the following
theorem, which gives the CLT of the conditional mean \eqref{R2} for
exponentially stabilizing random geometric graphs.

%
\begin{ppn}\label{CLT_R2} Let $\sG$ be a translation and scale
invariant directed graph functional in $\R^{d}$ which satisfies the
$\beta$-degree moment condition \eqref{degmoment} for some $\beta
>2$. If the out-degree measure $d_\sG^\uparrow$ is power law
stabilizing of order $q> \frac{\beta}{\beta-2}$, then
%
\begin{align}
\label{varcond} \lim_{N\rightarrow\infty}\Var\bigl( \cR_{2}
\bigl(\sG\bigl(\cZ_{N}'\bigr)\bigr)\bigr) = \tau
_{\sG}^{2},
\end{align}
where
%
\begin{align}
\label{sigma2} \tau_{\sG}^{2}=\frac{r^{2}}{4}
L_{0}
\int\frac
{f^{2}(x)g^{2}(x)}{\phi^{3}(x)} \,\mathrm{d}x,
\end{align}
where $L_{0}:=\int  (\E\{d^\uparrow(0, \sG(\cP_{1}^{z}))
d^\uparrow(z, \sG(\cP_{1}^{0})) \}-(\E\Delta_{0}^\uparrow)^{2}
 ) \,\mathrm{ d} z +\E(\Delta_{0}^\uparrow)^{2}$. Moreover, if
$d_\sG^\uparrow$ is exponentially stabilizing then $\cR_{2}(\sG(\cZ
_{N}')) \dto N(0, \tau_{\sG}^{2})$.
\end{ppn}

The proof of theorem is given in Section~A.6.1. Combining
Theorem~\ref{TH:CLT_R1} and Proposition~\ref{CLT_R2}, the CLT of $\cR
(\sG(\cZ_{N}'))$ (defined in \eqref{R}) can be obtained. The proof
is in Section~A.6.2.

%
\begin{thm}\label{TH:CLT_R} Let $\sG$ be a translation and scale
invariant directed graph functional which satisfies the $\beta$-degree
moment condition for some $\beta>4$ and the maximum out-degree
condition \eqref{maxdegree}. If the degree measure $d_\sG^\uparrow$
is exponentially stabilizing, then
%
\begin{align}
\label{cltt} \cR\bigl(\sG\bigl(\cZ_{N}'\bigr)
\bigr) \dto N\bigl(0, \sigma^{2}_\sG\bigr),
\end{align}
where $\sigma^{2}_\sG= \kappa_{\sG}^{2}+\tau_{\sG}^{2}$, with
$\kappa_{\sG}^{2}$ and $\tau_{\sG}^{2}$ as defined in \eqref
{sigma1} and \eqref{sigma2}, respectively.
\end{thm}

Many random geometric graphs, such as the $K$-NN graph and the Delaunay
graph \cite{penroseclt,py} are exponentially stabilizing. This theorem gives the asymptotic distribution of
two-sample tests based on such graphs, under general alternatives. This
can be used to the compute power of such tests as in Remark~\ref
{rem:approxpower}. Moreover, using this we can understand the
asymptotic performances of the tests, by identifying testable local
alternatives, as elaborated in the following section for the test based
on the $K$-NN graph. The techniques used in this section might also be
useful in studying limiting distributions of multivariate
goodness-of-fit tests based on nearest neighbors \cite
{bickel_breiman,henze_yukich}.

To see why the asymptotic variance in \eqref{cltt} is the sum of two
terms, note that
\begin{align*}
\Var\bigl(\cR\bigl(\sG\bigl(\cZ_{N}'\bigr)\bigr)
\bigr) &= \E\bigl(\Var\bigl(\cR\bigl(\sG\bigl(\cZ_{N}'
\bigr)\bigr)|\cF\bigr)\bigr)+ \Var\bigl(\E\bigl(\cR\bigl(\sG\bigl(
\cZ_{N}'\bigr)\bigr)|\cF\bigr)\bigr),
\end{align*}
where $\cF:=\sigma(\cZ', L_{N})$ is the sigma-algebra generated by
$\cZ'$ and the Poisson random variable $L_{N}$ (recall notation from
Section~\ref{sec:poissonization}). Now, recalling \eqref{R1} shows
$\Var(\cR(\sG(\cZ_{N}'))|\cF)=\Var(\cR_{1}(\sG(\cZ_{N}'))|\cF
)$, and \eqref{R2} gives $\Var(\E(\cR(\sG(\cZ_{N}'))|\cF))=\Var
(\cR_{2}(\sG(\cZ_{N}')))$. In the proof of Theorem~\ref{TH:CLT_R1},
we show that $\Var(\cR_{1}(\sG(\cZ_{N}'))|\cF)$ converges in
$L_{2}$ to $\kappa_{\sG}^{2}$ (Lemma~A.5), while the proof
of Proposition~\ref{CLT_R2} shows that $\Var(\cR_{2}(\sG(\cZ
_{N}'))) \rightarrow\tau_{\sG}^{2}$ (Section~A.6.1), hence
the asymptotic variance in \eqref{cltt} is the sum of these two terms.

%
\begin{remark}\label{rem:poisson}(Comments on de-Poissonization)
Poissionization is a commonly used trick in geometric probability,
where calculations become simpler because of the spatial independence
of the Poisson process. In fact, when the sample sizes are large, one
can pretend that the data comes from Poissonized samples with a
slightly smaller mean, since a Poisson random variable is tightly
concentrated around its expectation. De-Poissionization techniques are
well known (\cite{penrosebook}, Section~2.5 and \cite{penroseclt}, Theorem~2.3), using which one can expect to de-Poissonize the CLT
in Theorem~\ref{TH:CLT_R} for the test based on the $K$-NN graph. The
only thing that would change is the formula of the asymptotic variance,
but its derivation is quite tedious for general alternatives. However,
for the implementation of the test, we are more interested in the
asymptotic null variance, where the calculations are much simpler, and
the de-Poissonized null variance can be easily computed
(see Section~A.5). In fact, de-Poissonization would only change
the asymptotic variance (not the order), and the constants in the
limiting power (but, not the rates). Therefore, de-Poissonization would
not affect (most of) the results of Section~\ref{sec:knn} as these
mainly focus on detection thresholds. This is also validated by the
simulations in Section~\ref{sec:examples}.
\end{remark}

\section{Local power of the $K$-NN test}
\label{sec:knn}

The test based on the $K$-NN graph is exponentially stabilizing and,
therefore, the results obtained in the previous section apply. Recall that we assume $f$, $g$ have a common support $S$ which is compact and convex,
and $N \rightarrow\infty$ such that \eqref{eq:N12} hold. Then we
have the following corollary of Theorem~\ref{TH:CLT_R}.

%
\begin{cor}\label{cor:KNN}
For the $K$-NN graph functional $\cN_{K}$ and $f$ and $g$ as above:
%
\begin{align}
\label{cltKNN} \cR\bigl(\cN_{K}\bigl(\cZ_{N}'
\bigr)\bigr) \dto N\bigl(0, \sigma^{2}_{\cN_{K}}\bigr),
\end{align}
where $\sigma^{2}_{\cN_{K}}= \kappa_{\cN_{K}}^{2}+\frac
{r^{2}K^{2}}{4} \int\frac{f^{2}(x)g^{2}(x)}{\phi^{3}(x)} \,\mathrm{d}x$, with $
\kappa_{\cN_{K}}$ as defined in \eqref{sigma1}.
\end{cor}

\begin{proof} Note that the graph functional $\cN_{K}$ is exponentially
stabilizing \cite{penroseclt} and satisfies the degree moment
condition for $\beta>4$. Therefore, by Theorem~\ref{TH:CLT_R}, \eqref
{cltKNN} holds with $\sigma^{2}_{\cN_{K}}=\kappa_{\cN_{K}}^{2}+\tau
_{\cN_{K}}^{2}$. The result follows by noting that $\tau_{\cN
_{K}}^{2}=\frac{r^{2} K^{2}}{4} \int\frac{f^{2}(x)g^{2}(x)}{\phi
^{3}(x)} \,\mathrm{d}x$ (recall \eqref{sigma2}).
\end{proof}

%
\begin{remark}\label{eq:KNNnullvariance} Under the null $(f=g)$, $\tau
_{\cN_{K}, H_{0}}^{2}=\frac{r^{2} K^{2}}{4}$, and using \eqref
{knnnulldirected}, the asymptotic variance \eqref{cltKNN} simplifies to
%
\begin{align}
\label{sigmaK}
\begin{split} \sigma_{K}^{2}&:=\sigma^{2}_{\cN_{K}, H_{0}}
=\kappa_{\cN_{K},
H_{0}}^{2}+\tau_{\cN_{K}, H_{0}}^{2}
\\
& =\frac{r}{2} \bigl\{K(K+1)pq + (p-q)^{2}K^{2}+
p^{2}\Var\bigl(\Delta _{0}^\downarrow\bigr) \bigr
\}.
\end{split}
\end{align}
Then recalling \eqref{rejregionN}, the two-sample test based on $\cN
_{K}$ rejects when
%
\begin{align}
\label{rejregionKNN} \frac{1}{\sqrt N} \bigl\{ T\bigl(\cN_{K}\bigl(
\cZ_{N}'\bigr)\bigr)- \E_{H_{0}}\bigl(T\bigl(
\cN _{K}\bigl(\cZ_{N}'\bigr)\bigr)\bigr)
\bigr\} \leq\sigma_{K} z_{\alpha}.
\end{align}
\end{remark}

\subsection{Power against local alternatives}
\label{sec:secondorder}

In this section, we determine the power of the $K$-NN test against
local alternatives, that is, the power when the alternatives shrink
(with increasing sample size) toward the null at a certain rate. To
this end, let $\Theta\subseteq\R^{p}$ be the parameter space and $\{
\P_\theta\}_{\theta\in\Theta}$ be a parametric family of
distributions in $\R^{d}$ with density $f(\cdot|\theta)$. Let $\sX
_{N_{1}}'$ and $\sY_{N_{2}}'$ be samples from $\P_{\theta_{1}}$ and
$\P_{\theta_{2}}$ as in \eqref{2poisson}, respectively, and consider
the testing problem
%
\begin{equation}
\label{epsilonN} H_{0}: \theta_{2}-
\theta_{1}=0\quad \text{versus}\quad H_{1}: \theta
_{2}-\theta_{1}=\varepsilon_{N},
\end{equation}
for a sequence $(\varepsilon_{N})_{N\geq1}$ in $\R^{p}$, such that
$ \llVert  \varepsilon_{N} \rrVert  \rightarrow0$. The \emph{limiting power} of the
two-sample test based on the $K$-NN graph $\cN_{K}$ \eqref{rejregionKNN} is
\begin{align*}
\lim_{N \rightarrow\infty}\P_{\theta_{2}=\theta
_{1}+\varepsilon_{N}} \bigl(N^{-\frac{1}{2}}
\bigl\{ T\bigl(\cN_{K}\bigl(\cZ _{N}'
\bigr)\bigr)- \E_{H_{0}}\bigl(T\bigl(\cN_{K}\bigl(
\cZ_{N}'\bigr)\bigr)\bigr) \bigr\} \leq
\sigma_{K} z_\alpha \bigr),
\end{align*}
where $\sigma_{K}$ is the variance of the $K$-NN test under the null
(recall \eqref{sigmaK}). Our goal is to find the threshold on
$\varepsilon_{N}$ where the $K$-NN test transitions from powerless to
powerful. More precisely, we want to determine the sequence $a_{N}
\rightarrow0$, such that for $ \llVert  \varepsilon_{N} \rrVert   \ll a_{N}$, the
limiting power is $\alpha$, and for $ \llVert  \varepsilon_{N} \rrVert   \gg a_{N}$,
the limiting power is 1. The sequence $(a_{N})_{N \geq1}$ is often
known as the \emph{detection-threshold} of the test.

The parametric rate of detection is $O(N^{-\frac{1}{2}})$; however,
results in \cite{BBB} imply that the test based on the $K$-NN graph
has no power in this scale. As a result, the asymptotic performance of
these tests cannot be compared using their Pitman efficiencies
(limiting local power when $\varepsilon_{N}=hN^{-\frac{1}{2}}$, which
happens to be zero in this case), making the problem of determining the
exact detection threshold particularly important. We answer this
question in Theorem~\ref{EFFSECOND}, where the exact detection
threshold of the $K$-NN test is determined. Quite interestingly, the
threshold depends on several things, such as the dimension $d$, the
distribution of the data, and the direction of the alternative.

To state the assumptions required for computing the detection
threshold, we need a few definitions: For a function $g(z_{1}, z_{2}):
\R^{d} \times\R^{p} \rightarrow\R$, $\grad_{z_{1}} g(z_{1},
z_{2})$ denotes the $d\times1$ gradient vector
and $\mathrm H_{z_{1}} g(z_{1}, z_{2})$ the $d \times d$ Hessian matrix
of $g$, with respect to $z_{1}$ (with $z_{2}$ held fixed). Similarly,
$\grad_{z_{2}} g(z_{1}, z_{2})$ and $\mathrm H_{z_{2}} g(z_{1},
z_{2})$ is the $p\times1$ gradient vector and the $p \times p$ Hessian
matrix of $g$, with respect to $z_{2}$, respectively.

%
\begin{assumption}\label{assumptionlocalpower} Suppose the parameter
space $\Theta\subseteq\R^{p}$ is convex, and the family of
distributions $\{\P_\theta\}_{\theta\in\Theta}$ satisfy:
\begin{enumerate}[(a)]
\item[(a)] For all $\theta\in\Theta$, the density $f(\cdot|\theta
)$ has a compact and convex support $S\subset\R^{d}$, with a nonempty
interior, not depending on $\theta$.

\item[(b)] $\int_{\partial S}f(z|\theta)\,\mathrm{ d}z=0$, for all
$\theta\in\Theta$, where $\partial S$ denotes the boundary of $S$.

\item[(c)] For all $\theta\in\Theta$, the functions $f(\cdot
|\theta)$ and $\grad_\theta f(\cdot|\theta)$ are three times
continuously differentiable in the interior of $S$, and the expected
squared of the score function:\break $\E_{X\sim f(\cdot|\theta)} [
\frac{h^\top\grad_{\theta} f(X|\theta)}{f(X|\theta)}  ]^{2}
>0$, for all $h \in\R^{p} \backslash\{0\}$.

\item[(d)] For all $x \in S$, $f(x|\cdot)$ is three times
continuously differentiable in the interior of $\Theta$.
\end{enumerate}
\end{assumption}

The compactness of the support is required for establishing the CLT for
exponentially stabilizing graph functionals (recall Corollary~\ref
{cor:KNN}). However, we expect the CLT, and hence our results, to hold
even when the support is not compact, as long as, the distributions
have ``nice'' tails (see simulations in Section~\ref{sec:examples}
below). Under the above assumptions, the following theorem
characterizes the detection threshold of the $K$-NN test and determines
the exact limiting power at the threshold. To state the theorem, we
need to introduce some notations: Recall that $\cP_{1}^{0}$ denotes
the Poisson process of rate 1 in $\R^{d}$ with the origin $0 \in\R
^{d}$ added to it. Define
%
\begin{align}
\label{eq:C} C_{K, s}&:=\E \biggl\{\sum
_{x \in\cP_{1}^{0}} \llVert x \rrVert ^{s} \bm1\bigl\{(0, x)
\in E\bigl(\cN_{K}\bigl(\cP_{1}^{0}\bigr)
\bigr)\bigr\} \biggr\},
\end{align}
which is the expectation of the sum of the $s$-th power of the lengths
of the outward edges incident at the origin $0$ in the graph $\cN
_{K}(\cP_{1}^{0})$. This can be computed explicitly in terms of Gamma
functions (see (B.11) in the supplementary material for details).
Finally, define
%
\begin{align}
\label{eq:aKtheta1} a_{K, \theta_{1}}(h):=- \frac{rp C_{K, 2}}{4d\sigma_{K}}
\int_{S} h^\top\grad_{\theta_{1}} \biggl(
\frac{ \tr(\mathrm H_{x}
f(x|\theta_{1}))}{f(x|\theta_{1})} \biggr) f^{\frac
{d-2}{d}}(x|\theta_{1}) \,
\mathrm{ d }x,
\end{align}
where $C_{K, 2}$ is defined above in \eqref{eq:C}, $\sigma_{K}$ as in
\eqref{sigmaK}, and
%
\begin{align}
\label{eq:bKtheta1} b_{K, \theta_{1}}(h):=\frac{r^{2} K}{2 \sigma_{K}} \E \biggl[
\frac
{h^\top\grad_{\theta_{1}} f(X|\theta_{1})}{f(X|\theta_{1})} \biggr]^{2},
\end{align}
where the expectation is with respect to $X \sim f(\cdot|\theta_{1})$.

%
\begin{thm}\label{EFFSECOND} Let $\{\P_\theta\}_{\theta\in\Theta
}$ be a family of distributions satisfying Assumption~\ref
{assumptionlocalpower}, and $\sX_{N_{1}}'$ and $\sY_{N_{2}}'$ be
samples from $\P_{\theta_{1}}$ and $\P_{\theta_{2}}$ as in \eqref
{2poisson}, respectively. Consider the two-sample test based on the
$K$-NN graph functional $\cN_{K}$ with rejection region \eqref
{rejregionKNN} for the testing problem \eqref{epsilonN}.
\begin{enumerate}[(a)]

\item If the dimension $d \leq8$, then the following hold:
\begin{itemize}
\item[--] $ \llVert  N^{\frac{1}{4}}\varepsilon_{N} \rrVert   \rightarrow0$: The
limiting power of the test \eqref{rejregionKNN} is $\alpha$.

\item[--] $N^{\frac{1}{4}}\varepsilon_{N} \rightarrow h$: Then if
dimension dimension $d \leq7$, limiting power of the test \eqref
{rejregionKNN} is
%
\begin{align}
\label{eq:effhessian} \Phi \bigl(z_\alpha+b_{K, \theta_{1}}(h) \bigr).
\end{align}
Otherwise, dimension $d=8$ and the limiting power is
%
\begin{align}
\label{eq:effgradhessian}\Phi \bigl(z_\alpha+ a_{K,
\theta_{1}}(h)+b_{K, \theta_{1}}(h)
\bigr),
\end{align}
where $a_{K, \theta_{1}}(h)$ and $b_{K, \theta_{1}}(h)$ are as
defined above.

\item[--] $ \llVert  N^{\frac{1}{4}}\varepsilon_{N} \rrVert   \rightarrow\infty$:
The limiting power of the test \eqref{rejregionKNN} is $1$.
\end{itemize}

\item If the dimension $d \geq9$, then the following hold:
\begin{itemize}
\item[--] $ \llVert  N^{\frac{1}{2}-\frac{2}{d}}\varepsilon
_{N} \rrVert  \rightarrow0$: The limiting power of the test \eqref
{rejregionKNN} is $\alpha$.

\item[--] $N^{\frac{1}{2}-\frac{2}{d}}\varepsilon_{N} \rightarrow
h$: The limiting power of the test \eqref{rejregionKNN} is
%
\begin{align}
\label{eq:effgrad} \Phi \bigl(z_\alpha+a_{K, \theta_{1}}(h) \bigr).
\end{align}

\item[--] $ \llVert  N^{\frac{1}{2}-\frac{2}{d}}\varepsilon
_{N} \rrVert  \rightarrow\infty$ such that $ \llVert  N^{\frac{2}{d}}\varepsilon
_{N} \rrVert   \rightarrow0$: Then depending on whether
\[
\frac{1}{||\varepsilon_N||}\int_{S} \varepsilon_{N}^\top \grad
_{\theta_{1}} \biggl(\frac{ \tr(\mathrm H_{x} f(x|\theta
_{1}))}{f(x|\theta_{1})} \biggr) f^{\frac{d-2}{d}}(x|
\theta_{1}) \,\mathrm {d} x \rightarrow\begin{cases}
> 0 ,
\\
< 0 , \end{cases}
\]
the limiting power of the test \eqref{rejregionKNN} is $0$ or $1$,
respectively.

\item[--] $N^{\frac{2}{d}}\varepsilon_{N} \rightarrow h$: The
limiting power of the test \eqref{rejregionKNN} is $0$ or $1$,
depending on whether $a_{K, \theta_{1}}(h)+ b_{K, \theta_{1}}(h)$ is
negative or positive, respectively.

\item[--] $ \llVert  N^{\frac{2}{d}}\varepsilon_{N} \rrVert   \rightarrow\infty$:
The limiting power of the test \eqref{rejregionKNN} is $1$.\vadjust{\goodbreak}
\end{itemize}
\end{enumerate}
\end{thm}

The theorem is pictorially summarized in Figure~\ref{fig:effsecond},
and the proof is given in the Appendix. We elaborate on the
implications of this result, and its several interesting consequences below:
\begin{enumerate}[(a)]

\item[(a)] Theorem~\ref{EFFSECOND} shows that for dimension $d \leq7$, the
detection threshold of the test is at $N^{-\frac{1}{4}}$. More
precisely, if we fix an alternative direction $h \in\R^{p}$, and
suppose $\theta_{2}=\theta_{1}+\delta_{N} h$, for some positive
sequence $\delta_{N} \rightarrow0$, then, by Theorem~\ref
{EFFSECOND}, the limiting power of the test \eqref{rejregionKNN} is
\begin{align*}
\begin{cases} \alpha & \text{if } N^{\frac{1}{4}}
\delta_{N} \rightarrow0,
\\
\Phi \bigl(z_\alpha+ \lambda^{2} b_{K, \theta_{1}}(h)
\bigr)> \alpha & \text{if } N^{\frac{1}{4}}\delta_{N}
\rightarrow\lambda >0,
\\
1 & \text{if } N^{\frac{1}{4}}\delta_{N} \rightarrow
\infty. \end{cases}
\end{align*}
Note that the power at the threshold $N^{\frac{1}{4}}\delta_{N}
\rightarrow\lambda$ is always greater than $\alpha$, because $b_{K,
\theta_{1}}(h) >0$ by Assumption~\ref{assumptionlocalpower}(c). Here,
the limiting power is obtained from the limit of the Hessian of the
mean difference (defined below in \eqref{eq:meanT2}), which can be
thought of as the \emph{second-order efficiency} of the test \eqref
{rejregionKNN}, in comparison to the first-order (Pitman) efficiency,
which is zero in this case (see Section~\ref{sec:pfoutline} below for
more on this analogy).

\item[(b)] For dimension $d=8$, the behavior is similar to the case above,
but there is a subtle difference when $N^{\frac{1}{4}}\delta_{N}
\rightarrow\lambda$. Here, for an alternative direction $h \in\R
^{p}$ and $\theta_{2}=\theta_{1}+\delta_{N} h$ as above, the
limiting power of the test \eqref{rejregionKNN} is
\begin{align*}
\begin{cases} \alpha & \text{if } N^{\frac{1}{4}}
\delta_{N} \rightarrow0,
\\
\Phi \bigl(z_\alpha+ \lambda a_{K, \theta_{1}}(h)+
\lambda^{2} b_{K,
\theta_{1}}(h) \bigr) & \text{if }
N^{\frac{1}{4}}\delta_{N} \rightarrow\lambda>0,
\\
1 & \text{if } N^{\frac{1}{4}}\delta_{N} \rightarrow
\infty. \end{cases}
\end{align*}
Note that at the threshold $N^{\frac{1}{4}}\delta_{N} \rightarrow
\lambda$, the limiting power can be greater than $\alpha$ or less
than $\alpha$, depending on whether
$\lambda a_{K, \theta_{1}}(h)+\lambda^{2} b_{K, \theta_{1}}(h)$ is
positive or negative. In particular, considering the power as a
function of $\lambda$ gives: if $\lambda< -\frac{a_{K, \theta
_{1}}(h)}{b_{K, \theta_{1}}(h)}$, then the limiting power is less than
$\alpha$, and if $\lambda> -\frac{a_{K, \theta_{1}}(h)}{b_{K,
\theta_{1}}(h)}$, then the limiting power is greater than $\alpha$.
Therefore, in dimension 8, the limiting power function is nonmonotone
if $a_{K, \theta_{1}}(h)<0$. The asymptotic power starts off at
$\alpha$, decreases for a while, going below $\alpha$ and making it
\emph{asymptotically biased} (i.e., the limiting power is less than
the size of the test), then starts to increase, going past $\alpha$
and eventually becoming 1, as $\lambda\rightarrow\infty$. This also
shows that for every direction $h \in\R^{d}$ such that $a_{K, \theta
_{1}}(h) < 0$, there is a ``special'' point $\lambda= -\frac{a_{K,
\theta_{1}}(h)}{b_{K, \theta_{1}}(h)} > 0$, where the limiting power
is exactly $\alpha$.

\item[(c)] A surprising phenomenon happens for dimension $d \geq9$:
Here, unlike for dimension 8 or smaller, the precise location of the
detection threshold depends on the distribution of the data under the
null $f(\cdot|\theta_{1})$ and the direction of the alternative. As
before, fix an alternative direction $h \in\R^{p}$, and suppose
$\theta_{2}=\theta_{1}+\delta_{N} h$, for some positive sequence
$\delta_{N} \rightarrow0$. Then, depending on the sign of
$a_{K,\theta_{1}}(h)$ (recall \eqref{eq:aKtheta1}), there are two cases:
\begin{itemize}
\item[--] Suppose $a_{K,\theta_{1}}(h)>0$. Then, by Theorem~\ref
{EFFSECOND}, the limiting power of the test \eqref{rejregionKNN} is
%
\begin{align}
\label{eq:sigma_positive} \begin{cases} \alpha & \text{if }
N^{\frac{1}{2}-\frac{2}{d}}\delta_{N} \rightarrow0,
\\
\Phi \bigl(z_\alpha+ \lambda a_{K, \theta_{1}}(h) \bigr)> \alpha &
\text{if } N^{\frac{1}{2}-\frac{2}{d}}\delta_{N} \rightarrow
\lambda>0,
\\
1 & \text{if } N^{\frac{1}{2}-\frac{2}{d}}\delta_{N} \rightarrow
\infty. \end{cases}
\end{align}
Here, the detection threshold of the test \eqref{rejregionKNN} is at
$\Theta(N^{-\frac{1}{2}+\frac{2}{d}})$, that is, the limiting power
transitions from $\alpha$ to 1 at $\delta_{N}=\Theta(N^{-\frac
{1}{2}+\frac{2}{d}})$. Note that the detection threshold $N^{-\frac
{1}{2}+\frac{2}{d}}$ improves with dimension, moving closer to the
parametric rate of $N^{-\frac{1}{2}}$ as the dimension $d$ grows to
infinity, exhibiting a \emph{blessing of dimensionality}. An example
where this is attained is the truncated spherical normal problem (see
Section~\ref{sec:snormal} below).

\item[--] Suppose $a_{K,\theta_{1}}(h)<0$. By Theorem~\ref
{EFFSECOND}, the limiting power of the test \eqref{rejregionKNN} is
\begin{align*}
\begin{cases} \alpha & \text{if } N^{\frac{1}{2}-\frac{2}{d}}
\delta_{N} \rightarrow0,
\\
\Phi \bigl(z_\alpha+ \lambda a_{K, \theta_{1}}(h) \bigr) < \alpha &
\text{if } N^{\frac{1}{2}-\frac{2}{d}}\delta_{N} \rightarrow
\lambda>0,
\\
0 & \text{if } N^{\frac{1}{2}-\frac{2}{d}}\delta_{N} \rightarrow
\infty\text{ and } N^{\frac{2}{d}}\delta_{N} \rightarrow0,
\\
0 & \text{if } N^{\frac{2}{d}}\delta_{N} \rightarrow
\kappa\text{ and}
\\
& \kappa a_{K, \theta_{1}}(h)+ \kappa^{2} b_{K, \theta_{1}}(h)
<0 ,
\\
1 & \text{if } N^{\frac{2}{d}}\delta_{N} \rightarrow
\kappa\text{ and}
\\
& \kappa a_{K, \theta_{1}}(h)+ \kappa^{2} b_{K, \theta_{1}}(h) >
0 ,
\\
1 & \text{if } N^{\frac{2}{d}}\delta_{N} \rightarrow
\infty. \end{cases}
\end{align*}
Note that in this case the limiting power function is nonmonotone and
asymptotically biased, it starts off at $\alpha$, then goes below
$\alpha$, eventually drops to zero, and then transitions up to 1. This
surprising phenomenon happens because the test \eqref{rejregionKNN}
has a one-sided rejection region, and it is universally consistent.
Therefore, the limiting power when $N^{\frac{1}{2}-\frac{2}{d}}\delta
_{N} \rightarrow\lambda>0$ is given by the normal lower tail, more precisely,
$\Phi (z_\alpha+ \lambda a_{K, \theta_{1}}(h) )$.
Therefore, for a direction chosen such that $a_{K,\theta_{1}}(h)<0$,
the power drops below $\alpha$ and then goes to zero when $\lambda
\rightarrow\infty$, but it has to eventually go up to 1 because of
consistency, hence the nonmonotonicity. In this case, the power
transitions from 0 to 1, at $\delta_{N}=\Theta(N^{-\frac{2}{d}})$,
which becomes worse with dimension (converging finally to fixed
difference alternatives as $d \rightarrow\infty$), exhibiting a curse
of dimensionality. Again, this is attained in the truncated spherical
normal problem (see Section~\ref{sec:snormal} below). Theorem~\ref
{EFFSECOND} also gives the limiting power at the threshold $N^{\frac
{2}{d}} \delta_{N} \rightarrow\kappa>0$. Here, the limiting power of the
test \eqref{rejregionKNN} converges to $0$ or $1$, depending on
whether $\kappa a_{K, \theta_{1}}(h)+ \kappa^{2} b_{K, \theta
_{1}}(h)$ is negative or positive, respectively. In other words,
considering the limiting power as a function of $\kappa$ gives: if
$\kappa< -\frac{a_{K, \theta_{1}}(h)}{b_{K, \theta_{1}}(h)}$, then
the limiting power is 0, and if $\kappa> -\frac{a_{K, \theta
_{1}}(h)}{b_{K, \theta_{1}}(h)}$, then the limiting power is 1. This
happens because at the threshold $N^{\frac{2}{d}} \delta_{N}
\rightarrow\kappa>0$, the gradient and Hessian of the mean difference
(defined below in \eqref{eq:meanT2}) are of the same order, and the
limiting power is 0 or 1 depending on whether the sum of the gradient
and the Hessian diverges to $\infty$ or $-\infty$, which is in turn
determined by the sign of $\kappa a_{K, \theta_{1}}(h)+ \kappa^{2}
b_{K, \theta_{1}}(h)$.
(Note that, similar to case (b) above, there is a ``special'' point
$\kappa= -\frac{a_{K, \theta_{1}}(h)}{b_{K, \theta_{1}}(h)}$, where
the theorem is unable to say anything about the limiting power, when
$N^{\frac{2}{d}} \delta_{N} \rightarrow\kappa>0$. If this happens,
then the limiting power depends on the higher-order expansions of the
gradient and the Hessian of the mean difference, which has to be
calculated individually for specific examples.)
\end{itemize}
The discussion above shows that for dimension 9 and higher, given a
family of distributions $\{\P_\theta: \theta\in\Theta\subset\R
^{p}\}$ and an alternative direction $h \in\R^{p}$, there are some
``good directions'' (where $a_{K,\theta_{1}}(h)>0$) where the test
\eqref{rejregionKNN}
exhibits a blessing of dimensionality, but at the same time there are
``bad directions'' (where $a_{K,\theta_{1}}(h)<0$) where one sees a
curse of dimensionality. For simulations illustrating this phenomenon,
refer to Section~\ref{sec:snormal} below.

\item[(d)] Note that Theorem~\ref{EFFSECOND} does not tell us what the
detection threshold is when $a_{K,\theta_{1}}(h)=0$. These are the
``degenerate directions,'' for which the precise location of the
detection threshold has to be determined on a case by case basis: For
example, this happens in the normal location problem (see Section~\ref{sec:nlocation} below), where a direct calculation shows that,
irrespective of the dimension, the detection threshold is at $\Theta
(N^{-\frac{1}{4}})$, for all directions.

\item[(e)] The rates obtained in Theorem~\ref{EFFSECOND} can be summarized
in terms of the \emph{critical exponents},
%
\begin{align}
\label{betad} \beta_{d}=\begin{cases}
\frac{1}{4} & \text{if } d \leq8,
\\
\frac{1}{2}-\frac{2}{d} &\text{if } d\geq9,
\end{cases} \qquad \gamma_{d}=\begin{cases}
\frac{1}{4} & \text{if } d \leq8,
\\[2pt]
\frac{2}{d} & \text{if } d\geq9. \end{cases}
\end{align}
Theorem~\ref{EFFSECOND} says that (irrespective of the distribution of
the data) for the testing problem \eqref{epsilonN}: (1) if $ \llVert  N^{\beta
_{d}}\varepsilon_{N} \rrVert  \rightarrow0$, the limiting power of the
test \eqref{rejregionKNN} is $\alpha$; and (2) if $ \llVert  N^{\gamma
_{d}}\varepsilon_{N} \rrVert   \rightarrow\infty$, the limiting power of the
test \eqref{rejregionKNN} is $1$.
Note that they are equal up to dimension $d=8$, after which $\beta
_{d}$ increases with $d$ to $\frac{1}{2}$ (recall the $K$-NN test has
no power for $N^{-\frac{1}{2}}$ alternatives \cite{BBB}), and $\gamma
_{d}$ decreases with $d$ to zero (the $K$-NN test always has power
against fixed alternatives).
\end{enumerate}

\subsubsection{Proof outline}
\label{sec:pfoutline}

The proof of Theorem~\ref{EFFSECOND} is given in Appendix \ref{sec:pfeffsecond}. Here, we give an outline of the proof. To find the
limiting local power of the $K$-NN test \eqref{rejregionKNN}, it
suffices to derive the asymptotic distribution of
\begin{align*}
 \frac{1}{\sqrt N} \bigl\{ T\bigl(\cN_{K}\bigl(
\cZ_{N}'\bigr)\bigr)- \E_{H_{0}}\bigl(T\bigl(
\cN _{K}\bigl(\cZ_{N}'\bigr)\bigr)\bigr)
\bigr\} = T_{1}+T_{2},
\end{align*}
where $T_{1}:=\frac{1}{\sqrt N}\{ T(\cN_{K}(\cZ_{N}'))- \E
_{H_{1}}(T(\cN_{K}(\cZ_{N}')))\}$ and the \emph{mean difference}
%
\begin{align}
\label{eq:meanT2} T_{2}:=\frac{1}{\sqrt N} \E_{H_{1}}
\bigl(T\bigl(\cN_{K}\bigl(\cZ_{N}'\bigr)
\bigr)\bigr) - \E _{H_{0}}\bigl(T\bigl(\cN_{K}\bigl(
\cZ_{N}'\bigr)\bigr)\bigr) \},
\end{align}
when $\theta_{2}=\theta_{1}+\varepsilon_{N}$, where $\varepsilon
_{N}$ is as in \eqref{epsilonN}. The proof of Corollary~\ref{cor:KNN}
shows that the first term converges in distribution to $N(0, \sigma
_{K}^{2})$. Therefore, determining the limiting power boils down to
computing the limit of the mean difference $T_{2}$. In the parametric
setup of \eqref{epsilonN}, $\E_{H_{1}}(T(\cN_{K}(\cZ
_{N}'))):=\delta_{N}(\theta_{1}, \theta_{2})$ for some function
$\delta_{N}:\Theta^{2} \rightarrow\R$. (The expression of $\delta
_{N}$ is given in (B.1) in the supplementary materials. Note
that $\delta_{N}$ is related to the function $\mu_{N}$ in the
Appendix as: $\delta_{N}(\theta_{1}, \theta_{2}) =\frac
{N_{1}N_{2}}{N^{2}} \mu_{N}(\theta_{1}, \theta_{2})$.) Then by a
Taylor series expansion in the second coordinate (and ignoring the
error term) gives
\begin{align*}
T_{2} & = \frac{1}{\sqrt N} \bigl\{ \delta_{N}(
\theta_{1}, \theta _{1}+\varepsilon_{N})-
\delta_{N}(\theta_{1}, \theta_{1}) \bigr\}
\\
& \approx\frac{\varepsilon_{N}^\top}{\sqrt
N} \grad\delta_{N}(
\theta_{1}, \theta_{1}) + \frac{\varepsilon
_{N}^\top\mathrm H[\delta_{N}(\theta_{1}, \theta_{1})] \varepsilon
_{N}}{2 \sqrt N},
\end{align*}
where $\grad\delta_{N}(\theta_{1}, \theta_{1}):=\grad_{\theta}
\delta_{N}(\theta_{1}, \theta)|_{\theta=\theta_{1}} \in\R^{p}$
is the gradient vector (with respect to the second coordinate $\theta
$) of $\delta_{N}(\theta_{1}, \theta)$ evaluated at $\theta=\theta
_{1}$, and $\mathrm H[\delta_{N}(\theta_{1}, \theta_{1})]\in\R^{p
\times p}$ is the Hessian matrix (with respect to $\theta$) of $\delta
_{N}(\theta_{1}, \theta)$.
The proof of Theorem~\ref{EFFSECOND} involves showing the following steps:
\begin{itemize}
\item[--] $\frac{1}{\sqrt N} \varepsilon_{N}^\top\grad\delta
_{N}(\theta_{1}, \theta_{1})$ has finite limit when $\varepsilon
_{N}=\frac{h}{N^{\frac{1}{2}-\frac{2}{d}}}$ (Lemma~B.1 in
Appendix \ref{sec:pfeffsecond}).

\item[--] $\frac{1}{\sqrt N} \varepsilon_{N}^\top\mathrm H[\delta
_{N}(\theta_{1}, \theta_{1})]\varepsilon_{N}$ has finite limit when
$\varepsilon_{N}=\frac{h}{N^{\frac{1}{4}}}$ (Lemma~B.2
in Appendix \ref{sec:pfeffsecond}).
\end{itemize}

Note that when $d \leq7$, $N^{-\frac{1}{2}+\frac{2}{d}} \ll
N^{-\frac{1}{4}} $ and the Hessian term dominates the gradient term,
giving the formula in \eqref{eq:effhessian}. This can be thought of as
the \emph{second-order efficiency} of the test \eqref{rejregionKNN}.
(This is in analogy with the classical first-order (Pitman) efficiency,
which is derived under local alternatives $\varepsilon_{N}=h/\sqrt N$.
However, the $K$-NN test \eqref{rejregionKNN} has zero-Pitman
efficiency, because the first-order term $\frac{h^\top}{N} \grad
\delta_{N}(\theta_{1}, \theta_{1})$ is asymptotically zero in this
scale, hence the local power is given by the second-order Hessian
term.) On the other hand, when $d=8$, the rate of convergence of the
gradient and the Hessian terms match, and, as a result the
contributions from both the terms show up in \eqref
{eq:effgradhessian}. Finally, when $d\ge9$, the gradient term
dominates the Hessian term (since $N^{-\frac{1}{2}+\frac{2}{d}} \gg
N^{-\frac{1}{4}}$), which explains the shift in the location of the
detection threshold at dimension 8 and gives the expression in \eqref
{eq:effgrad}.

\subsection{Examples}
\label{sec:examples} In this section, we discuss examples which attain
the threshold obtained in Theorem~\ref{EFFSECOND}. In order to meet
compactness assumption in Theorem~\ref{EFFSECOND} (recall Assumption~\ref{assumptionlocalpower}), we consider standard distributions
truncated to a compact, convex set. However, as mentioned earlier, we
expect the results to hold for the un-truncated family (with ``nice''
tails), as well.

\subsubsection{Example: Normal location}
\label{sec:nlocation} Let $A \subset \R^{d}$ be a compact and convex set
which is symmetric around the origin $\bm0 \in\R^{d}$, that is,
$A=-A$. For $\theta\in\R^{d}$, define a family of densities $\phi
_{A}(x|\theta)=\frac{1}{Z_{A}(\theta)}e^{-\frac{1}{2}  \llVert  x-\theta
 \rrVert  ^{2}}$, where $Z_{A}(\theta):=\int_{A} e^{-\frac{1}{2} \llVert  x-\theta
 \rrVert  ^{2}}\,\mathrm{ d} x$, is the normalizing constant.
This is the $d$-dimensional multivariate normal $N(\theta, \mathrm
I_{d})$ truncated to the set $A$.

Now, consider the problem of testing \eqref{epsilonN} based on \eqref
{rejregionKNN}, given i.i.d. samples $\sX_{N_{1}}$ and $\sY_{N_{2}}$
from $\phi_{A}(\cdot|\theta_{1})$ and $\phi_{A}(\cdot|\theta
_{2})$, respectively. There are two cases depending whether the true
$\theta_{1}$ is zero or nonzero. Here, we discuss the case $\theta
_{1}=\bm0$: When $\theta_{1}=\bm0$, it is easy to check that
\[
\int_{A} \grad_{\theta=\bm0} \biggl(
\frac{ \tr(\mathrm H_{x} \phi
_{A}(x|\theta))}{\phi_{A}(x|\theta)} \biggr) \phi_{A}^{\frac
{d-2}{d}}(x| \theta)\,
\mathrm {d} x=0,
\]
which implies Theorem~\ref{EFFSECOND} cannot be directly applied to
the case $d \geq9$. However, in this case a direct calculation shows
that the gradient term is exactly zero across all dimensions, which
implies the following (calculations are given in Lemma~C.1
in Appendix \ref{sec:pfnlocation}): For any $d \geq1$:
\begin{itemize}
\item[--]If $ \llVert  N^{\frac{1}{4}}\varepsilon_{N} \rrVert  \rightarrow0$, the
limiting power of the test is $\alpha$.

\item[--] If $ \llVert  N^{\frac{1}{4}}\varepsilon_{N} \rrVert  \rightarrow\infty
$, the limiting power of the test is $1$.

\item[--] If $N^{\frac{1}{4}}\varepsilon_{N} \rightarrow h$, for
some $h \in\R^{p}\backslash\{\bf{0}\}$, the limiting power of the
test is given by $\Phi(z_\alpha+\frac{r^{2} K}{2 \sigma_{K}} \E_{X
\sim\phi_{A}(\cdot|\bm0)}( h^\top X )^{2})$.
\end{itemize}

Details of the other case $\theta_{1} \ne\bm 0$ can be found in
Appendix \ref{sec:pfnlocation}. In this case, because of the asymmetry
introduced by the truncation,
\[
\int_{A} \grad_{\theta_{1}} \biggl(\frac{ \tr(\mathrm H_{x} \phi
_{A}(x|\theta_{1}))}{\phi_{A}(x|\theta_{1})}
\biggr) \phi _{A}^{\frac{d-2}{d}}(x|\theta_{1}) \,
\mathrm {d} x\ne0,
\]
and hence, the detection threshold undergoes a phase-transition at
dimension 8 as in Theorem~\ref{EFFSECOND}. However, in the untruncated
normal family ($\{\P_\theta\sim N(\theta, \mathrm I_{d}): \theta\in
\R^{d}\}$)
\[
\int_{\R^{d}} \grad_{\theta_{1}} \biggl(
\frac{ \tr(\mathrm H_{x}
\phi_{\R^{d}}(x|\theta_{1}))}{\phi_{\R^{d}}(x|\theta_{1})} \biggr) \phi_{\R^{d}}^{\frac{d-2}{d}}(x|
\theta_{1}) =0,
\]
for all $\theta_{1} \in\R^{d}$, that is, for the untruncated normal
location problem we expect the detection threshold to be at $N^{-\frac
{1}{4}}$, for all dimensions, as seen in the simulations below.

\begin{figure*}[h]\vspace{-0.1in}
\centering
\begin{minipage}[l]{0.49\textwidth}
\centering
\includegraphics[width=2.05in]
    {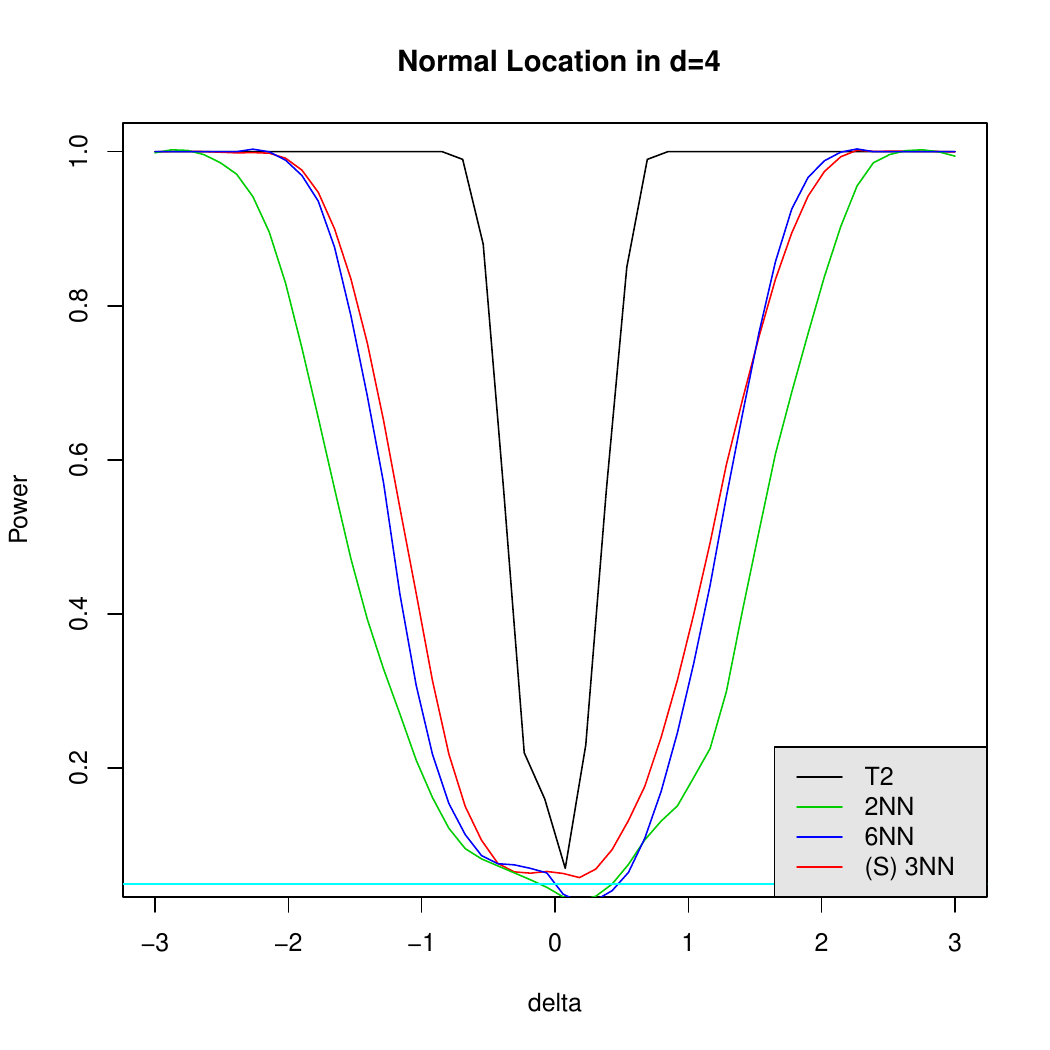}\\
\footnotesize{(a)}
\end{minipage}
\begin{minipage}[l]{0.49\textwidth}
\centering
\includegraphics[width=2.05in]
    {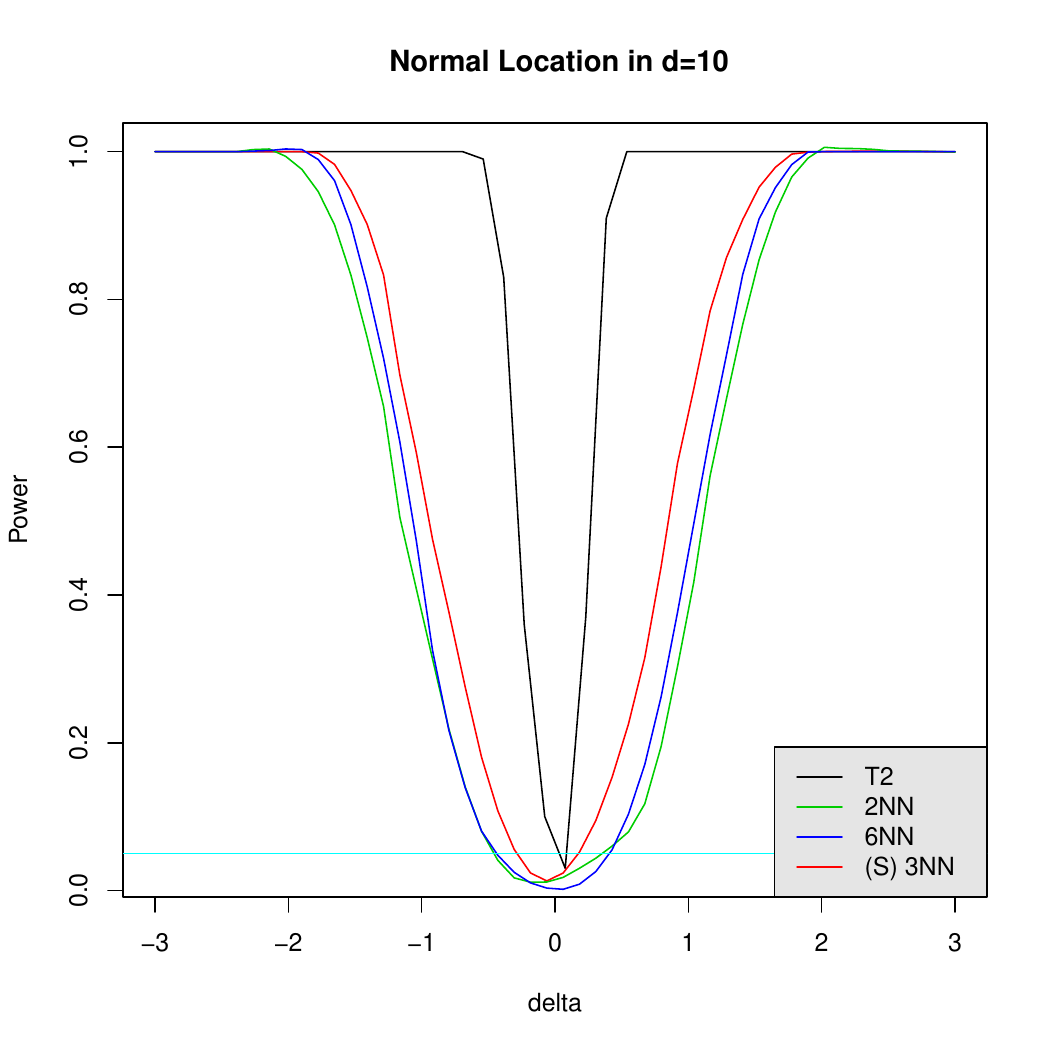}\\
\footnotesize{(b)}
\end{minipage}
\caption{\small{Empirical power for the normal location problem against $\delta N^{-\frac{1}{4}}$ alternatives, in dimension (a) $d=4$ and (b) $d=10$.}}
\label{fig:nlocationI}\vspace{-0.15in}
\end{figure*}

To illustrate the results above, we consider the following simulation:
Consider the parametric family $\P_{\theta}\sim N(\theta, \mathrm
I_{d})$, for $\theta\in\R^{d}$. Figure~\ref{fig:nlocationI} shows
the empirical power (out of 100 repetitions) of the tests based on the
2-NN and 6-NN graphs, the test based on the symmetrized 3-NN graph (see
Appendix \ref{sec:KNNsymmetry} for details on the limiting power of
the symmetrized 3-NN test), and the Hotelling's $T^{2}$ test, with
$N_{1}=2000 \text{ samples from } N(2\cdot\bm1, \mathrm I_{d}) \text
{ and } N_{2}=1000 \text{ samples from } N(2\cdot\bm1+ \delta
N^{-\frac{1}{4}} \bm1, \mathrm I_{d})$, over a grid of 40 values of
$\delta$ in $[-3, 3]$ (smoothed out using the \texttt{loess} function
in \texttt{R}), in (a) dimension 4 and (b) dimension 10. (Here,
$N=N_{1}+N_{2}=3000$.) The level of the tests are set to $\alpha
=0.05$. The plots show that the tests based on the NN graphs have
nontrivial local power as a function of $\delta$, as predicted by the
calculations above. Note that, in this case, the most powerful test is
the Hotelling's $T^{2}$-test, which has detection threshold at
$N^{-\frac{1}{2}}$ and, therefore, has high power at the $N^{-\frac
{1}{4}}$ scale, as seen in the plots.

Figure~\ref{fig:nlocationII} shows the empirical power (out of 100
repetitions) of the different tests with $N_{1}=5000$ samples from $N(2
\cdot\bm1, \mathrm I_{d})$ and $N_{2}=3000$ samples from $N( 2 \cdot
\bm1 + N^{-b} \cdot\bm h, \mathrm I_{d})$, where $b$ varies over a
grid of 100 values in $[0, 1]$, $\bm h=\bm1$ and dimension (a) $d=4$
and (b) $d=10$. Note that $b=0$ corresponds to fixed alternatives where
the power is expected to be near 1 because of consistency. The level of
the tests are set to $\alpha=0.25$. Note that the power of the tests
based on the $K$-NN graphs transitions from $\alpha$ to 1 around
$b=0.25$, which corresponds to the rate $N^{-\frac{1}{4}}$, in both
dimensions, as predicted by the calculations above. On the other hand,
the power of the Hotelling's $T^{2}$ test transitions from $\alpha$ to
1 around $b=0.5$, which corresponds to the parametric rate of
$N^{-\frac{1}{2}}$. The corresponding plots for the negative direction
$\bm h= -\bm1$ are given in Appendix \ref{sec:normal_location}.

\begin{figure*}[h]\vspace{-0.1in}
\centering
\begin{minipage}[l]{0.49\textwidth}
\centering
\includegraphics[width=2.05in]
    {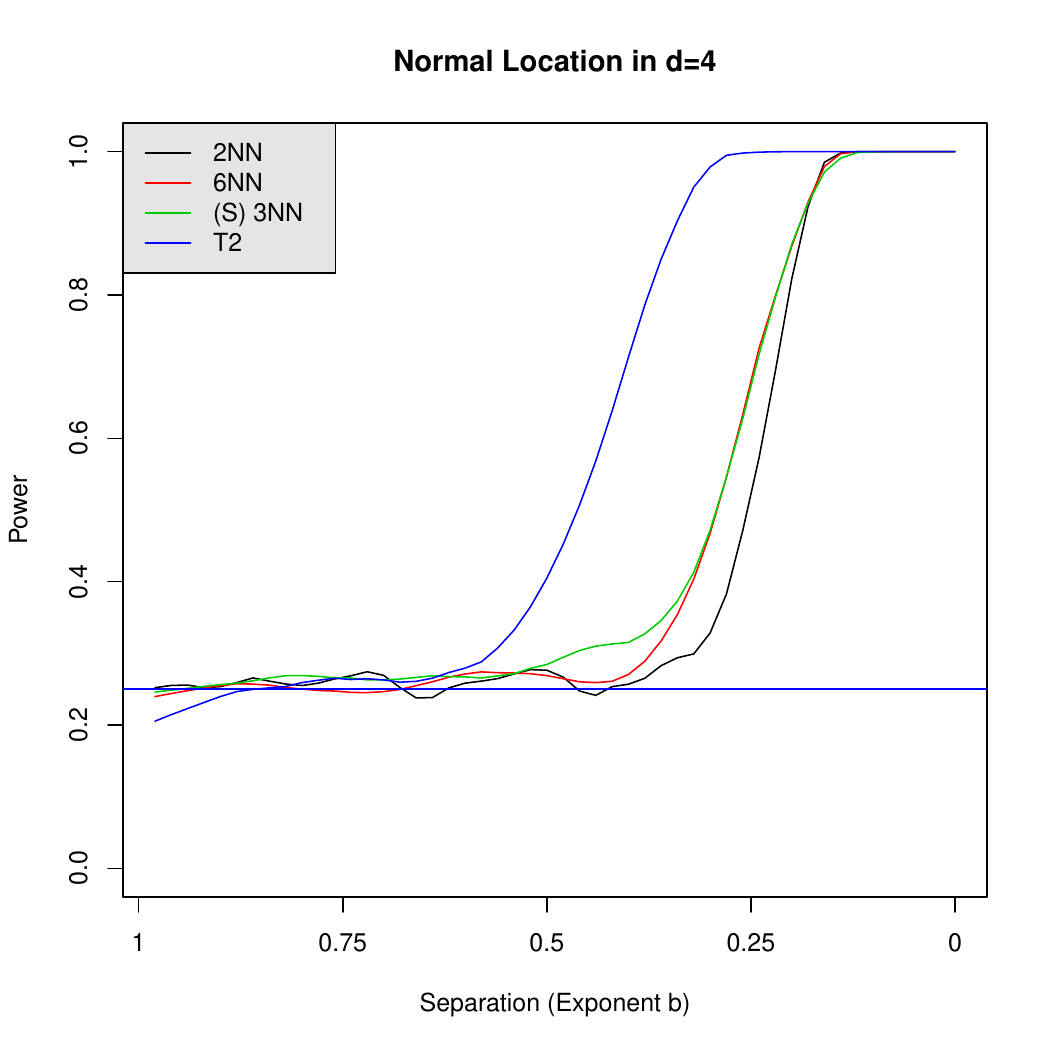}\\
\footnotesize{(a)}
\end{minipage}
\begin{minipage}[l]{0.49\textwidth}
\centering
\includegraphics[width=2.05in]
    {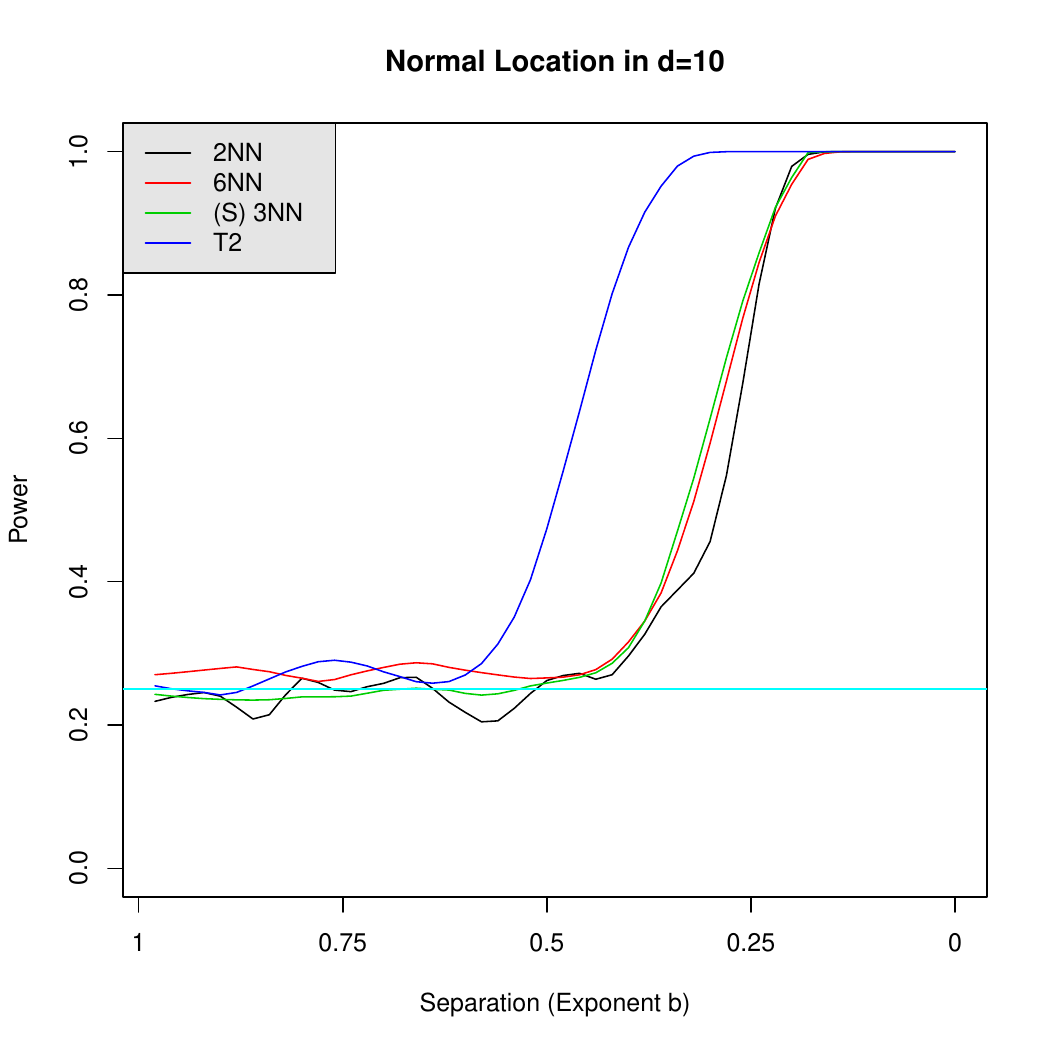}\\
\footnotesize{(b)}
\end{minipage}
\caption{\small{Empirical power in the normal location problem with $N_1=5000$ samples from $N(2 \cdot \bm 1, \mathrm I_d)$ and  $N_2=3000$  samples from $N( 2  \cdot \bm 1 + N^{-b} \cdot \bm h, \mathrm I_d)$, where $b$ varies over a grid of 100 values in $[0, 1]$,  $\bm h=\bm 1$, and dimension (a) $d=4$ and (b) $d=10$.}}
\label{fig:nlocationII}\vspace{-0.15in}
\end{figure*}

\subsubsection{Example: Spherical normal}
\label{sec:snormal}

Let $M$ be a convex, compact subset of $\R^{d}$. For $\lambda>0$,
define a family of densities $\phi_{M}(\cdot|\lambda^{2})$:
\[
\phi_{M}\bigl(x|\lambda^{2}\bigr)=
\frac{1}{Z_{M}(\lambda^{2})}e^{-\frac{1}{2
\lambda^{2}}  \llVert  x \rrVert  ^{2}}\quad \text{for } x \in M,
\]
where $Z_{M}(\lambda^{2}):=\int_{M} e^{-\frac{1}{2\lambda
^{2}} \llVert  x \rrVert  ^{2}}\,\mathrm {d} x$ is the normalizing constant. (Note that
$Z_{\R^{d}}(\lambda^{2})=(2 \pi\lambda^{2})^\frac{d}{2}$.) This is
the $d$-dimensional spherical normal distribution $N(0, \lambda^{2}
\mathrm I_{d})$ truncated to the set $M$. Now, consider the problem of
testing \eqref{epsilonN} based on \eqref{rejregionKNN}, given i.i.d.
samples $\sX_{N_{1}}$ and $\sY_{N_{2}}$ from $\phi_{M}(\cdot
|\lambda_{1}^{2})$ and $\phi_{M}(\cdot|\lambda_{2}^{2})$,
respectively. In this case, for $h \in\R$,
%
\begin{align}
\label{eq:nscalegradhessian}
\begin{split}
&\int_{M} h \cdot\grad_{\lambda_{1}} \biggl(
\frac{ \tr(\mathrm
H_{x} \phi_{M}(x|\lambda_{1}^{2}))}{\phi_{M}(x|\lambda_{1}^{2})} \biggr) \phi_{M}^{\frac{d-2}{d}}\bigl(x|
\lambda_{1}^{2}\bigr)
\\
& \quad = -\frac{Z_{M}(\frac{d\lambda_{1}^{2}}{d-2} )}{Z_{M}(\lambda
_{1}^{2})^{\frac{d-2}{d}}}\frac{2 h}{\lambda_{1}^{5}} \bigl(2 \E \llVert W \rrVert
^{2} - \lambda_{1}^{2} d \bigr),
\end{split}
\end{align}
where $W=(W_{1}, W_{2}, \ldots, W_{d})'$ are i.i.d. from the density
$\phi_{M}(\cdot|\frac{d\lambda_{1}^{2}}{d-2})$ (see Section~D in the supplementary material for details). Therefore
(recall \eqref{eq:aKtheta1}),
%
\begin{align}
\label{eq:aKtheta_scale} a_{K, \lambda_{1}}(h)=\frac{r p C_{K, 2}}{2 d \sigma_{K}}
\frac
{Z_{M}(\frac{d\lambda_{1}^{2}}{d-2} )}{Z_{M}(\lambda_{1}^{2})^{\frac
{d-2}{d}}}\frac{h}{\lambda_{1}^{5}} \bigl(2\E \llVert W \rrVert
^{2} -\lambda _{1}^{2} d \bigr),
\end{align}
which is positive or negative depending on whether $h$ is positive or
negative. Therefore, the limiting power of the test \eqref
{rejregionKNN} for dimension $d \geq9$, at the threshold $N^{\frac
{1}{2}-\frac{2}{d}}\varepsilon_{N} \rightarrow h$, is $\Phi(z_\alpha
+ a_{K, \lambda_{1}}(h))$.
(Note that in the simulations below we will consider the untruncated
spherical normal family $\{\P_\lambda\sim N(0, \lambda^{2} \mathrm
I_{d}): \lambda>0 \}$. The limiting power in this case can be obtained
by choosing $M=[-L, L]^{d}$, and taking $L \rightarrow\infty$ in
\eqref{eq:aKtheta_scale}.)

Now, suppose we are given i.i.d. samples $\sX_{N_{1}}$ from $\phi
_{M}(\cdot|\lambda_{1}^{2})$ and $\sY_{N_{2}}$ from $\phi_{M}(\cdot
|\lambda_{2}^{2})$, where $\lambda_{2}=\lambda_{1}+ h \delta_{N} >
0$, for some $h$ fixed and $\delta_{N} \rightarrow0$, as $N
\rightarrow\infty$. Then, by Theorem~\ref{EFFSECOND}, depending on
the dimension and the sign of $h$ we have the following cases:
\begin{itemize}
\item For dimension $d \leq8$, irrespective of the sign of $h$, the
limiting power of the test \eqref{rejregionKNN} is $0$ or $1$,
depending on whether $N^{\frac{1}{4}}\delta_{N} \rightarrow0$ or
$N^{\frac{1}{4}}\delta_{N} \rightarrow\infty$. At the threshold,
$N^{\frac{1}{4}}\delta_{N} \rightarrow\kappa$, the limiting power
is given by \eqref{eq:effhessian} or \eqref{eq:effgradhessian} (with
$h$ replaced by $\kappa h$). This is illustrated in Figure~\ref{fig:nscale_separationI}, which shows the empirical power (out of 100
repetitions) of the tests based on the 2-NN and 6-NN graphs, the test
based on the symmetrized 3-NN graph, and the generalized likelihood
ratio test (GLR), in dimension $d=4$, with $N_{1}=12\text{,}000$ samples from
$N(0, 3^{2} \cdot\mathrm I_{d})$ and $N_{2}=6000$ samples from $N(0,
(3+ h N^{-b})^{2} \mathrm I_{d})$, where $b$ varies over a grid of 100
values in $[0, 1]$ and (a) $h=2$ (b) $h=-2$. (Here,
$N=N_{1}+N_{2}=18\text{,}000$.) The level of the tests are set to $\alpha
=0.25$. Note that the power of the tests based on the $K$-NN graphs
transitions from $\alpha$ to 1 around $b=0.25$ (irrespective of the
sign of $h$), which corresponds to the rate $N^{-\frac{1}{4}}$, as
shown in the calculations above. On the other hand, the power of the
GLR test transitions from $\alpha$ to 1 around $b=0.5$, which
corresponds to the parametric rate of $N^{-\frac{1}{2}}$.

\begin{figure*}[h]\vspace{-0.1in}
\centering
\begin{minipage}[l]{0.495\textwidth}
\centering
\includegraphics[width=2.05in]
    {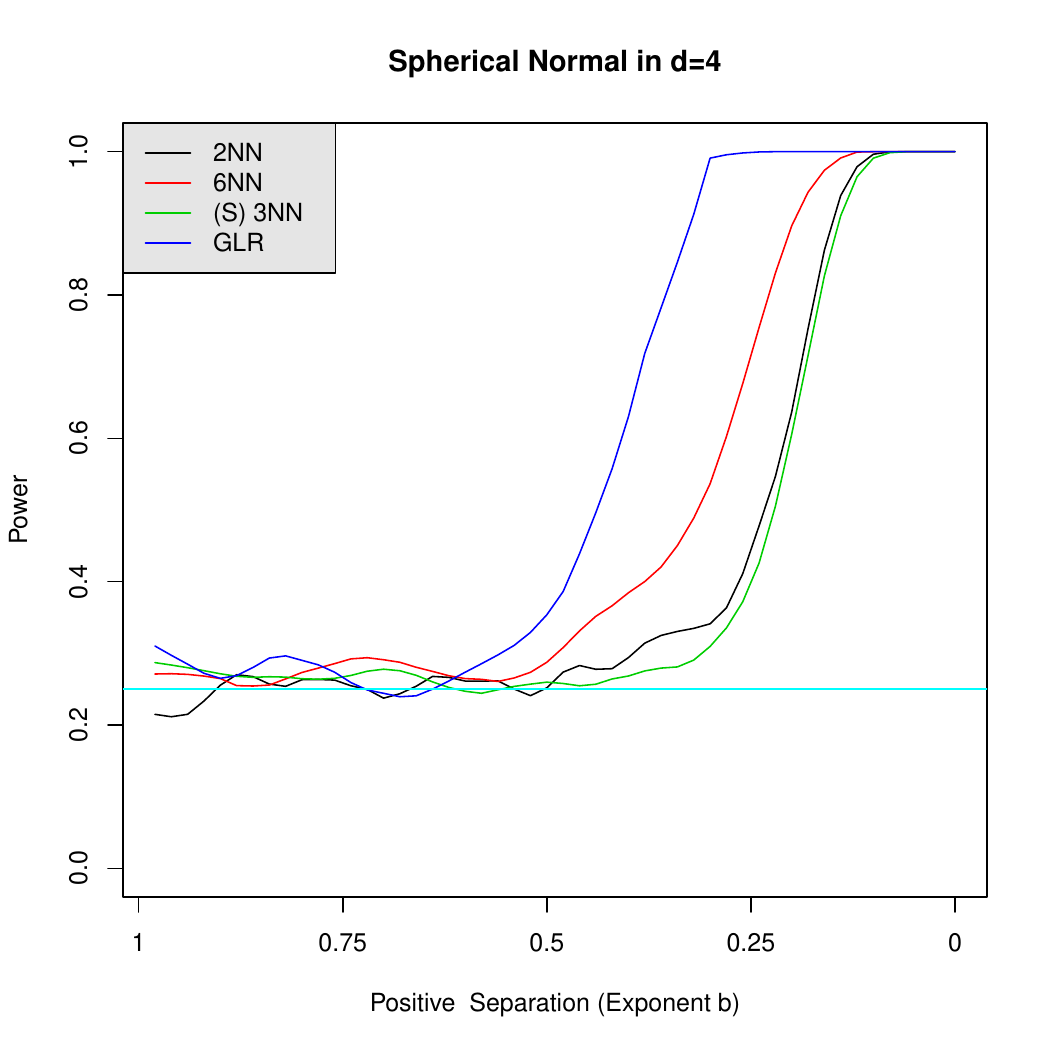}\\
\small{(a)}
\end{minipage}
\begin{minipage}[c]{0.495\textwidth}
\centering
\includegraphics[width=2.05in]
    {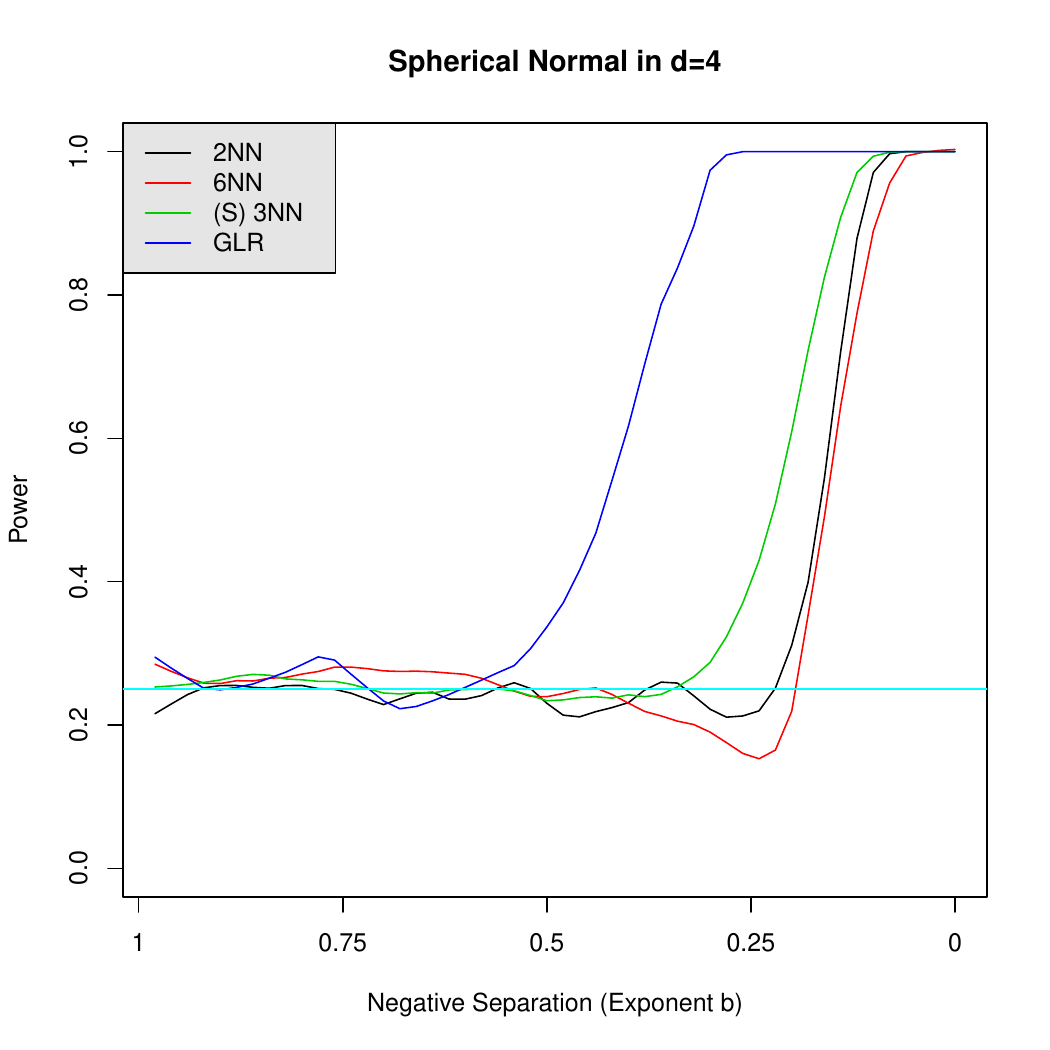}\\
\small{(b)}
\end{minipage}
\caption{\small{Empirical power in the spherical normal problem in dimension $d=4$ with $N_1=12000$ samples from $N(0, 3^2 \cdot \mathrm I_d)$ and  $N_2=6000$  samples from $N(0, (3+ h N^{-b})^2 \mathrm I_d)$, where $b$ varies over a grid of 100 values in $[0, 1]$ and (a) $h=2$ (b) $h=-2$.}}
\label{fig:nscale_separationI}\vspace{-0.15in}
\end{figure*}
%

\item Next, suppose $d\geq9$. Then depending on the sign of $h$ the
following cases arise:
\begin{itemize}
\item[--] Suppose $h >0$ (then $a_{K,\lambda_{1}}(h)>0$). By \eqref
{eq:sigma_positive}, the limiting power of the test \eqref
{rejregionKNN} is
\[
\begin{cases} \alpha & \text{if } N^{\frac{1}{2}-\frac{2}{d}}
\delta_{N} \rightarrow0,
\\
\Phi \bigl(z_\alpha+ \kappa a_{K, \lambda_{1}}(h) \bigr)> \alpha &
\text{if } N^{\frac{1}{2}-\frac{2}{d}}\delta_{N} \rightarrow \kappa>0,
\\
1 & \text{if } N^{\frac{1}{2}-\frac{2}{d}}\delta_{N} \rightarrow
\infty, \end{cases}
\]
where $a_{K, \lambda_{1}}(h)$ is defined above in \eqref
{eq:aKtheta_scale}. Here, the detection threshold exhibits a blessing
of dimensionality, improving with dimension to the parametric rate of
$N^{-\frac{1}{2}}$ as the dimension $d$ grows to infinity. This is
illustrated in Figure~\ref{fig:nscale_separationII}(a), which shows
the empirical power (out of 100 repetitions) of the different tests in
dimension $d=10$, with $N_{1}=300\text{,}000$ samples from $N(0, 3^{2} \cdot
\mathrm I_{d})$ and $N_{2}=200\text{,}000$ samples from $N(0, (3+ h N^{-b})^{2}
\mathrm I_{d})$, where $b$ varies over a grid of 100 values in $[0, 1]$
and $h=2$. As before, the level of the tests are set to $\alpha=0.25$.
Note that the power of the tests based on the $K$-NN graphs transitions
from $\alpha$ to 1 around $b=\frac{1}{2}-\frac{2}{d}=0.3$, which is
the predicted rate of $N^{-\frac{1}{2}+\frac{2}{d}}$. As before, the
power of the GLR test transitions from 0 to 1 around $b=0.5$. To see
the transitions more sharply and observe the local power of the
different tests, we can zoom in at the thresholds (see Appendix \ref{sec:normal_scale}).

\begin{figure*}[h] 
\centering
\begin{minipage}[l]{0.495\textwidth}
\centering
\includegraphics[width=2.05in]
    {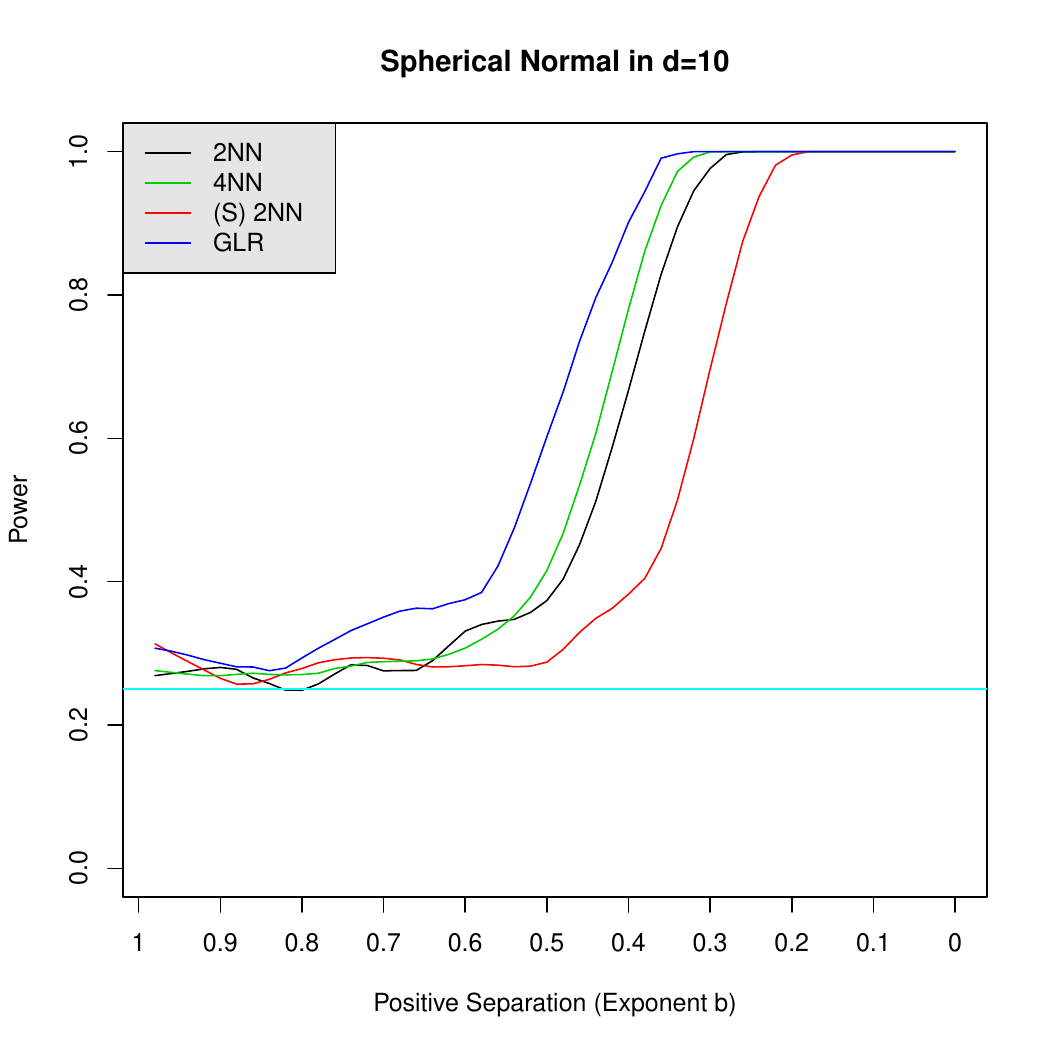}\\
\small{(a)}
\end{minipage}
\begin{minipage}[c]{0.495\textwidth}
\centering
\includegraphics[width=2.05in]
    {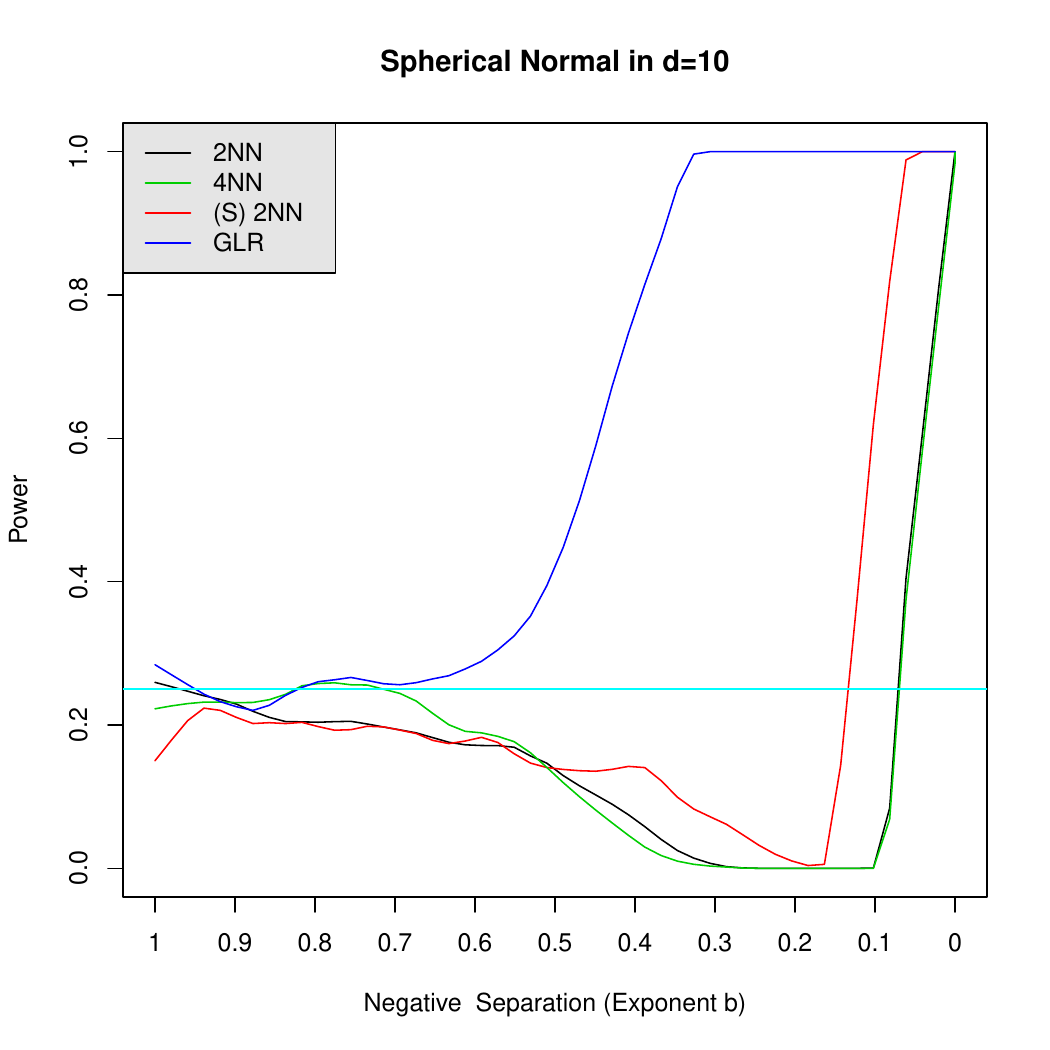}\\
\small{(b)}
\end{minipage}
\caption{\small{Empirical power in the spherical normal problem in dimension $d=10$ with $N_1$ samples from $N(0, 3^2 \cdot \mathrm I_d)$ and  $N_2$  samples from $N(0, (3+ h N^{-b})^2 \mathrm I_d)$, where $b$ varies over a grid of 100 values in $[0, 1]$ and (a) $h=2$ (b) $h=-2$.}}
\label{fig:nscale_separationII}\vspace{-0.15in}
\end{figure*}
%

\item[--] Suppose $h<0$ (then $a_{K,\theta_{1}}(h)<0$). By Theorem~\ref{EFFSECOND}, the limiting power of the test \eqref{rejregionKNN} is
\[
\begin{cases} \alpha & \text{if } N^{\frac{1}{2}-\frac{2}{d}}
\delta_{N} \rightarrow0,
\\
\Phi \bigl(z_\alpha+ \kappa a_{K, \theta_{1}}(h) \bigr) < \alpha &
\text{if } N^{\frac{1}{2}-\frac{2}{d}}\delta_{N} \rightarrow \kappa>0,
\\
0 & \text{if } N^{\frac{1}{2}-\frac{2}{d}}\delta_{N} \rightarrow
\infty\text{ and } N^{\frac{2}{d}}\delta_{N} \rightarrow0,
\\
1 & \text{if } N^{\frac{2}{d}}\delta_{N} \rightarrow
\infty. \end{cases}
\]
This is illustrated in Figure~\ref{fig:nscale_separationII}(b), which
shows the empirical power (out of 100 repetitions) of the different
tests in dimension $d=10$, with $N_{1}=200\text{,}000$ samples from $N(0, 3^{2}
\cdot\mathrm I_{d})$ and $N_{2}=100\text{,}000$ samples from $N(0, (3+ h
N^{-b})^{2} \mathrm I_{d})$, where $b$ varies over a grid of 100 values
in $[0, 1]$ and $h=-2$. Here, we observe the predicted non-monotonicity
of the power of the $K$-NN tests. The asymptotic power starts of at the
level $\alpha=0.25$, goes down to zero (predicted by the theorem at
$b=\frac{1}{2}-\frac{2}{d}=0.3$), stays at zero for a while and jumps
up to 1 (predicted by the theorem at $b=\frac{2}{d}=0.2$). Additional
simulations zooming in to the different thresholds are given in
Appendix \ref{sec:normal_scale}.
\end{itemize}
\end{itemize}

\section*{Acknowledgments} 
\small{The author is indebted to his advisor Persi Diaconis for introducing
him to graph-based tests and
for his constant encouragement and support. The author thanks
Riddhipratim Basu, Sourav Chatterjee, Jerry Friedman,
Shirshendu Ganguly and Susan Holmes for illuminating discussions and
helpful comments. The author also thanks
the associate editor and the anonymous referees for their detailed and
thoughtful comments, which greatly improved the quality of the paper.}

\normalsize

\appendix

\section{Proof of CLT under general alternatives}
\label{sec:pfclt}

In this section the asymptotic distribution of the two-sample test based on stabilizing geometric graphs in the Poissonized setting is derived. The section is organized as follows:  Begin by recalling preliminaries about geometric graphs in Section \ref{sec:preliminaries}. In Section \ref{sec:techlemmas} a few technical lemmas are proved, which will be required to derive the asymptotic variance of the statistic \eqref{R1}. The  consistency of these tests under general alternatives (Proposition \ref{CONSISTENT}) is given in Section \ref{sec:pfconsistency}. The proof of the conditionally centered CLT of the test statistic (Theorem \ref{TH:CLT_R1}) is described in Section \ref{R}. The CLT of the conditional mean and the proofs of Proposition \ref{CLT_R2} and Theorem \ref{TH:CLT_R} are given in Section \ref{sec:pfR2R}. 

\subsection{Preliminaries on stabilizing graphs and Palm theory}\label{sec:preliminaries}

Given a graph functional $\sG$, $\varphi(z, \sG(\cZ))$ is a measurable $\R^+$ valued function defined for all locally finite set $\cZ\subset \R^d$ and $z\in \cZ$. If $z\notin \cZ$, then  $\varphi(z, \sG(\cZ)):=\varphi(z, \sG(\cZ\cup\{z\}))$. The function $\varphi$ is {\it translation invariant} if $\varphi(y+z, \sG(y+\cZ))=\varphi(z, \sG(\cZ))$, and {\it scale invariant} if $\varphi(a z, \sG(a\cZ))=\varphi(z, \sG(\cZ))$, for all $y \in \R^d$ and $a\in \R^+$.  Similar to stabilizing graph functionals, Penrose and Yukich \cite{py} defined stabilizing functions of  graph functionals as follows:

\begin{defn}(Penrose and Yukich \cite{py}) \label{defn:stabilize}
For any locally finite point set $\cZ\subset \R^d$ and any integer $m \in \N$
$$\overline \varphi(\sG(\cZ), M):=\sup_{N\in \N}\left(\esssup_{\substack{\cA\subset \R^d\setminus B(0, M)\\|\cA|=N}}\left\{\varphi(0, \sG(\cZ\cap B(0, M) \cup \cA))\right\}\right)$$
and
$$\underline \varphi(\sG(\cZ), M):=\inf_{N\in \N}\left(\essinf_{\substack{\cA\subset \R^d\setminus B(0, M)\\|\cA|=N}}\left\{\varphi(0, \sG(\cZ\cap B(0, M) \cup \cA))\right\}\right),$$
where the essential supremum/infimum is taken with respect to the Lebesgue measure on $\R^{dN}$. The functional $\varphi$ is said to {\it stabilize} on $\sG(\cZ)$ if
$$\liminf_{M\rightarrow \infty}\underline \varphi(\sG(\cZ), M)=\limsup_{M\rightarrow \infty}\overline\varphi(\sG(\cZ), M)=\varphi (0, \sG(\cZ)).$$
\end{defn}

\begin{remark}
It is important to distinguish the difference between translation/scale invariance of graph functional $\sG$ and the translation/scale invariance a functional $\varphi(\cdot, \sG(S))$ defined on the graph, and how it fits into the notation defined above. Throughout the paper, all graphs considered will be translation and scale invariant. However, at times we will consider functionals on these graphs which might not be scale invariant, in which case $\varphi(x, \sG(a S)) \ne \varphi(x, \sG(S))$ (even though $\sG(aS)=\sG(S)$).
For example, if $\varphi(0, \cP_\lambda)=\sum_{y \in \cP_\lambda} ||y||^r \bm 1\{(0, y) \in E(\sG(\cP_\lambda)\}=\left(\frac{1}{\lambda}\right)^{\frac{r}{d}} \sum_{y \in \cP_1} ||y||^r \bm 1\{(0, y) \in E(\sG(\cP_1)\}=\left(\frac{1}{\lambda}\right)^{\frac{r}{d}} \varphi(0, \cP_1)$, using 
$\cP_\lambda =\lambda^{-1/d}  \cP_1$.
\end{remark}

Hereafter, for a stabilizing function $\varphi$, define the rescaled functional $$\varphi_N(x, \sG(S))=\varphi(0, \sG(N^{\frac{1}{d}}(S-x))).$$ 
It follows from \cite[Lemma 3.2]{py} that, given a density $\kappa$ in $\R^d$, under appropriate moment  conditions, $\varphi_N(z, \sG(\cP_{N\kappa}))$\footnote{Recall that for a density $\kappa$ in $\R^d$, $\cP_{N\kappa}$ denotes the inhomogeneous Poisson process in $\R^d$ of rate $N\kappa(\cdot)$.} converges to $\varphi(0, \sG(\cP_{\kappa(z)}))$, if $z$ is a Lebesgue point of $\kappa$ (see definition in \eqref{eq:phi_point} below). 
 The proof of \cite[Lemma 3.2]{py} can be easily modified to show that the same holds for any sequence of densities $\kappa_N\rightarrow \kappa$ uniformly, which is summarized below: 

\begin{lem}\label{poissonstabilize} Let $\sG$ be a translation and scale invariant graph functional in $\R^d$, and $\phi_N, \phi$ as in~\eqref{phi}. Suppose $\varphi$ is translation invariant and almost surely stabilizing on $\sG(\cP_\lambda)$, with limit $\varphi(0, \sG(\cP_{\lambda}))$ for all $\lambda\in (0, \infty)$, and for some $\beta>1$
\begin{equation}\label{funcmoment}
\sup_{N \in \N}\sup_{\substack{z \in \R^d,\\ \cA\subset \R^d}}\E\left\{\varphi_N(z, \sG(\cP_{N\phi_N}\cup \cA))^\beta\right\}< \infty,
\end{equation}
where the set $\cA$ ranges over all finite subsets of $\R^d$. 
\begin{enumerate}[(a)]
\item Then, if $z$ is a Lebesgue point of $\kappa$, as $N\rightarrow \infty$,
\begin{eqnarray}\label{eq:mb1}
\varphi_N(z, \sG( \cP_{N\phi_N}))\rightarrow  \varphi(0, \sG(\cP_{\phi(z)})),
\end{eqnarray}
in expectation and in distribution.
\item For $y\in \R^d$, as $N\rightarrow \infty$,
\begin{eqnarray}\label{eq:mb2}
\varphi_N(z+N^{-\frac{1}{d}}y, \sG( \cP_{N\phi_N}))\rightarrow  \varphi(0, \sG(\cP_{\phi(z)})),
\end{eqnarray}
in expectation and in distribution. \hfill $\Box$
\end{enumerate}
\end{lem}

One aspect of the spatial independence of the Poisson process, which we will use in our proofs, is its Palm theory, which says, if conditioned to have points at particular locations, the distribution of the Poisson points elsewhere is unchanged. 

\begin{lem} [Palm theory for Poisson process {\cite[Theorem 1.6]{penrosebook}}] Let  $\lambda>0$ and $f$ be a density function (with respect to the Lebesgue measure) in $\R^d$ with support $\cK$. Suppose $s \geq 1$ is a positive integer and $h(\cY, \cX)$ is a bounded measurable function defined on all pairs of the form $(\cY, \cX)$, where $\cX$ is a finite subset of $\R^d$ and $\cY$ is a subset of $\cX$, satisfying $h(\cY, \cX)=0$, except when $\cY$ has $s$ elements. Then 
$$\E \left(\sum_{\cY \subseteq \cP_{\lambda f} } h(\cY, \cP_{\lambda f}) \right) = \frac{\lambda^s}{s!} \int_{\cK^s} \E \left(h(\{z_1, \ldots, z_s\}, \{z_1 \ldots, z_s\} \bigcup \cP_{\lambda f})  \right) \prod_{i=1}^s f(z_i) \mathrm dz_i$$
where the sum on the LHS is over all subsets $\cY$ of the random point set $\cP_{\lambda f}$.
\end{lem}

\subsection{Technical lemmas}

\label{sec:techlemmas}

In this section a few technical lemmas required for deriving the limit of the conditional variance in Theorem~\ref{TH:CLT_R1} are proved. Begin with a few definitions: For $A\subset \R^d$, denote by $|A|$ the Lebesgue measure of the set $A$.\footnote{The notation $|S|$ is also used to denote the cardinality of a finite set $S$, depending on the context.} A point $x\in \R^d$ is a {\it Lebesgue point} of $\phi$ if 
\begin{equation}
\lim_{\varepsilon\rightarrow 0}\frac{1}{|B(x, \varepsilon)|}\int_{B(x, \varepsilon)}|\phi(y)-\phi(x)|\mathrm dy=0,
\label{eq:phi_point}
\end{equation}
where $B(x, \varepsilon)$ is the Euclidean ball in $\R^d$ with center at $x$ and radius $\varepsilon$. Almost every point $x\in \R^d$ is a Lebesgue point of $\phi$ \cite[Theorem 7.7]{rudin}.


Let $\phi_N, \phi$ as in~\eqref{phi}, and $h: \R^d\times \R^d \rightarrow [0, 1]$ a symmetric and jointly measurable function, such that for almost every $x\in \R^d$, $h(x, \cdot )$ is measurable and $x$ a Lebesgue point of the function $\phi(\cdot)h(x, \cdot)$. Define 
\begin{align}\label{kappadefn}
\kappa_N(z)=\sum_{w\in \cP_{N\phi_N}} h(z, w) \bm 1\{(z, w)\in E(\sG(\cP_{N\phi_N}))\}.
\end{align}

\begin{lem}\label{lem:hzz} Let $\sG$ be a translation and scale invariant graph functional in $\R^d$ which satisfies the $\beta$-degree moment condition~\eqref{degmoment} for some $\beta>2$.  Then for $h: \R^d\times \R^d \rightarrow [0, 1]$  as above 
\begin{equation}\label{kappaz}
\lim_{N\rightarrow \infty}\E\kappa_N(z) = h(z, z) \E \Delta_0^\uparrow,
\end{equation}
if $z$ is a Lebesgue point of $\phi$ and $h(z, \cdot)\phi$, and where $\Delta_0^\uparrow$ is defined in \eqref{Delta0}. Moreover, as $N\rightarrow \infty$, 
\begin{equation}\label{hzz}
\frac{1}{N}\sum_{z\in \cP_{N\phi_N}}\kappa_N(z) \stackrel{L^2}\rightarrow \E \Delta_0^\uparrow \int_{\R^d} h(z, z)\phi(z)\mathrm d z.
\end{equation}
\end{lem}

The lemma is proved below in Section~\ref{PFHZZ}. The same proof shows that 
\begin{equation}\label{hzzupdown}
\frac{1}{N}\sum_{z\in \cP_{N\phi_N}}\kappa_N^+(z) \stackrel{L^2}\rightarrow \E \Delta_0^+ \int_{\R^d} h(z, z)\phi(z)\mathrm d z,
\end{equation}
where $\kappa^+_N(z)=\sum_{w\in \cP_{N\phi_N}} h(z, w) \bm 1\{(z, w), (z, w)\in E(\sG(\cP_{N\phi_N}))\}$ and $\Delta_0^+$ is as defined in~\eqref{Delta0}.

Next, define $\omega^\uparrow, \omega^\downarrow: \R^d\times \R^d\times \R^d\rightarrow [0, 1]$ as follows:
\begin{align}\label{rxyzup}
\omega^\uparrow(x, y, z)=\frac{pq^2 f(x)g(y)g(z)}{(pf(x)+qg(x))(pf(y)+qg(y))(pf(z)+qg(z))},
\end{align}
and
\begin{align}\label{rxyzdown}
\omega^\downarrow(x, y, z)=\frac{p^2q g(x)f(y)f(z)}{(pf(x)+qg(x))(pf(y)+qg(y))(pf(z)+qg(z))}.
\end{align}
Let $$\tau_N^\uparrow(z):=\frac{1}{2}\sum_{w_1\ne w_2\in \cP_{N\phi_N}} \omega^\uparrow(z, w_1, w_2) \bm 1\{(z, w_1), (z, w_2) \in E(\sG(\cP_{N\phi_N}))\},$$
and
$$\tau_N^\downarrow(z):=\frac{1}{2} \sum_{w_1\ne w_2\in \cP_{N\phi_N}} \omega^\downarrow(z, w_1, w_2) \bm 1\{(w_1, z), (w_2, z) \in E(\sG(\cP_{N\phi_N}))\}.$$
where $\phi_N$ is defined in~\eqref{phi}.


\begin{lem}\label{lem:hzzz} Let $\sG$ be an translation and scale invariant graph functional in $\R^d$ which satisfies the $\beta$-degree moment condition~\eqref{degmoment} for some $\beta>4$. Then 
\begin{eqnarray}\label{hzzz}
\frac{1}{N}\sum_{z\in \cP_{N\phi_N}} \tau_N^\uparrow(h, z) &\stackrel{L^2}\rightarrow& \E T_2^\uparrow \int_{\R^d} \omega^\uparrow(z, z, z)\phi (z)\mathrm d z,
\end{eqnarray}
for $\omega^\uparrow$ as in~\eqref{rxyzup}. The same result holds for $\tau_N^\downarrow(h, z)$, with $\E T_2^\uparrow $ replaced by $\E T_2^\downarrow $, and $\omega^\uparrow$ replaced by $\omega^\downarrow$.
\end{lem}

The proof of Lemma~\ref{lem:hzz} is given in Section~\ref{PFHZZ}. The proof of Lemma~\ref{lem:hzzz} is described in Section~\ref{pfwzzz}.


\subsubsection{Proof of Lemma \ref{lem:hzz}}
\label{PFHZZ}

The proof of Lemma \ref{lem:hzz} is organized as follows: Begin with the proof of~\eqref{kappaz} below. This together with uniform integrability, which follows from the degree moment condition~\eqref{degmoment}, implies the convergence in expectation in \eqref{hzz}. Following this the convergence of (\ref{hzz}) in $L^2$ is shown, by computing the limit of the second moment.\\ 

\noindent \textbf{\textit{Proof of~\eqref{kappaz} and Convergence in Expectation}}: Fix $K>0$. By the Palm theory of Poisson processes~\cite[Theorem 1.6]{penrosebook}, 
\begin{eqnarray}\label{hzzI}
& & \E \sum_{w\in \cP_{N\phi_N}}|h(z, w)-h(z, z)|\bm1\{w\in B(z, KN^{-\frac{1}{d}})\}\nonumber\\
&=&N\int_{B(z, KN^{-\frac{1}{d}})} |h(z, w)-h(z, z)|\phi_N(w)\mathrm dw,
\end{eqnarray}
which tends to zero as $N\rightarrow \infty$, if $z$ is a Lebesgue point of both $\phi$ and $h(z, \cdot)\phi(\cdot)$ (using $\phi_N \rightarrow \phi$ uniformly and $|h| \leq 1$). Since $h$ has range $[0, 1]$, this implies that
\begin{align}\label{explimitI}
 & \limsup_{N\rightarrow \infty}\E \sum_{w\in \cP_{N\phi_N}}|h(z, w)-h(z, z)|\bm1\{(z, w)\in \sG(\cP_{N\phi_N})\} \nonumber\\
\leq & \limsup_{N\rightarrow \infty} \E d_K^\uparrow(z, \sG(\cP_{N\phi_N})),
\end{align}
where $d_K^\uparrow(z, \sG(\cP_{N\phi_N})$ is the number of edges $(z, w)\in E(\sG(\cP_{N\phi_N}))$ incident on $z$ such that $w\notin B(z, KN^{-\frac{1}{d}})$.

Since $\sG$ is stabilizes on $\cP_{\lambda}$ for all $\lambda\in (0, \infty)$ (as in Definition~\ref{stabilization}), the functions $d^\uparrow$ and $d_K^\uparrow$ defined on $\sG$ stabilizes on $\cP_\lambda$, for all $\lambda\in (0, \infty)$ (as in Definition~\ref{defn:stabilize}). Therefore, by Lemma \ref{poissonstabilize}, 
\begin{equation}\label{explimitII}
\limsup_{K\rightarrow \infty} \limsup_{N\rightarrow \infty} \E d_K^\uparrow(z, \sG(\cP_{N\phi_N})=\limsup_{K\rightarrow \infty}  \E d_K^\uparrow(0, \sG(\cP_{\phi(z)})=0.
\end{equation}
Now, recall the definition of $\kappa_N(\cdot)$ from~\eqref{kappadefn}. Then from (\ref{explimitI}) and (\ref{explimitII}), it follows that for every $z$ which is a Lebesgue point of $\phi$ and $h(z, \cdot)\phi$,  
\begin{align}
\lim_{N\rightarrow \infty}\E \kappa_N(z)=&h(z, z)\lim_{N\rightarrow \infty} \E d^\uparrow(z, \sG(\cP_{N\phi_N}))\nonumber\\
=&h(z, z)\E d^\uparrow(0, \sG(\cP_{\phi(z)})) \tag*{(by Lemma~\ref{poissonstabilize})}\nonumber\\
\label{explimitIII}=&h(z, z)\E d^\uparrow(0, \sG(\cP_{1})),
\end{align}
where the last equality uses $\sG$ is scale invariant. Therefore, as $N\rightarrow \infty$
\begin{align}\label{exp1limit}
\E\left(\frac{1}{N}\sum_{z \in \cP_{N\phi_N}} \kappa_N(z)\right)=\int \phi_N(z) \E\kappa_N(z) \rightarrow&\E \Delta_0^\uparrow \int \phi(z) h(z, z)\mathrm d z\nonumber\\
:=&\mu(\sG, h),
\end{align}
by  (\ref{explimitIII}) and the Dominated Convergence Theorem. \\

\noindent \textbf{\textit{Proof of~\eqref{hzz} (Convergence in $L^2$)}}: By the Palm theory of Poisson processes,  
\begin{align}\label{exp2}
\E& \left(\frac{1}{N} \sum_{z\in \cP_{N\phi_N}} \kappa_N(z)\right)^2 \nonumber \\ 
& = \frac{1}{N} \int \phi_N(z)\E\kappa_N^2(z)\mathrm d z+\int \phi_N(z_1)\phi_N(z_2) \E \kappa_N(z_1)\kappa_N(z_2) \mathrm d z_1\mathrm d z_2.
\end{align}

Now, since $h$ is bounded in $[0, 1]$, $\E\kappa_N(z)^2\leq \E \left(d^\uparrow(z, \cP_{N\phi_N})\right)^2\rightarrow \E\left(d^\uparrow(z,  \cP_{\phi(z)})\right)^2$, as (\ref{degmoment}) holds for $\beta> 2$. Thus,  $\int \phi_N(z)\E\kappa_N^2(z)\mathrm d z= O(1)$, and the first term in (\ref{exp2}) goes to 0 as $N\rightarrow \infty$. Therefore, it suffices to consider the second term. Fix $K>0$ and let $z_1$ and $z_2$ be Lebesgue points of $\phi$. Define $A(z_1, z_2):=B(z_1, KN^{-\frac{1}{d}})\times B(z_2, KN^{-\frac{1}{d}}) \subset \R^d \times \R^d$. Then by triangle inequality, for almost all $z_1, z_2$
\begin{eqnarray}\label{exp2t1}
N^2\int_{A(z_1, z_2)} |\phi(w_1)\phi(w_2) - \phi(z_1)\phi(z_2)|\mathrm dw_1\mathrm dw_2\rightarrow 0,
\end{eqnarray}
as $N\rightarrow \infty$. 

Similarly, if $z_1, z_2$ are Lebesgue points of $h(z_1, \cdot)\phi(\cdot)$ and $h(z_2, \cdot)\phi(\cdot)$, respectively, then as $N\rightarrow \infty$
\begin{eqnarray}
& &N^2\int_{A(z_1, z_2)} |h(z_1, w_1)h(z_2, w_2)\phi(w_1)\phi(w_2) - h(z_1, z_1)h(z_2, z_2) \phi(z_1)\phi(z_2)|\mathrm dw_1\mathrm dw_2\nonumber\\
&\leq & N^2\int_{A(z_1, z_2)} \phi(w_2) h(z_2, w_2) |h(z_1, w_1)\phi(w_1) - h(z_1, z_1) \phi(z_1)|\mathrm dw_1\mathrm dw_2\nonumber\\
&+& N^2\int_{A(z_1, z_2)} h(z_1, z_1)  \phi(z_1) |h(z_2, w_2) \phi(w_2) - h(z_2, z_2) \phi(z_2)|\mathrm dw_1\mathrm dw_2\nonumber\\
&\rightarrow & 0.
\label{exp2t2}
\end{eqnarray}

Now, let $S_{w_1, w_2}=\{w_1, w_2\in \cP_{N\phi_N}: (z_1, w_1), (z_2, w_2)\in E(\sG(\cP_{N\phi_N})) \}$. Then, since $h$ has range $[0, 1]$, (\ref{exp2t1}) and (\ref{exp2t2}) gives 
\begin{eqnarray}\label{exp21}
 && \limsup_{N\rightarrow \infty}\E \sum_{S_{w_1, w_2}}|h(z_1, w_1)h(z_2, w_2)-h(z_1, z_1)h(z_2, z_2)| \nonumber\\
&\lesssim & \limsup_{N\rightarrow \infty} T_{N, K}(z_1)+  \limsup_{N\rightarrow \infty} T_{N, K}(z_2).
\end{eqnarray}
where  
\begin{align}\label{exp211}
\limsup_{N\rightarrow \infty} T_{N, K}(z_1)=&\lim_{N\rightarrow \infty}\E d_K^\uparrow(z_1, \sG(\cP_{N\phi_N})d^\uparrow(z_2, \sG(\cP_{N\phi_N})\nonumber\\
\leq & \left(\E\{d_K^\uparrow(0, \sG(\cP_{\phi(z_1)})\}^2 \E \{d^\uparrow(0, \sG(\cP_{\phi(z_2)}\}^2 \right)^{\frac{1}{2}},
\end{align}
by Lemma~\ref{poissonstabilize}, since (\ref{degmoment}) holds for $\beta> 2$. 

Similarly,
\begin{align}\label{exp212}
\limsup_{N\rightarrow \infty} T_{N, K}(z_2)=&\lim_{N\rightarrow \infty}\E d^\uparrow(z_1, \sG(\cP_{N\phi_N})d_K^\uparrow(z_2, \sG(\cP_{N\phi_N}) \nonumber\\
\leq & \left(\E\{d^\uparrow(0, \sG(\cP_{\phi(z_1)})\}^2 \E \{d_K^\uparrow(0, \sG(\cP_{\phi(z_2)}\}^2 \right)^{\frac{1}{2}}.
\end{align}
Combining~\eqref{exp211} and~\eqref{exp212} and taking $K\rightarrow \infty$ it follows that the LHS of~(\ref{exp21}) goes to zero. Therefore, 
\begin{align}\label{exp22}
 & \lim_{N\rightarrow \infty} \E\kappa_N(z_1)\kappa_N(z_2)\nonumber\\
=&\lim_{N\rightarrow \infty} \sum_{w_1, w_2\in \cP_{N\phi_N}} h(z_1, w_1)h(z_2, w_2) \bm 1\{(z_1, w_1), (z_2, w_2)\in E(\sG(\cP_{N\phi_N}))\nonumber\\
=&h(z_1, z_1)h(z_2, z_2) \lim_{N\rightarrow \infty}\E d^\uparrow(z_1, \sG(\cP_{N\phi_N}) d^\uparrow(z_2, \sG(\cP_{N\phi_N})
\end{align}
for $z_1, z_2$ Lebesgue points of $\phi$ and $h(z_1, \cdot)\phi$ and $h(z_2, \cdot)\phi$, respectively. Now, by a modification of the coupling argument used in Lemma \ref{poissonstabilize}, similar to the proof of \cite[Lemma 3.1]{py}, it can be shown that 
\begin{align}\label{exp23}
\lim_{N\rightarrow \infty}\E d^\uparrow(z_1, \sG(\cP_{N\phi_N}) d^\uparrow(z_2, \sG(\cP_{N\phi_N})=&\E d^\uparrow(0, \sG(\cP_{\phi(z_1)}) \E d^\uparrow(0, \sG(\cP_{\phi(z_2)})\nonumber\\
=& \{ \E d^\uparrow(0, \sG(\cP_1)\}^2,
\end{align}
where the last step uses $\sG$ is scale invariant. Combining (\ref{exp23}) with (\ref{exp22}) gives  
$$\lim_{N\rightarrow \infty} \E\kappa_N(z_1)\kappa_N(z_2)=\{ \E d ^\uparrow(0, \sG(\cP_1)\}^2 h(z_1, z_1)h(z_2, z_2). $$ Thus, taking limit as $N\rightarrow\infty$ in (\ref{exp2}) gives
\begin{eqnarray}\label{exp2limit}
\lim_{N\rightarrow}\E\left(\frac{1}{N} \sum_{z\in \cP_{N\phi_N}} \kappa_N(z)\right)^2 &=& \mu(\sG, h)^2,
\end{eqnarray}
where $\mu(\sG, h)$ is defined in (\ref{exp1limit}). Combining (\ref{exp1limit}) and (\ref{exp2limit}) gives the $L^2$ convergence in (\ref{hzz}). \\



\noindent \textbf{\textit{Proof of~\eqref{hzzupdown}}}: It follows from~\eqref{hzzI},~\eqref{explimitI}, and~\eqref{explimitII} that if $z$ is a Lebesgue point of both $\phi$ and $h(z, \cdot)\phi(\cdot)$, then
\begin{align*}
\E \sum_{w\in \cP_{N\phi_N}}|h(z, w)-h(z, z)|\bm1\{(z, w), (w, z)\in \sG(\cP_{N\phi_N})\}\rightarrow 0.
\end{align*}
Lemma~\ref{poissonstabilize} then implies that $\E\kappa_N^+(z)\rightarrow \E\Delta^+_0h(z,z)$, where $\Delta^+_0$ is as defined in~\eqref{Delta0}. This shows convergence in expectation. The $L^2$ convergence is  similar to the proof of~\eqref{hzz}.

\subsubsection{Proof of Lemma \ref{lem:hzzz}}
\label{pfwzzz}

The proof is very similar to Lemma \ref{lem:hzz}. Without loss of generality consider the function $\omega^\uparrow$   defined in~\eqref{rxyzup} (the proof for $\omega^\downarrow$ is identical).

Fix $K>0$. Define $B^2(z):=B(z, KN^{-\frac{1}{d}})\times B(z, KN^{-\frac{1}{d}})$, then by Palm theory, 
\begin{eqnarray}
\label{exphzzz}
&&\E \sum_{w_1\ne w_2\in \cP_{N\phi_N}} \left|\omega^\uparrow(z, w_1, w_2)-\omega^\uparrow(z, z, z)\right| \bm 1\{w_1, w_2 \in  B(z, KN^{-\frac{1}{d}})\}\nonumber\\
&=& N^2 \int_{B^2(z)} \left|\omega^\uparrow(z, w_1, w_2)-\omega^\uparrow(z, z, z)\right|\phi(w_1) \phi(w_2)\mathrm dw_1\mathrm dw_2 +o(1), 
\end{eqnarray}
since $\phi_N\rightarrow\phi$ uniformly.

Note that $$\omega^\uparrow(z, w_1, w_2)\phi(w_1)\phi(w_2)=\frac{pq^2f(z)g(w_1)g(w_2)}{pf(z)+qg(z)}.$$ Therefore, if $z$ is a Lebesgue point of both $f$ and $g$, then $(z, z)$ is the Lebesgue point of $\omega^{\uparrow}(z, \cdot, \cdot) \phi(\cdot) \phi(\cdot)$, and as $N\rightarrow \infty$, 
\begin{align}\label{exphzzz1}
& N^2\int_{B^2(z)} |\omega^\uparrow(z, w_1, w_2)\phi(w_1)\phi(w_2)-\omega^\uparrow(z, z, z)\phi^2(z)|\mathrm dw_1\mathrm dw_2\rightarrow 0.
\end{align}
Moreover, 
\begin{eqnarray}
\phi(z)\lim_{N\rightarrow \infty} N^2\int_{B^2(z)}|\phi(w_2)-\phi(z)|\mathrm dw_1\mathrm dw_2=0.
\label{phi21}
\end{eqnarray}
and by the Lebesgue differentiation theorem
\begin{eqnarray}
\lim_{N\rightarrow \infty}N^2\int_{B(z, KN^{-\frac{1}{d}})} |\phi(w_1)-\phi(z)|\mathrm dy_1\int_{B(z, KN^{-\frac{1}{d}})} \phi(w_2)\mathrm dw_2=0.
\label{phi22}
\end{eqnarray}
Combining (\ref{phi21}) and (\ref{phi22}) gives
\begin{equation}
\lim_{N\rightarrow \infty} N^2 \int_{B^2(z)} \omega^\uparrow(z, z, z)|\phi_N(w_1)\phi(w_2)-\phi^2(z)|\mathrm dw_1\mathrm dw_2\rightarrow 0.
\label{exphzzz2}
\end{equation}
The triangle inequality combined with (\ref{exphzzz1}) and (\ref{exphzzz2}) implies that the RHS of (\ref{exphzzz}) goes to 0 as $N\rightarrow \infty$. This implies that
\begin{align}\label{hzzzI}
\E\sum_{w_1\ne w_2\in \cP_{N\phi_N} }& \left|\omega^\uparrow(z, w_1, w_2)-\omega^\uparrow(z, z, z)\right|\bm 1\{(z, w_1), (z, w_1) \in E(\sG(\cP_{N\phi_N}))\} \nonumber\\
& \leq  2 \E T_{2, K}^\uparrow(z, \sG(\cP_{N\phi_N}))+o(1),
\end{align}
where 
\begin{align*}
T_{2, K}^\uparrow(& z, \sG(\cP_{N\phi_N}))\\
=&\sum_{w\in \cP_{N\phi_N}}\bm  1\{(z, w_1), (z, w_2)\in E(\sG(\cP_{N\phi_N}))\} \bm 1\{w_1\text{ or } w_2 \notin B(z, KN^{-\frac{1}{d}})\}.
\end{align*} 
Now, since $\sG$ stabilizes on $\cP_{\lambda}$ for all $\lambda\in (0, \infty)$, the function $T_{2}^\uparrow$ and hence $T_{2, K}^\uparrow$ stabilizes on $\cP_\lambda$, for all $\lambda\in (0, \infty)$. Moreover, $T_{2, K}^\uparrow(z, \sG(\cP_{N\phi_N}))$ satisfies the bounded moment condition (\ref{funcmoment}) for $\beta > 1$, since $d^\uparrow(z, \sG(\cP_{N\phi_N})$ satisfies (\ref{degmoment}) for some $\beta> 2$. Therefore, by Lemma \ref{poissonstabilize}, 
\begin{equation}
\limsup_{K\rightarrow \infty} \limsup_{N\rightarrow \infty} \E T_{2, K}^\uparrow(z, \sG(\cP_{N\phi_N})=0.
\label{hzzzII}
\end{equation}

From (\ref{hzzzI}) and (\ref{hzzzII}) it follows that for every $z$ which is a Lebesgue point of both $f$ and $g$,
\begin{eqnarray}
\lim_{N\rightarrow \infty}\E \tau_N^\uparrow(z)&=&\omega^\uparrow(z, z, z)\lim_{N\rightarrow \infty} \E 2 T_2^\uparrow(z, \sG(\cP_{N\phi_N}))\nonumber\\
&=&2\omega^\uparrow(z, z, z)\E T_2^\uparrow(0, \sG(\cP_{\phi(z)}))\label{hzzzIIIa}\\
&=&2 \omega^\uparrow(z, z, z)\E T_2^\uparrow,
\label{hzzzIIIb}
\end{eqnarray}
where (\ref{hzzzIIIa}) uses Lemma \ref{poissonstabilize} and (\ref{hzzzIIIb}) uses $\sG$ is scale invariant. Therefore, by the  Dominated Convergence Theorem,
\begin{eqnarray}
\lim_{N\rightarrow \infty} \E\left(\frac{1}{N}\sum_{z\in \cP_{N\phi_N} } \tau_N^\uparrow(z)\right)&=&\lim_{N\rightarrow \infty} \int \phi_N(z) \E\tau_N^\uparrow(z)\nonumber \\
&\rightarrow&\E T_2^\uparrow \int \phi(z) \omega^\uparrow(z, z, z)\mathrm d z.
\label{hzzzexp}
\end{eqnarray}

Now, as in Lemma~\ref{lem:hzz}, it can be shown that the second moment of the quantity $\frac{1}{N}\sum_{z\in \cP_{N\phi_N} } \tau_N^\uparrow(z)$ converges to $\left(\E T_2^\uparrow \int \phi(z) \omega^\uparrow(z, z, z)\mathrm d z\right)^2$. This proves the $L^2$ convergence of (\ref{hzzz}).

%

\subsection{Proof of Proposition~\ref{CONSISTENT}}

\label{sec:pfconsistency}

Recall from \eqref{Tpoisson} the definition of $T(\sG(\cZ_N'))$. Define 
$$h_N(x, y)=\frac{N_1N_2 f(x)g(y)}{(N_1f(x)+N_2g(x))(N_1f(y)+N_2g(y))}.$$ 
Note that $h_N(x, y)\rightarrow h(x, y)=\frac{pq f(x)g(y)}{(pf(x)+qg(x))(pf(y)+qg(y))}$, 
uniformly in $x, y\in \R^d$ as $N\rightarrow \infty$. Then, with $\phi_N$ and $\phi$ as in~\eqref{phi},  
\begin{align}
\frac{1}{N}\E(T(\sG(\cZ_N'))|\cF)&=\frac{1}{N}\sum_{1\leq i\ne j\leq L_{N}}h_{N}(Z_i, Z_j)\bm1\{(Z_i, Z_j)\in E(\sG(\cZ_N'))\} \tag*{(by \eqref{pooledlabelalternative})}\nonumber\\
&=\frac{1}{N}\sum_{1\leq i\ne j\leq L_{N}}h(Z_i, Z_j)\bm1\{(Z_i, Z_j)\in E(\sG(\cZ_N'))\}+o(1)\nonumber\\
&\stackrel{D}=\frac{1}{N}\sum_{z, w\in \cP_{N\phi_N}}h(z, w)\bm1\{(z, w)\in E(\sG(\cZ_N'))\}+o(1)\nonumber\\
\label{exp} &\stackrel{L^2}\rightarrow\E \Delta_0^\uparrow \int h(z,z)\phi(z)\mathrm d z=\frac{\E \Delta_0^\uparrow}{2} (1-\delta(f, g, p)),
\end{align}
where the last step uses Lemma~\ref{hzz} (recall that $\Delta_0^\uparrow=d(0, \sG(\cP_1))$, and $\delta(f, g, p))$ is defined in~\eqref{hpdis}). 

The limit in~\eqref{exp} shows convergence in expectation. To show convergence in probability, compute $$\Var(T(\sG(\cZ_N')))=\E(\Var(T(\sG(\cZ_N'))|\cF))+\Var(\E(T(\sG(\cZ_N'))|\cF)).$$ From~\eqref{exp} it follows that $\Var(\frac{1}{N}\E(T(\sG(\cZ_N'))|\cF))\rightarrow 0$. Moreover, by Lemma~\ref{var_R} below, $$\frac{1}{N} \E(\Var(T(\sG(\cZ_N'))|\cF))= \E(\Var(\cR_1(\sG(\cZ_N'))|\cF))= 
\kappa_{\sG}^2.$$ This implies $\frac{1}{N^2}\Var(T(\sG(\cZ_N')))\rightarrow 0$, completing the proof of~\eqref{eq:weaklimit}.



\subsection{Proof of Theorem~\ref{TH:CLT_R1}}

\label{sec:pfcltR1}

This section is organized as follows: Section~\ref{pfconditionalvar} below derives the limit of the conditional variance \eqref{sigma1}, using results proved in Section~\ref{sec:techlemmas}. The proof of the CLT is given in Section~\ref{pfcondclt}. 

\subsubsection{Limiting conditional variance}
\label{pfconditionalvar}

The limit of the conditional variance $\Var(\cR_1(\sG(\cZ_N'))|\cF)$ can be computed in terms of the graph functional $\sG$ and the unknown densities $f$ and $g$.

\begin{lem}\label{var_R} Let $\sG$ as in Theorem~\ref{TH:CLT_R1} and $\mathcal R_1(\sG(\cZ_N'))$ as in (\ref{R1}). Then, 
\begin{eqnarray}\label{lmsigma1}
\Var(\cR_1(\sG(\cZ_N'))|\cF)\stackrel{L^2}\rightarrow  
\kappa_{\sG}^2,
\end{eqnarray}
where $ 
\kappa_{\sG}$ is defined in \eqref{sigma1}.
\end{lem}

\begin{proof} Define the function $h_N: \R^d\times \mathbb R^d \rightarrow [0, 1]$ as 
\begin{equation}\label{hnxy}
h_N(x, y)=\frac{N_1N_2 f(x)g(y)}{(N_1f(x)+N_2g(x))(N_1f(y)+N_2g(y))}.
\end{equation}
By construction $\cZ_N'$ is a Poisson process in $\R^d$ with intensity function $N\phi_N$, where $\phi_N\rightarrow \phi$ uniformly as $N\rightarrow \infty$ (as defined in~\eqref{phi}). Recall that $\psi(c_x, c_y)=\bm 1\{c_x=1, c_y=2\}$, which implies $\P(\psi(c_x, c_y)|\cF)=h_N(x, y)$. Moreover, 
\begin{equation*}
h_N(x, y)\rightarrow h(x, y)=\frac{pq f(x)g(y)}{(pf(x)+qg(x))(pf(y)+qg(y))},
\end{equation*}
uniformly in $x, y\in \R^d$ as $N\rightarrow \infty$.

Now, let 
\begin{eqnarray}
V_{x, y}:=(\psi(c_x, c_y)-h_{N}(x, y) )\bm1\{(x, y)\in E(\sG(\cZ_N')) \}.
\label{v}
\end{eqnarray}
Then by (\ref{R1}), $\cR_1(\sG(\cZ_N')):=\frac{1}{\sqrt N}\sum_{x, y \in \cZ_N'} V_{x, y}$. Therefore, 
\begin{align}\label{var_Q}
\Var& (\cR_1(\sG(\cZ_N'))|\cF) \nonumber \\ 
=&\frac{1}{N} \sum_{x \ne y \in \cZ_N'}\Var(V_{x, y}|\cF) + \frac{2}{N}\left\{\sum_{x \in \cZ_N'} \sum_{\substack{y, z \in \cZ_N' \\ y < z}} \Cov(V_{x, y}, V_{x, z}|\cF) + \sum_{x \in \cZ_N'} \sum_{\substack{y, z \in \cZ_N' \\ y < z}} \Cov(V_{y, x}, V_{z, x}|\cF) \right\} \nonumber \\ 
& \;\;\;\;\;\;\;\;\;\;\;\;\;\;\;\;\;\;\;\;\;\;\;\;\;\; +\frac{2}{N}\left\{ \sum_{x \in \cZ_N'} \sum_{\substack{y, z \in \cZ_N' \\ y < z}}  \Cov(V_{x, y}, V_{z, x}|\cF) + \sum_{x < y \in \cZ_N'} \Cov(V_{x, y}, V_{y, x}|\cF) \right\}.
\end{align} 
By Lemma \ref{lem:hzz}, as $N\rightarrow \infty$,
\begin{align*}
\frac{1}{N}\sum_{x, y \in \cZ_N'} &\Var(V_{x, y}|\cF)\nonumber\\
=&\frac{1}{N} \sum_{x, y \in \cZ_N'} \Big\{h_{N}(x, y)(1-h_{N}(x, y))\Big\} \bm1\{(x, y)\in E(\sG(\cZ_N'))\}\nonumber\\
\stackrel{L_2} \rightarrow & \E \Delta^\uparrow_0 \int \Big\{h(x, x)(1-h(x, x))\Big\} \phi(x)\mathrm dx.
\end{align*}

It remains to compute the limit of the covariance term in (\ref{var_Q}). To this end, observe that   
\begin{align*}
\Cov(V_{x, y}, &V_{x, z}|\cF)\\
=&\{\omega_{N}(x, y, z)-h_N(x, y)h_N(x, z)\} \bm1\{(x, y), (x, z)\in E(\sG(\cZ_N'))\},
\end{align*}
where 
\begin{eqnarray*}
\omega_{N}^\uparrow (x, y, z)&:=&\frac{N_1N_2^2f(x)g(y)g(z)}{(N_1f(x)+N_2g(x))(N_1f(y)+N_2g(y))(N_1f(z)+N_2g(z))}\nonumber\\
&\rightarrow& \omega^\uparrow(x, y, z),
\end{eqnarray*} 
where $\omega^\uparrow(\cdot, \cdot, \cdot)$ is defined in~\eqref{rxyzup}. The convergence above is uniformly in $x, y, z\in \R^d$, as $N\rightarrow \infty$. Therefore, by Lemma \ref{lem:hzzz}, 
\begin{align}\label{var3}
\frac{1}{ N}\sum_{x \in \cZ_N'} \sum_{\substack{y, z \in \cZ_N' \\ y < z}}& \Cov(V_{x, y}, V_{x, z}|\cF)\nonumber\\
\stackrel{L_2} \rightarrow & \E T_2^\uparrow \int \left\{\omega^\uparrow(x,x,x)-h^2(x,x) \right\}\phi(x)\mathrm dx,
\end{align}
where $T_2^\uparrow=T_2^\uparrow(0, \sG(\cP_1))$ is as defined in~\eqref{2star0}. Similarly, define 
\begin{eqnarray*}
\omega_{N}^\downarrow (x, y, z)&:=&\frac{N_1^2N_2g(x)f(y)f(z)}{(N_1f(x)+N_2g(x))(N_1f(y)+N_2g(y))(N_1f(z)+N_2g(z))}\nonumber\\
&\rightarrow& \omega^\downarrow(x, y, z),
\end{eqnarray*} 
where $\omega^\downarrow(\cdot, \cdot, \cdot)$ is defined in~\eqref{rxyzdown}. Then, by Lemma \ref{lem:hzzz}, 
\begin{align}\label{var4}
\frac{1}{N} \sum_{x \in \cZ_N'} \sum_{\substack{y, z \in \cZ_N' \\ y < z}} & \Cov(V_{y, x}, V_{z, x}|\cF)\nonumber\\
\stackrel{L_2} \rightarrow & \E T_2^\downarrow \int \left\{\omega^\downarrow(x,x,x)-h^2(x,x) \right\}\phi(x)\mathrm dx,
\end{align}
where $T_2^\downarrow=T_2^\downarrow(0, \sG(\cP_1))$ is as defined in~\eqref{2star0}.
 
Similarly, it can be shown that 
\begin{align}\label{var5}
\frac{1}{N} \sum_{x \in \cZ_N'} \sum_{\substack{y, z \in \cZ_N' \\ y < z}} \Cov(V_{x, y}, V_{z, x}|\cF)\pto & -\E T_2^+ \int h^2(x,x) \phi(x)\mathrm dx,
\end{align}
where $T_2^+$ is as defined in~\eqref{2star0}.

Also, from~\eqref{hzzupdown}
\begin{align}\label{var6}
\Cov(V_{x, y}, V_{y, x}|\cF)=&-\frac{1}{N} \sum_{x < y \in \cZ_N'}  h_N(x, y)h_N(y, x) \bm1\{(x, y), (y, x)\in E(\sG(\cZ_N'))\}\nonumber\\
\stackrel{L_2} \rightarrow & - \frac{1}{2} \E \Delta^+_0 \int h(x, x)^2 \phi(x)\mathrm dx,
\end{align}
where $\Delta^+_0$ is as defined in~\eqref{Delta0}.

Finally, observe $h(x, x)=\frac{pqf(x)g(x)}{\phi^2(x)}$ and $\omega^\uparrow(x, x, x)=\frac{pq^2f(x)g(x)^2}{\phi^3(x)}$ and $\omega^\downarrow(x, x, x)=\frac{p^2qf(x)^2g(x)}{\phi^3(x)}$. Using this, and combining together (\ref{var3}), (\ref{var4}), (\ref{var5}), (\ref{var6}) with (\ref{var_Q}), the limit in~\eqref{lmsigma1} follows. \end{proof}

\subsubsection{Completing the proof of Theorem~\ref{TH:CLT_R1}}
\label{pfcondclt}

The CLT of $\cR_1(\sG(\cZ_N'))$ will be proved using Stein's method based on dependency graphs given below:  

\begin{thm}[Chen and Shao \cite{stein_local}] \label{th:steinsmethod}
Let $\{W_i, i \in \cV\}$ be random variables indexed by the vertices of a dependency graph $H=(\cV, \cE)$ with maximum degree $D$. If  
$W=\sum_{i \in \cV}W_i$ with $\E(W_i)=0$, $\E W^2 = 1$ and $\E|W_i|^3 \leq\theta^3$ for all $i\in \cV$ and for some $\theta > 0$, then 
\begin{equation}\label{eq:steinsmethod}
\sup_{z\in \R} |\P(W \leq z)-\Phi(z)|\lesssim D^{10}|\cV|\theta^3.
\end{equation}
\end{thm}

Using this, the proposition below shows that the statistic $T(\sG(\cZ_N'))$, centered by the conditional mean and scaled by the conditional variance converges to $N(0, 1)$. This, along with Lemma~\ref{var_R}, completes the proof of Theorem~\ref{TH:CLT_R1}.

\begin{ppn}\label{cltR1} Let $T(\sG(\cZ_N'))$ be as defined in~\eqref{R1} and $\Phi$ the standard normal distribution function. Then 
\begin{eqnarray}\label{clt_R}
\sup_{x \in \R}\left|\P\left(\frac{T(\sG(\cZ_N'))-\E(T(\sG(\cZ_N'))|\cF)}{\sqrt{\Var(T(\sG(\cZ_N'))|\cF)}}\leq x \Big| \cF \right)- \Phi(x)\right| \pto 0,
\end{eqnarray}
\end{ppn}

\begin{proof}Let $W:=\frac{T(\sG(\cZ_N'))-\E(T(\sG(\cZ_N'))|\cF)}{\sqrt{\Var(T(\sG(\cZ_N'))|\cF)}}$. Recall the definition of $V_{x, y}$ from (\ref{v}) and let $$U_{x , y}=\frac{V_{x,  y}}{\sqrt{\Var(T(\sG(\cZ_N'))|\cF)}}.$$ Note that $W=\sum_{x, y\in\cZ_N }U_{x, y}$, and $\E(U_{x, y}|\cF)=0$ and $\E W^2=1$.  Construct a dependency graph $H=(V(H), E(H))$ of the random variables $\{U_{x, y}, (x, y)\in E(\sG(\cZ_N'))\}$ as follows: $V(H)=\{1, 2, \ldots, L_N\}$ and $(i, j)\in E(H)$ whenever the graph distance between $Z_i$ and $Z_j$ in $\sG(\cZ_N')$ is at most 2. Let $D$ be the maximum degree of this dependency graph. It is easy to see that $D\leq 2\Delta(\sG(\cZ_N'))^2=o_P(N^{\frac{1}{20}})$, by Assumption \ref{assummaxdeg}. Moreover, for $(x, y)\in E(\sG(\cZ_N'))$
$$|U_{x, y}|^3\leq \frac{1}{(\Var(T(\sG(\cZ_N')|\cF))^{\frac{3}{2}}}.$$

Therefore, by Theorem \ref{th:steinsmethod} above, conditional on $\cF$,
\begin{eqnarray}
\left|\P\left(W\leq x|\cF\right) -\Phi(x) \right| &\lesssim & D^{10}\frac{|\cZ_N'|}{(\Var(T(\sG(\cZ_N')) |\cF))^{\frac{3}{2}}}\nonumber\\
&=&\frac{D^{10}}{\sqrt{N}}\frac{\frac{|\cZ_{N}'|}{N}}{(\Var(\cR_1(\sG(\cZ_N'))|\cF))^{\frac{3}{2}}}.
\label{ste}
\end{eqnarray}
Note that $\frac{|\cZ_N'|}{N}\pto 1$, as $|\cZ_N'|=L_N$ is a Poisson random variable with mean $N$, and by Lemma \ref{var_R}, $\Var(\cR_1(\sG(\cZ_N'))|\cF)\pto  
\kappa_{\sG}^2$. Therefore,  (\ref{ste}) goes to zero in probability, since $D^{10}=o_P(N^{\frac{1}{2}})$, by Assumption \ref{assummaxdeg}. This completes the proof of (\ref{clt_R}).
\end{proof}

\begin{remark}\label{rem:consistency} (Consistency)  Let $\sG$ be graph functional satisfying the assumptions of Theorem \ref{TH:CLT_R1}. In addition, assume $\Var(|E(\sG(\cZ_N'))|)=O(N)$, where $|E(\sG(\cZ_N'))|$ denotes the number of edges in the graph $\sG(\cZ_N')$. (This assumption is easy to verify in most cases: For example, for the MST $|E(\cT(\cZ_N'))|=L_N-1$, and for the directed $K$-NN graph $|E(\cN_K(\cZ_N'))|)=K L_N$, where $L_N \sim \dPois(N)$, hence the variance condition holds trivially.)  Now, consider the two-sample test based on $\sG$ with rejection region 
\begin{align}\label{eq:rejregion}
\frac{1}{\sqrt N}\Big\{ T(\sG(\cZ_N'))- \E_{H_0}(T(\sG(\cZ_N'))) \Big\} \leq -a_N,
\end{align}
where $a_N$  is a positive sequence going infinity such that $a_N \ll \sqrt N$, as $N \rightarrow \infty$.  Note that the proof of Theorem \ref{TH:CLT_R1} shows that  $\E(\Var_{H_0}(T(\sG(\cZ_N'))|\cF)) =O(N)$. Therefore, 
\begin{align*}
\Var_{H_0}(T(\sG(\cZ_N'))) &= \E(\Var_{H_0}(T(\sG(\cZ_N'))|\cF))+  \Var (\E_{H_0}(T(\sG(\cZ_N'))|\cF)) \nonumber \\
&=O(N) + \frac{N_1^2 N_2^2}{(N_1+N_2)^4}\Var(|E(\sG(\cZ_N'))) =O(N),
\end{align*}
by the variance assumption above. Therefore, by Chebyshev's inequality, the probability of Type I error of the test \eqref{eq:rejregion} goes to zero, as $N \rightarrow \infty$. For the Type II error, suppose $f$ and $g$ differ on a set of positive Lebesgue measure, and  choose $N$ large enough so that $$-\frac{a_N}{\sqrt N} + \frac{\E_{H_0}(T(\sG(\cZ_N')))}{N}  > \frac{\E \Delta_0^\uparrow}{2}\left(1-\delta(f, g, p) \right).$$ 
(This is possible because $\delta(f, g, p)\geq \delta(f, f, p)=p^2+q^2$ and the inequality is strict for densities $f$ and $g$ differing on a set of positive measure (see \cite[Theorem 1 and Corollary 1]{gyorfinemetz1}).) Then the probability of Type II error 
\begin{align*}
\P_{H_1}\left( \frac{T(\sG(\cZ_N'))}{N} >  -\frac{a_N}{\sqrt N} + \frac{\E_{H_0}(T(\sG(\cZ_N')))}{N}  \right) \rightarrow 0,
\end{align*}
by Proposition \ref{CONSISTENT}. This shows the universal consistency of the test \eqref{eq:rejregion}. 
\end{remark}

\subsection{Calculation of the null variance}
\label{sec:varnullKNN}

One can obtain the original de-Poissionized distribution of the 2-sample statistic \eqref{graph2test} by conditioning on the event $\{L_{N_1}=N_1, L_{N_2}=N_2\}$ (recall notations from \eqref{2poisson}). More precisely,  
\begin{align}
\cR(\sG(\sX_{N_1}' \cup \sY_{N_2}'))|\{L_{N_1}=N_1, L_{N_2}=N_2\}&\stackrel{D}=\frac{1}{\sqrt N} \Big\{T(\sG(\sX_{N_1}\cup \sY_{N_2}))-\E_{H_0}(T(\sG(\sX_{N_1}\cup \sY_{N_2}))) \Big\} \nonumber \\
&:=\cR(\sG(\cZ_N)), 
\label{eq:poisson_N1N2}
\end{align}
where $\cZ_N=\sX_{N_1}\cup \sY_{N_2}$ are i.i.d. samples of size $N_1$ and $N_2$ from $f$ and $g$, respectively, as in \eqref{2}. Therefore, one approach to de-Poissionize the CLT in Theorem \ref{TH:CLT_R1} is to derive the joint asymptotic distribution of $(\cR(\sG(\sX_{N_1}' \cup \sY_{N_2}'), L_{N_1}, L_{N_2})$, and then using \eqref{eq:poisson_N1N2} to obtain the conditional distribution. However, the calculation of the limiting variance is quite tedious for general alternatives. On the other hand, for the implementation of the test, we only require the null variance, where the calculations are much simpler. In the following, we compute the (de-Poissonized) null variance of the statistic $\cR(\sG(\cZ_N))$, for a graph functional $\sG$ satisfying the assumptions of Theorem \ref{TH:CLT_R1}. 

To this end, given a set $S$, denote by $E^+(\sG (S))$ the set of pairs of vertices in the graph $\sG(S)$ with edges in both directions (that is, the set of ordered pairs of vertices $(x, y)$ such that both
$(x, y) \in E(\sG (S))$ and $(y, x) \in E(G (S))$).  Then, under the null $H_0$, using the exchangeability of the data, we get 
\begin{align}\label{eq:varnullKNN}
\Var_{H_0}(\cR(\sG(\cZ_N)))=A_1+A_2,
\end{align}
where 
\begin{align}
\label{eq:varKNNI}A_1&= \frac{a|E(\sG(\cZ_N))| + 2c|E^+(\sG(\cZ_N))| + 2 b^{\uparrow}T_2^{\uparrow}(\sG(\cZ_N)) + 2 b^{\downarrow} T_2^{\downarrow}(\sG(\cZ_N)) + 2c T_2^{+}(\sG(\cZ_N))}{N}, \\
 A_2&= L \left( \frac{2 {|E(\sG(\cZ_N))| \choose 2} - 2 (|E^+(\sG(\cZ_N))| - T_2^{\uparrow}(\sG(\cZ_N)) -  T_2^{\uparrow}(\sG(\cZ_N)) - T_2^{+}(\sG(\cZ_N)) )}{N}\right);  
\label{eq:varKNNII} 
\end{align}
with $a=\frac{N_1N_2}{N (N-1)}-\frac{N_1^2N_2^2}{N^2 (N-1)^2}$, $c=-\frac{N_1^2N_2^2}{N^2 (N-1)^2}$, $b^{\uparrow}=\frac{N_1N_2 (N_2-1)}{N (N-1) (N-2)}-\frac{N_1^2N_2^2}{N^2 (N-1)^2}$, $b^{\downarrow}=\frac{N_1 (N_1-1)N_2}{N (N-1) (N-2)}-\frac{N_1^2N_2^2}{N^2 (N-1)^2}$, and $L=\frac{N_1 N_2 (N_1-1)(N_2-1)}{N(N-1)(N-2)(N-3)}-\frac{N_1^2N_2^2}{N^2 (N-1)^2}$.  

Now, from the proof Lemma \ref{var_R}, it follows that  
$$a \frac{|E(\sG(\cZ_N))|}{N} \pto (pq-p^2q^2) \E \Delta_0^\uparrow, \quad \frac{2c|E^+(\sG(\cZ_N))|}{N} \pto - p^2 q^2 \E \Delta_0^{+},$$
and 
$$\frac{2 b^{\downarrow} T_2^{\downarrow}(\sG(\cZ_N))}{N} \pto 2(p^2 q-p^2q^2) \E T_2^\downarrow, \quad \frac{2 b^{\uparrow} T_2^{\uparrow}(\sG(\cZ_N))}{N} \pto 2(pq^2-p^2q^2) \E T_2^\uparrow,$$
and $\frac{2c T_2^{+}(\sG(\cZ_N))}{N} \pto - 2 p^2q^2  \E T_2^+$. Substituting these limiting values in \eqref{eq:varKNNI} gives 
\begin{align}
\label{eq:limitKNNI}
A_1 \pto \frac{r}{4}\Big\{2\E\Delta_0^\uparrow + 4 \left(q\E T_2^\uparrow+ p\E T_2^\downarrow \right) -2 r\Gamma_0\Big\},
\end{align}
which, as expected, matches the expression in \eqref{sigmanulldirected} (recall definition of $\Gamma_0$ from Theorem \ref{TH:CLT_R1}). Next, note that $L \rightarrow 0$ and 
$N L \rightarrow 4 p^2q^2 - pq= -\frac{r}{2}(1-2r)$, as $N \rightarrow \infty$. 
This implies,  
$$L \frac{|E(\sG(\cZ_N))|^2- |E(\sG(\cZ_N))|}{N}  = \frac{L |E(\sG(\cZ_N))|^2}{N} + o_P(1) = -\frac{r}{2}(1-2r)(\E \Delta_0^\uparrow)^2+o_P(1),$$
using $\frac{|E(\sG(\cZ_N))|^2}{N^2} \pto (\E \Delta_0^\uparrow)^2$. Therefore,  recalling \eqref{eq:limitKNNI}, $A_2 \pto - \frac{r}{2}(1-2r) (\E \Delta_0^\uparrow)^2$. This combined with \eqref{eq:varnullKNN} and \eqref{eq:limitKNNI} gives 
\begin{align}\label{eq:varnullG}
\Var_{H_0}(\cR(\sG(\cZ_N)))= \frac{r}{4}\Big\{2\E\Delta_0^\uparrow + 4 \left(q\E T_2^\uparrow+ p\E T_2^\downarrow \right) -2 r\Gamma_0 - 2 (p-q)^2 (\E \Delta_0^\uparrow)^2 \Big\},
\end{align}
which is the limiting (de-Poissonized) variance of the two-sample test based on $\sG$. For instance, for the MST graph $\sG=\cT$, using  \eqref{sigmanullundirected} and $\E \Delta_0^\uparrow =2$, \eqref{eq:varnullG} simplifies to $$\Var_{H_0}(\cR(\cT(\cZ_N)))= r \Big\{r + \frac{1}{2}\Var (\Delta_0)(1-2r)\Big\},$$
which is the null variance of the FR-test as in \cite[Theorem 1]{henzepenrose}.

\subsection{Proofs of Proposition \ref{CLT_R2} and Theorem \ref{TH:CLT_R}}
\label{sec:pfR2R}

The proof of Proposition \ref{CLT_R2} is given in Section~\ref{pfcltR2} below.  Theorems \ref{TH:CLT_R1} and \ref{CLT_R2} can be easily combined to complete the proof of Theorem~\ref{TH:CLT_R} (see Section~\ref{pfcltR}).

\subsubsection{Proof of Proposition \ref{CLT_R2}}
\label{pfcltR2} 
Recall the definition of $h_N(x, y)$ from~\eqref{hnxy}. Note that $h_N(x, y)\rightarrow h(x, y)=\frac{pq f(x)g(y)}{(pf(x)+qg(x))(pf(y)+qg(y))}$, 
uniformly in $x, y\in \R^d$ as $N\rightarrow \infty$. It is easy to see that $$\lim_{N\rightarrow \infty}\sup_{x, y\in \R^d}N^{\frac{1}{2}}\left|\frac{h_N(x, y)}{h(x, y)}-1\right|=0,$$ since $\sqrt N (\frac{N_1}{N_1+N_2}-p)\rightarrow 0$ and $\sqrt N (\frac{N_1}{N_1+N_2}-q)\rightarrow 0$ (recall \eqref{eq:N12}). Therefore, for any $\varepsilon>0$
\begin{eqnarray*}
 & &  \frac{1}{\sqrt N} \sum_{x, y\in \cZ_N'} |h_N(x, y)-h(x, y)| \bm 1\{(x, y)\in E(\sG(\cZ_N'))\}\nonumber\\
&\leq &\frac{\varepsilon}{N} \sum_{x, y\in \cZ_N'} |h(x, y)| \bm 1\{(x, y)\in E(\sG(\cZ_N'))\} \leq \varepsilon \frac{|E(\sG(\cZ_N'))|}{N}
\end{eqnarray*}
for $N$ large enough. By stabilization, $\frac{|E(\sG(\cZ_N'))|}{N}\pto  \E \Delta_0^\uparrow $, as $N\rightarrow \infty$. Therefore, the RHS above is arbitrarily small as $N\rightarrow \infty$, and 
\begin{align}\label{R2I}
\cR_2(\sG(\cZ_N'))= \frac{1}{\sqrt N} \sum_{x, y \in \cZ_N'}(J_{x, y}-\E(J_{x, y}))+o(1),
\end{align}
where $J_{x, y}=h(x, y)\bm 1\{(x, y)\in E(\sG(\cZ_N'))\}$.

Let $w(x)=\frac{p f(x)}{p f(x)+q g(x)}$. For $x\in \R^d$, a locally finite point set $\cH$ and any Borel set $A\subseteq \R^d$, define
\begin{align}
\zeta^x_{\cH}(A):=\zeta(x, \cH, A):=&w(x) \sum_{y\in \cH^x\cap A}\bm 1\{(x, y)\in E(\sG(\cH^x))\},\nonumber\\
=& w(x)d_\sG^\uparrow(x,  \cH, A),
\end{align}
where $d_\sG^\uparrow(\cdot, \cdot, \cdot)$ is the out-degree measure defined in~\eqref{outdegmeasure}. 

Let $$\mu_{N}:=\sum_{x\in \cP_{N\phi_N}}\zeta^x_{\cP_{\phi_N}}$$ and $v(y)=\frac{q g(y)}{p f(y)+q g(y)}$. Therefore, $$\langle v, \zeta^x_{\cP_{\phi_N}}\rangle=\int v(y) \zeta^x_{\cP_{\phi_N}}(\mathrm dy)=w(x)\sum_{y\in \cP_{N\phi_N}^x} v(y)\bm 1\{(x, y)\in E(\sG(\cP_{N\phi_N}^x))\}$$ and $$\langle v, \mu_N\rangle=\sum_{x, y\in \cP_{N\phi_N}} h(x, y)\bm 1\{(x, y)\in E(\sG(\cP_{N\phi_N}))\}.$$
Thus, from~\eqref{R2I} 
\begin{align}\label{R2II}
\cR_2(\sG(\cZ_N'))=\frac{1}{\sqrt N}\left(\langle v, \mu_N\rangle-\E \langle v, \mu_N\rangle \right)+o(1).
\end{align}

\noindent \textit{Proof of~\eqref{varcond}}: By \cite[Lemma 4.1]{penroseclt}
\begin{eqnarray}\label{anbn}
\frac{1}{N}\Var\E(\cR(\sG(\cZ_N'))|\cF)=\frac{1}{N} \Var(\langle v, \mu_N\rangle):=a_N+b_N,
\end{eqnarray}
where $a_N=\int \E\{ \langle v, \zeta^x_{\cP_{\phi_N}}\rangle^2 \}\phi_N(x)\mathrm dx$,
and
\begin{align}\label{bn}
b_N=\int_{\R^2} \Big(&\E\{\langle v,  \zeta^x_{\cP_{N\phi_N}^{x_N(z)}} \rangle  \langle v, \zeta^{x_N(z)}_{\cP_{N\phi_N}^{x}} \rangle\}\nonumber\\
&-\E\{ \langle v, \zeta^x_{\cP_{\phi_N}} \rangle \}\E\{\langle v, \zeta^{x_N(z)}_{\cP_{\phi_N}} \rangle \} \Big)\phi_N(x)\phi_N(x_N(z))\mathrm d z\mathrm dx,
\end{align}
where $x_N(z):=x+N^{-1/d}z$.

From Lemma~\ref{hzz} and the assumption that $\sG$ is $\beta$-degree bounded for $\beta>4$, we get \begin{align}\label{an}
a_N\rightarrow \E(\Delta_0^\uparrow)^2 \int h^2(x, x)\phi(x)\mathrm dx.
\end{align} 
Therefore, to compute of~\eqref{anbn}, it suffices to derive the limit of the $b_N$. To this end, we have the following lemma, which is proved similarly to Lemma \ref{hzz}.

\begin{lem}Let $\sG$ be as in Proposition~\ref{CLT_R2}. For $x$ a Lebesgue point of $\phi$ and  $v(\cdot)\phi(\cdot)$,
\begin{eqnarray}\label{eq:vx}
\lim_{N\rightarrow\infty}\E\{\langle v, \zeta^x_{\cP_{\phi_N}} \rangle\}=h(x, x)\E\Delta_0^\uparrow.
\end{eqnarray}
Moreover, any $z\in \R^d$ and $x_N(z):=x+N^{-1/d}z$,
\begin{eqnarray}\label{eq:vxn}
\quad \lim_{N\rightarrow\infty}\E\{\langle v, \zeta^{x_N(z)}_{\cP_{\phi_N}} \rangle\}=h(x, x)\E\Delta_0^\uparrow.
\end{eqnarray}
\end{lem}

\begin{proof}The limit in (\ref{eq:vx}) follows from~\eqref{kappaz} in Lemma \ref{hzz}, since $h(x, y)=w(x)v(y)$.  

It remains to show~\eqref{eq:vxn}. By the Palm theory of Poisson processes 
\begin{eqnarray}
& & \E \sum_{w\in \cP_{N\phi_N}}|v(w)-v(x)|\bm1\{w\in B(x_N(z), KN^{-\frac{1}{d}})\}\nonumber\\
&=&N\int_{B(x_N(z), KN^{-\frac{1}{d}})} |v(w)-v(x)|\phi_N(w)\mathrm dw\nonumber\\
&\leq&N\int_{B(x_N(z), KN^{-\frac{1}{d}})} \left(|v(w)\phi_N(w)-v(x)\phi_N(x)|+v(x)|\phi_N(w)-\phi_N(x)|\right)\mathrm dw, \nonumber
\end{eqnarray}
which tends to zero as $N\rightarrow \infty$, if $z$ is a Lebesgue point of both $\phi$ and $v(\cdot)\phi(\cdot)$. 

Since $v$ has range $[0, 1]$, this implies that
\begin{align}
\limsup_{N\rightarrow \infty}\E \sum_{w\in \cP_{N\phi_N}}| & v(w)-v(x)|\bm1\{(x_N(z), w)\in \sG(\cP_{N\phi_N})\}  \nonumber\\
\leq & \limsup_{N\rightarrow \infty} \E d_K^\uparrow(x_N(z), \sG(\cP_{N\phi_N})
\label{eI}
\end{align}
where 
\begin{align}
d_K^\uparrow&(x_N(z), \sG(\cP_{N\phi_N}))\nonumber\\
&=\sum_{w\in \cP_{N\phi_N}}\bm  1\{(x_N(z), w)\in E(\sG(\cP_{N\phi_N}))\} \bm 1\{w\notin B(x_N(z), KN^{-\frac{1}{d}})\}.\nonumber
\end{align}
Therefore, by Lemma \ref{poissonstabilize}(b), 
\begin{align}\label{eII}
\limsup_{K\rightarrow \infty} \limsup_{N\rightarrow \infty} & \E d_K^\uparrow(x_N(z), \sG(\cP_{N\phi_N}))\nonumber\\
=&\limsup_{K\rightarrow \infty}  \E d_K^\uparrow(0, \sG(\cP_{\phi(z)})=0.
\end{align}
From (\ref{eI}) and (\ref{eII}), for every $x$ which is a Lebesgue point of $\phi$ and $v(\cdot)\phi(\cdot)$,  
\begin{align}
\lim_{N\rightarrow \infty}\E\{\langle v, \zeta^{x_N(z)}_{\phi_N} \rangle\}=&\lim_{N\rightarrow \infty} w(x_N(z))v(x)\E d^\uparrow(x_N(z), \sG(\cP_{N\phi_N}))\nonumber\\
=&\E \Delta_0^\uparrow h(z, z),
\end{align}
completing the proof of~\eqref{eq:vxn}.
\end{proof}

By \cite[Lemma 3.6]{penroseclt} it follows that $$\langle v, \zeta^x_{\cP_{N\phi_N}^{x_N(z)}} \rangle \langle v, \zeta^{x_N(z)}_{\cP_{N\phi_N}^{x}} \rangle\dto h^2(x, x)\E\{d^\uparrow(0, \sG(\cP_1^z)) d^\uparrow(z, \sG(\cP_1^0) )\}.$$ Moreover, since $\sG$ is $\beta$-degree moment bounded for $\beta>4$ by assumption, the LHS above is uniformly integrable, which implies the convergence in expectation. This is summarized in the following lemma:

\begin{lem} [{Penrose \cite[Lemma 3.6]{penroseclt}}] \label{degjoint}Let $\sG$ be as in Proposition~\ref{CLT_R2}. For every Lebesgue point $x$ of $\phi$ and any $z\in \R^d$,
\begin{align}\label{eq:vxz}
\lim_{N\rightarrow\infty}\E\{\langle v, \zeta^x_{\cP_{N\phi_N}^{x_N(z)}} \rangle & \langle v, \zeta^{x_N(z)}_{\cP_{N\phi_N}^{x}} \rangle \}\nonumber\\
= & h^2(x, x)\E\{d^\uparrow(0, \sG(\cP_{\phi(x)}^z)) d^\uparrow(z, \sG(\cP_{\phi(x)}^0) )\}.
\end{align}
\end{lem}

Using this the limit of $b_N$ (defined in~\eqref{bn}) can be derived. The limit in~\eqref{varcond} then follows from~\eqref{an} and the following lemma.

\begin{lem}Let $\sG$ be as in Proposition~\ref{CLT_R2} and $b_N$ as in~\eqref{bn}. Then as $N\rightarrow \infty$
$$b_N\rightarrow r^2  \int  \left(\E\{d^\uparrow(0, \sG(\cP_{1}^z)) d^\uparrow(z, \sG(\cP_{1}^0) )  \}-(\E\Delta^\uparrow_0)^2 \right)\mathrm d z\int \frac{f(y)^2g^2(y)}{\phi^3(y)}\mathrm dy.$$
\end{lem}

\begin{proof} Let $u(x)=\frac{f^2(x)g^2(x)}{\phi^2(x)}$. Since $\sG$ is power-law stabilizing (recall Definition~\ref{expstabilize}), by \cite[Theorem 2.1]{penroseclt} and Lemma~\ref{degjoint} it follows that
\begin{align}\label{pfbn}
b_N \rightarrow& r^2 \int_{\R^2}  \left(\E\{d^\uparrow(0, \sG(\cP_{\phi(x)}^z)) d^\uparrow(z, \sG(\cP_{\phi(x)}^0) )  \}-(\E\Delta^\uparrow_0)^2 \right)u(x) \mathrm dx\mathrm d z.
\end{align}
Note that $$\E\{d^\uparrow(0, \sG(\cP_{\phi(x)}^z)) d^\uparrow(z, \sG(\cP_{\phi(x)}^0) )=\E\{d^\uparrow(0, \sG(\cP_{1}^{\phi(x)^{-\frac{1}{d}} z})) d^\uparrow(\phi(x)^{-\frac{1}{d}} z, \sG(\cP_{1}^0) ),$$ by the  scale invariance of the degree function. Now, substituting $y=\phi(x)^{-\frac{1}{d}}z$ in~\eqref{pfbn}, the result  follows.
%
%
%
%
%
%
\end{proof}

The proof of the limiting normal distribution in Proposition \ref{CLT_R2} now follows from \cite[Theorem 2.2]{penroseclt} (the slight modification required to adapt the result in our setup is straightforward and the details are omitted).

\subsubsection{Proof of Theorem \ref{TH:CLT_R}}
\label{pfcltR}

Recall from~\eqref{R1} and~\eqref{R2} that $\cR(\sG(\cZ_N'))=\cR_1(\sG(\cZ_N'))+\cR_2(\sG(\cZ_N'))$.  Let $ 
\kappa_{\sG}$ and $\tau_\sG$ be as defined in Theorems~\ref{TH:CLT_R1} and~\ref{CLT_R2}. Using $\cR_2(\sG(\cZ_N')) \dto V\sim N(0, \tau_\sG^2)$ (Proposition~\ref{CLT_R2}), Proposition~\ref{cltR1}, and the contimous mapping theorem, for every $t \in \R$, 
$$\E(e^{i t \cR(\sG(\cZ_N'))}|\cF)=e^{i t \cR_2(\sG(\cZ_N'))}\E(e^{i t \cR_1(\sG(\cZ_N'))}|\cF)\dto e^{i t V+\frac{1}{2} 
\kappa_{\sG}^2 t^2}.$$ 

Now, by the Dominated Convergence Theorem 
\begin{align*}
\E(e^{i t \cR(\sG(\cZ_N'))})= \E\E(e^{i  t \cR(\sG(\cZ_N'))}|\cF)\rightarrow & e^{\frac{1}{2}( 
\kappa_{\sG}^2+\tau_\sG^2) t^2},
\end{align*}
which is the characteristic function of $N(0, \kappa_{\sG}^2+\tau_\sG^2)$, as required.

\section{Proof of  Theorem~\ref{EFFSECOND}}
\label{sec:pfeffsecond}

Let $\sX_{N_1}'$ and $\sY_{N_2}'$ be samples from $\P_{\theta_1}$ and $\P_{\theta_2}$ as in~\eqref{epsilonN}, respectively. Let $\phi_N^{\theta_1, \theta_2}(x)=\frac{N_1}{N}f(x|\theta_1)+\frac{N_2}{N}f(x|\theta_2)$, and define
\begin{equation*}\label{hnxyparametric}
h_N^{\theta_1, \theta_2}(x, y)=\frac{N_1N_2 f(x|\theta_1)f(y|\theta_2)}{(N_1f(x|\theta_1)+N_2f(x|\theta_2))(N_1f(y|\theta_1)+N_2f(y|\theta_2))}.
\end{equation*}
Moreover, denote by $\rho^{\theta_1, \theta_2}_{K}(x, y)=\P((x, y)\in E(\cN_K(\cP_{N\phi_N}^{x, y})))$, the probability there is an edge from $x$ to $y$ in the $K$-NN graph on $\cP_{N\phi_N}^{x, y}$. Then, as in~\eqref{exp}, 
\begin{align}\label{beta0}
\E_{H_1}(T(\sG(\cZ_N')))=&N^2 \int_{S\times S} h_{N}^{\theta_1, \theta_2}(x, y)\P((x, y)\in E(\cN_K(\cP_{N\phi_N}^{x, y}))) \phi_N^{\theta_1, \theta_2}(x)\phi_N^{\theta_1, \theta_2}(y)\mathrm dx\mathrm dy \nonumber\\
=&N_1N_2 \int_{S\times S}  f(x|\theta_1)f(y|\theta_2) \rho^{\theta_1, \theta_2}_{K}(x, y) \mathrm dx\mathrm dy \nonumber\\
:=&\frac{N_1N_2}{N^2} \mu_N(\theta_1, \theta_2),
\end{align}
where  $\mu_N(\theta_1, \theta_2)= N^2\int_{S\times S}  f(x|\theta_1)f(y|\theta_2) \rho^{\theta_1, \theta_2}_{K}(x, y) \mathrm dx\mathrm dy$. (Note that $\mu_N$ is related to the function $\delta_N$ in Section \ref{sec:pfoutline} as: $\delta_N(\theta_1, \theta_2) =\frac{N_1N_2}{N^2} \mu_N(\theta_1, \theta_2)$.)

Under the null, $\E_{H_0}(T(\sG(\cZ_N')))=\frac{N_1N_2}{N^2}\mu_N(\theta_1, \theta_1)=\frac{N_1N_2}{N^2}\E(|E(\cN_K(\cP_{N\phi_N}))|)=\frac{K N_1N_2}{N}$, that is, $\mu_N(\theta_1, \theta_1)=K N$. To derive power against local alternatives \eqref{epsilonN}, we need to understand the distribution of 
$$W_N:=\frac{1}{\sqrt N}\left\{ T(\sG(\cZ_N'))-  \frac{N_1N_2}{N^2}\mu_N(\theta_1, \theta_1) \right\},$$ when $\theta_2-\theta_1=\varepsilon_N$ such that $||\varepsilon_N|| \rightarrow 0$. Note that 
\begin{align}\label{eq:W}
W_N=W_N^{(1)}+ \frac{N_1N_2}{N^2}  \cdot \frac{1}{\sqrt N} \bigg\{\mu_N(\theta_1, \theta_2) -\mu_N(\theta_1, \theta_1)\bigg\},
\end{align}
where $$W_N^{(1)}=\frac{1}{\sqrt N} \left\{ T(\sG(\cZ_N'))-  \frac{N_1N_2}{N^2}\mu_N(\theta_2, \theta_1) \right\}.$$ It follows from the proof of Theorem \ref{TH:CLT_R}, that under $\theta_2=\theta_1+\varepsilon_N$, with $||\varepsilon_N|| \rightarrow 0$, $W_N^{(1)}\dto N(0,  \sigma_K^2)$, where $\sigma_K^2$ is the variance of the test statistic under the null, as defined in \eqref{sigmaK}. Therefore, to derive the limiting distribution of \eqref{eq:W}, it suffices to derive the limit of 
\begin{align}
\frac{1}{\sqrt N} \bigg\{\mu_N(\theta_1, \theta_2) -\mu_N(\theta_1, \theta_1)\bigg\} 
=&\frac{\varepsilon_N^\top \grad \mu_N(\theta_1, \theta_1)}{\sqrt N}+\frac{1}{2}\frac{\varepsilon_N^\top \mathrm H \mu_N(\theta_1, \theta_1) \varepsilon_N}{\sqrt N}+ \cR_N,
\label{D}
\end{align}  
where $\grad\mu_N(\theta_1, \theta_1):=\grad_{\theta}\mu_N(\theta_1, \theta)|_{\theta=\theta_1} \in \R^p$ is the gradient vector (with respect to the second coordinate $\theta$) of $\mu_N(\theta_1, \theta)$ evaluated at $\theta=\theta_1$, and $\mathrm H \mu_N(\theta_1, \theta_1)\in \R^{p \times p}$ is the Hessian matrix (with respect to the second coordinate $\theta$) of $\mu_N(\theta_1, \theta)$, at $\theta=\theta_1$, and the remainder term 
\begin{align}\label{errorterm}
\cR_N=\frac{1}{6\sqrt N}\sum_{1 \leq a, b, c \leq p}\varepsilon_{N_a}\varepsilon_{N_b}\varepsilon_{N_c}\frac{\partial^3\mu_N(\theta_1, \theta)}{\partial \theta_{a}\partial \theta_{b}\partial \theta_{c}}\Big|_{\theta=c \theta_1+(1-c) \theta_2},
\end{align} for some $c \in (0, 1)$, where $\theta=(\theta_{1}, \theta_{2}, \ldots, \theta_{p})'$ and $\varepsilon_N=(\varepsilon_{N_1}, \varepsilon_{N_2}, \ldots, \varepsilon_{N_p})'$. 

The limits and orders of the gradient term, the Hessian term, and the remainder term are given in the following lemmas: 

\begin{lem}[Limit of the Gradient Term]\label{gradppn} Let $\varepsilon_N=\frac{h}{N^{\frac{1}{2}-\frac{2}{d}}}$, for some $h \in \R^p\backslash \{\bm 0\}$. Then, under the assumptions in Theorem~\ref{EFFSECOND},
\begin{align*}
\frac{\varepsilon_N^\top \grad  \mu_N(\theta_1, \theta_1)}{\sqrt N} \rightarrow    \frac{p C_{K, 2}}{2d}  \int_{S}  h^\top \grad_{\theta_1} \left(\frac{ \tr(\mathrm H_x f(x|\theta_1))}{f(x|\theta_1)}  \right) f^{\frac{d-2}{d}}(x|\theta_1) \mathrm d x =  -\frac{2 \sigma_K}{r} a_{K, \theta_1}(h),
\end{align*}
where $C_{K, 2}$ is as defined in \eqref{eq:C} and $a_{K, \theta_1}(h)$ as defined in \eqref{eq:aKtheta1}. 
\end{lem}

\begin{lem}[Limit of the Hessian Term] \label{hessianppn} Let $\varepsilon_N=h N^{-\frac{1}{4}}$, for some $h \in \R^p\backslash \{\bm 0\}$. Then, under the assumptions in Theorem~\ref{EFFSECOND},
\begin{align}\label{hessianlimit}
\frac{\varepsilon_N^\top \mathrm H  \mu_N(\theta_1, \theta_1) \varepsilon_N}{\sqrt N} \rightarrow  -  r K  \cdot \E \left( \frac{h^\top\grad_{\theta_1} f(X|\theta_1)}{f(X|\theta_1)} \right)^2 = - \frac{2 \sigma_K}{r} b_{K, \theta_1}(h), 
\end{align}
where the expectation is taken over $X \sim f(\cdot|\theta_1)$ and $b_{K, \theta_1}(h)$ as defined in \eqref{eq:bKtheta1}. 
\end{lem}

\begin{lem}[Controlling the Reminder Term] \label{remainderppn} Let $\varepsilon_N=h N^{-\delta_d}$, for some $h \in \R^p\backslash \{\bm 0\}$. Then, under the assumptions in Theorem~\ref{EFFSECOND},
$$|\cR_N|=O(||h||^3N^{\frac{1}{2}-3 \delta_d}).$$
\end{lem}

The proof of these lemmas are given below.  From the proofs of Lemmas \ref{gradppn} and \ref{hessianppn} it follows that for $\varepsilon_N=h N^{-\delta_d}$, where $h \in \R^p\backslash \{\bm 0\}$, 
\begin{align}\label{eq:ghdeltad}
\frac{|\varepsilon_N^\top \grad  \mu_N(\theta_1, \theta_1)|}{\sqrt N} = O(||h|| N^{\frac{1}{2}-\frac{2}{d}-\delta_d}) \quad \text{and} \quad \frac{\varepsilon_N^\top \mathrm H  \mu_N(\theta_1, \theta_1) \varepsilon_N}{\sqrt N} = \Theta(||h||^2N^{\frac{1}{2}-2 \delta_d}),
\end{align}
where the lower bound on  the Hessian uses Assumption \ref{assumptionlocalpower}(c). Before describing the proofs of these lemmas. we show how they can be used to complete the proof of Theorem \ref{EFFSECOND}. There are two cases depending on the whether dimension is 8 or higher. 

\begin{itemize}

\item[1.]  {\it Suppose the dimension $d \leq 8$}: There are three cases depending how $\varepsilon_N$ decays with $N$. 
\begin{itemize}
\item[(a)] Suppose $N^{\frac{1}{4}}\varepsilon_N=h_N$ such that $||h_N||\rightarrow 0$: Taking $\delta_d=\frac{1}{4}$ in Lemma \ref{remainderppn} it follows that $|\cR_N| \rightarrow 0$. Moreover, by \eqref{eq:ghdeltad}  it follows that 
\begin{align}\label{eq:gradhessian1}
\frac{1}{\sqrt N} \varepsilon_N^\top \grad \mu_N(\theta_1, \theta_1) =O(N^{\frac{1}{4}-\frac{2}{d}}||h_N||), \quad  \frac{1}{\sqrt N} \varepsilon_N^\top \mathrm H \mu_N(\theta_1, \theta_1) \varepsilon_N =\Theta(||h_N||^2),
\end{align}
respectively. Then by \eqref{D}, $$\frac{1}{\sqrt N}\bigg\{\mu_N(\theta_1, \theta_2) -\mu_N(\theta_1, \theta_1)\bigg\} \rightarrow 0,$$ which implies $W_N \dto N(0, \sigma_K^2)$ under $H_1$ as well, where $\sigma_K^2$ is the variance under the null $H_0$, as defined in \eqref{sigmaK}. Therefore, in this regime, the limiting power of the test~\eqref{rejregionN} is $\alpha$.

\item[(b)] Suppose $N^{\frac{1}{4}}\varepsilon_N \rightarrow h$. As before, by  Lemma \ref{gradppn} and \ref{hessianppn}, combined with \eqref{eq:W} and \eqref{D}, it follows that the limiting power of the test~\eqref{rejregionN} is \eqref{eq:effhessian} for dimension $d \leq 7$, and \eqref{eq:effgradhessian} when dimension $d=8$.

\item[(c)] Suppose $N^{\frac{1}{4}}\varepsilon_N = h_N$ such that $||h_N||\rightarrow \infty$ (and $\varepsilon_N \rightarrow 0$). In this case, $|\cR_N|=O(||h_N||^3N^{-\frac{1}{4}})$. Note that $||h_N||^2 \gg N^{-\frac{1}{4}}||h_N||^3$, whenever $||\varepsilon_N|| \rightarrow 0$, which implies, by \eqref{eq:gradhessian1} $$\frac{1}{\sqrt N}\bigg\{\mu_N(\theta_1, \theta_2) -\mu_N(\theta_1, \theta_1)\bigg\} \rightarrow \infty,$$ and the limiting power of the test~\eqref{rejregionN} is 1.

\end{itemize}

\item[2.]  {\it Suppose the dimension  $d \geq 9$}: There are four cases depending how $\varepsilon_N$ decays with $N$. 

\begin{itemize}
\item[(a)] Suppose $N^{\frac{1}{2}-\frac{2}{d}}\varepsilon_N=h_N$ such that $||h_N||\rightarrow 0$: In this case, by (\ref{eq:ghdeltad}), 
$$\frac{1}{\sqrt N}|\varepsilon_N^\top \grad \mu_N(\theta_1, \theta_1)| = O(||h_N||) \rightarrow 0.$$ Again, by (\ref{eq:ghdeltad}) and noting that $d \geq 9$, 
$$\frac{1}{\sqrt N} \varepsilon_N^\top \mathrm H \mu_N(\theta_1, \theta_1) \varepsilon_N = O(||h_N||^2 N^{-\frac{1}{2}+\frac{4}{d}})\rightarrow 0,$$ and $|\cR_N|=O(||h_N||^3 N^{-1+\frac{6}{d}}) \rightarrow 0$ (taking $\delta_d=\frac{1}{2}-\frac{2}{d}$ in Lemma \ref{remainderppn}). Then, by \eqref{eq:W} and \eqref{D}, the limiting power of the test~\eqref{rejregionN} is $\alpha$.

\item[(b)] Suppose $N^{\frac{1}{2}-\frac{2}{d}}\varepsilon_N \rightarrow h$.  As before, $\frac{1}{\sqrt N} \varepsilon_N^\top \mathrm H \mu_N(\theta_1, \theta_1) \varepsilon_N  = O(||h||^2 N^{-\frac{1}{2}+\frac{4}{d}}) \rightarrow 0$ and $|\cR_N|=O(||h_N||^3 N^{-1+\frac{6}{d}}) \rightarrow 0$. Therefore, by Lemma \ref{gradppn} combined  \eqref{eq:W} and \eqref{D}, it follows that the limiting power of the test~\eqref{rejregionN} is 
$$\Phi\left(z_\alpha+ \frac{rp C_{K, 2}}{4d\sigma_K}   \int_{S}  h^\top \grad_{\theta_1} \left(\frac{ \tr(\mathrm H_x f(x|\theta_1))}{f(x|\theta_1)}  \right) f^{\frac{d-2}{d}}(x|\theta_1) \mathrm d x \right),$$
as required.

\item[(c)]  Suppose $N^{\frac{1}{2}-\frac{2}{d}}\varepsilon_N=h_N$  such that $||h_N|| \rightarrow \infty$ and $||N^{\frac{2}{d}}\varepsilon_N|| \rightarrow 0$. In this case, using the assumption $\left| \frac{1}{||h_N||} \int_{S} h_N^\top \grad_{\theta_1} \left(\frac{ \tr(\mathrm H_x f(x|\theta_1))}{f(x|\theta_1)}  \right) f^{\frac{d-2}{d}}(x|\theta_1) \mathrm d x\right|\rightarrow \tau > 0$ (for some $\tau > 0$) and the proof of Lemma \ref{gradppn}, it can be shown that  
$$ \frac{1}{\sqrt N}|\varepsilon_N^\top \grad \mu_N(\theta_1, \theta_1)| =\Theta(||h_N||).$$ Then, from \eqref{eq:ghdeltad}, 
\begin{align*} 
\frac{1}{\sqrt N} \varepsilon_N^\top \mathrm H \mu_N(\theta_1, \theta_1) \varepsilon_N &= O(||h_N||^2 N^{-\frac{1}{2}+\frac{4}{d}}) \nonumber \\
&\ll ||h_N|| \tag*{(since $|| N^{-\frac{1}{2}+\frac{4}{d}}h_N|| =||N^{\frac{2}{d}} \varepsilon_N|| \ll 1$)}\nonumber \\
&\lesssim \frac{1}{\sqrt N}|\varepsilon_N^\top \grad \mu_N(\theta_1, \theta_1)| \nonumber \\
& \lesssim ||h_N||.
\end{align*} 
Moreover,  $|\cR_N|=O(||h_N||^3 N^{-1+\frac{6}{d}}) \ll ||h_N||$, since $||h_N|| \ll N^{\frac{1}{2}-\frac{4}{d}}$ by assumption.  This implies, by \eqref{D}, that $$\frac{1}{\sqrt N}\bigg\{\mu_N(\theta_1, \theta_2) -\mu_N(\theta_1, \theta_1)\bigg\}$$ converges to $-\infty$ or $\infty$ depending on whether $\frac{1}{||h_N||} \int_{S}  h_N^\top \grad_{\theta_1} \left(\frac{ \tr(\mathrm H_x f(x|\theta_1))}{f(x|\theta_1)}  \right) f^{\frac{d-2}{d}}(x|\theta_1) \mathrm d x$ converges to $<0 $ or $> 0$, and the power of the test~\eqref{rejregionN} converges to 1 or 0, respectively.

\item[(d)]  Suppose $N^{\frac{2}{d}}\varepsilon_N \rightarrow h$.  In this case, by  Lemma \ref{remainderppn}, 
\begin{align}\label{eq:remainder_threshold} 
|\cR_N|=O(N^{\frac{1}{2}-\frac{6}{d}}) \ll N^{\frac{1}{2}-\frac{4}{d}}.
\end{align}
Moreover, by \eqref{eq:ghdeltad}, $\frac{1}{\sqrt N}|\varepsilon_N^\top \grad \mu_N(\theta_1, \theta_1)|  =O(N^{\frac{1}{2}-\frac{4}{d}}) $ and $\frac{1}{\sqrt N} \varepsilon_N^\top \mathrm H \mu_N(\theta_1, \theta_1) \varepsilon_N = \Theta(N^{\frac{1}{2}-\frac{4}{d}})$, that is, the gradient and the Hessian terms are of the same order. In particular, defining $\bar \varepsilon_N=  \frac{1}{N^{\frac{1}{2}-\frac{4}{d}}} \varepsilon_N$ and $\hat \varepsilon_N=  \frac{1}{N^{\frac{1}{4}-\frac{2}{d}}} \varepsilon_N$ gives, 
\begin{align}\label{eq:gradhessiansum} 
\Delta_N&:=\frac{1}{N^{\frac{1}{2}-\frac{4}{d}}}   \left(\frac{1}{\sqrt N}\varepsilon_N^\top \grad \mu_N(\theta_1, \theta_1) + \frac{1}{\sqrt N} \varepsilon_N^\top \mathrm H \mu_N(\theta_1, \theta_1) \varepsilon_N \right) \nonumber \\
& =\frac{1}{\sqrt N}\bar \varepsilon_N^\top \grad \mu_N(\theta_1, \theta_1) + \frac{1}{\sqrt N} \hat \varepsilon_N^\top \mathrm H \mu_N(\theta_1, \theta_1) \hat \varepsilon_N.  
\end{align} 
Note that $N^{\frac{1}{2}-\frac{2}{d}}\bar \varepsilon_N \rightarrow h$ and $N^{\frac{1}{4}} \hat \varepsilon_N \rightarrow h$. Then using Lemma \ref{gradppn} and Lemma \ref{hessianppn} in \eqref{eq:gradhessiansum} above gives,
$$\Delta_N \rightarrow  -\frac{2 \sigma_K}{r} (a_{K, \theta_1}(h)+ b_{K, \theta_1}(h)). 
$$ This implies that 
$$\frac{1}{\sqrt N}\varepsilon_N^\top \grad \mu_N(\theta_1, \theta_1) + \frac{1}{\sqrt N} \varepsilon_N^\top \mathrm H \mu_N(\theta_1, \theta_1) \varepsilon_N = N^{\frac{1}{2}-\frac{4}{d}}  \Delta_N,$$
goes to $-\infty$, or $+\infty$, depending on whether $a_{K, \theta_1}(h)+ b_{K, \theta_1}(h)$ is positive or negative. Thus, by \eqref{eq:remainder_threshold} and \eqref{D}, $$\frac{1}{\sqrt N}\bigg\{\mu_N(\theta_1, \theta_2) -\mu_N(\theta_1, \theta_1)\bigg\}$$ converges to $-\infty$ or $\infty$, depending on whether $a_{K, \theta_1}(h)+ b_{K, \theta_1}(h)$ is positive or negative, respectively. Hence, the power of the test~\eqref{rejregionN} converges to 1 or 0,  depending on whether $a_{K, \theta_1}(h)+ b_{K, \theta_1}(h)$ is positive or negative, respectively.

\item[(e)]  Suppose $N^{\frac{2}{d}}\varepsilon_N=h_N$ such that $||h_N|| \rightarrow \infty$ (and $||\varepsilon_N|| \rightarrow 0$). Then by \eqref{eq:ghdeltad}, 
\begin{align*} 
\frac{1}{\sqrt N}|\varepsilon_N^\top \grad \mu_N(\theta_1, \theta_1)|  & =O(N^{\frac{1}{2}-\frac{4}{d}} ||h_N||) \nonumber \\
& \ll N^{\frac{1}{2}-\frac{4}{d}} ||h_N||^2 \nonumber \\
&\lesssim \frac{1}{\sqrt N} \varepsilon_N^\top \mathrm H \mu_N(\theta_1, \theta_1) \varepsilon_N .
\end{align*} 
Moreover, from Lemma \ref{remainderppn}, 
$$|\cR_N|=O(||h_N||^3 N^{\frac{1}{2}-\frac{6}{d}}) \ll N^{\frac{1}{2}-\frac{4}{d}}||h_N||^2 \lesssim \frac{1}{\sqrt N} \varepsilon_N^\top \mathrm H \mu_N(\theta_1, \theta_1) \varepsilon_N.  $$ 
Therefore, using $\frac{1}{\sqrt N} \varepsilon_N^\top \mathrm H \mu_N(\theta_1, \theta_1) \varepsilon_N = \Theta(N^{\frac{1}{2}-\frac{4}{d}} ||h_N||^2)$ (by \eqref{eq:ghdeltad}) and \eqref{D} implies $$\frac{1}{\sqrt N}\bigg\{\mu_N(\theta_1, \theta_2) -\mu_N(\theta_1, \theta_1)\bigg\} \rightarrow \infty,$$ and the power of the test~\eqref{rejregionN} converges 1.

\end{itemize}

\end{itemize}

This completes the proof of Theorem \ref{EFFSECOND}, assuming  Lemma \ref{gradppn}, \ref{hessianppn}, and \ref{remainderppn}. The rest of this section is organized as follows: We begin by with some general preparations in Section \ref{sec:general}. Then proofs of Lemma  \ref{gradppn}, Lemma \ref{hessianppn}, and Lemma \ref{remainderppn}, are then given in Sections~\ref{gradientlimit},~\ref{hessianlimit}, and \ref{remainderlimit}, respectively.

\subsection{Preparations}
\label{sec:general}

Recall that $\cP_1^0$ is the Poisson process with rate 1 in $\R^d$ with origin $0$ added to it. Fix $s\geq 0$, and applying \eqref{eq:mb1} to the function $\varphi^{(s)}(x, \cN_K(S))=\sum_{y \in S} ||x-y||^s \bm 1\{(x, y) \in E(\cN_K(S))\}$, gives 
\begin{align*}
\varphi^{(s)}(0, \cN_K(N^{\frac{1}{d}}(\cP_{Nf(\cdot|\theta_1)}-x))& = N^{\frac{s}{d}}\sum_{y \in \cP_{Nf(\cdot|\theta_1)}} ||x-y||^s \bm 1\{(x, y) \in E(\cN_K(\cP_{Nf(\cdot|\theta_1)}))\}\\
& \rightarrow  \varphi^{(s)}(0, \cN_K(\cP_1)),
\end{align*}
in distribution and in expectation. This implies 

\begin{align}\label{eq:deglength}
\E \varphi^{(s)}(0, \cN_K(N^{\frac{1}{d}}(\cP_{Nf(\cdot|\theta_1)}-x)) & = N^{\frac{s+d}{d}}\int_S ||x-y||^s \P((x, y) \in E(\cN_K(\cP_{Nf(x|\theta_1)}))) f(y|\theta_1) \mathrm dy \nonumber \\ 
& \rightarrow \E\varphi^{(s)}(0, \cN_K(\cP_{f(x|\theta_1)})) \nonumber \\
&=\frac{1}{f(x|\theta_1)^{\frac{s}{d}}} \E \varphi^{(s)}(0, \cN_K(\cP_1)).
\end{align}
Moreover, since the support of $f(\cdot|\theta_1)$ is compact, this convergence is uniformly in $x \in S$. Recalling \eqref{eq:C} it is easy to see that $C_{K, s}=\E \varphi^{(s)}(0, \cN_K(\cP_1))$. By the Palm formula, 
\begin{align} 
C_{K, s} & =  \sum_{b=0}^{K-1}  \int_{\R^d} ||x||^s  e^{-V_d ||x||^d } \frac{V_d^b ||x||^{db}}{b!} \mathrm d x \tag*{($V_d$ is the volume of the unit ball in $\R^d$ )}\nonumber\\
&=S_d \sum_{b=0}^{K-1} \frac{V_d^b}{b!} \int _0^\infty r^{s +bd+d-1} e^{-V_d r^d } \mathrm d r \tag*{($S_d$ is the surface area of the unit ball in $\R^d$ )} \nonumber\\
&= \left(\frac{1}{V_d}\right)^{\frac{s}{d}}    \sum_{b=0}^{K-1} \frac{1}{b!}\int _0^\infty t^{\frac{s}{d}+b} e^{-t} \mathrm d t\nonumber \tag{substituting $t=V_dr^d $}\\
\label{cd} &= \left(\frac{1}{V_d}\right)^{\frac{s}{d}} \sum_{b=0}^{K-1} \frac{1}{b!} \Gamma\left(\frac{s+bd+d}{d}\right).
\end{align}
This formula will be required in calculating the leading constants in the asymptotics  of the gradient and the Hessian terms. In particular, for $K=1$, the following useful identity is easy to derive from \eqref{cd}:
\begin{align}\label{eq:C1}
C_{1, s+d}=\frac{s+d}{d V_d} C_{1,s}.
\end{align}

The gradient and Hessian terms in \eqref{D} involve integrals of $f(\cdot|\theta)$ or $h^\top \grad_\theta f(\cdot|\theta)$ over small balls. To heal with such terms, the following result will be useful:

\begin{ppn}\label{dd+1}Let $S \subset \R^d$ be a compact and convex set, and $h: S \rightarrow \R$ be a three times continuously differentiable function, with gradient vector $\grad_x h(x) \in \R^d$ and Hessian matrix $\mathrm H_x h(x) \in \R^{d\times d}$, at the point $x$. For $x$ in the interior of $S$ and $r> 0$, define $$v_{x} (r)=\int_{B_S(x, r)} h(z)\mathrm dz,$$ where $B_S(x, r)=B(x, r)\cap S$. Then, for every fixed $x\in S$, the derivatives $v^{(s)}_{x}(r)=\frac{\partial^{s}}{\partial r^s}v_{x} (r)$ satisfy: 
\begin{itemize}
\item [(a)] $v^{(s)}_{x}(0)=0$, for $1 \leq s \leq d-1$.
\item [(b)] $v^{(d)}_{x}(0)=d! V_d h(x)$, where $V_d$ is the volume of the unit ball in $\R^d$.
\item [(c)] $v^{(d+1)}_{x}(0)=0$, 
\item [(d)] $v^{(d+2)}_{x}(0)=\frac{1}{2} (d+1)!V_d\tr(\mathrm H_x h(x))$. 
\item[(e)] For any compact interval $I \subset [0, \infty)$, $\sup_{r \in I}|v^{(d+3}_{x}(r)| < \infty $, uniformly in $x \in S$. 
\end{itemize}
\end{ppn}

\begin{proof} Note that 
\begin{align*}
v_{x} (r)=\int_{0}^r  \int_{\partial B_S(x, t)} h(z)  \mathrm dz\mathrm d t=\int_0^r H_x(t)\mathrm dt,
\end{align*}
where $H_x(t)=\int_{\partial B_S(x, t)} h(z)\mathrm dz$. (For any set $A \subset \R^d$, $\partial A$ will denote the boundary of the set $A$.) By the Taylor series expansion of $h(z)$ around $x$, 
\begin{align}\label{d}
& H_x(t) \nonumber \\
&=h(x)t^{d-1}S_d+ \int_{\partial B(x, t)}\left(\langle x-z, \grad_x h(x)\rangle + \frac{1}{2}   (x-z)^\top \mathrm H_x h(x) (x-z)  +   R_x(z) \right) \mathrm d z,
\end{align}
where $S_d=|\partial B(0, 1)|$ is the surface area of the unit ball in $\R^d$, and 
$$R_x(z)=\frac{1}{6}\sum_{1 \leq a, b, c \leq d}(x_a-z_a)(x_b-z_b)(x_c-z_c) \frac{\partial^3 h(s)}{\partial{s_a}\partial{s_b}\partial{s_c}}\Big|_{s=\zeta_x(z)},$$ 
where $\zeta_x(z) = c x+(1-c)z$, for some $c \in (0, 1)$. Since, by assumption, $h(\cdot)$ is a three times continuously differentiable function on a compact set $S$, $\sup_{a, b, c}\sup_{x \in S} |\frac{\partial^3 h(s)}{\partial{s_a}\partial{s_b}\partial{s_c}}|:= M< \infty$. Then defining $e_x(r):=\int_{B_S(x, r)}R_x(z) \mathrm dz$, 
\begin{align}
\left|e_x(r)\right| & \leq \frac{M}{6} \int_{B(x, r)} \sum_{1 \leq a, b, c \leq d}|x_a-z_a||x_b-z_b||x_c-z_c| \mathrm dz \nonumber \\
& \leq  \frac{M}{6} \int_{B(x, r)} ||x-z||_1^3 \mathrm dz \nonumber \\
& \leq  \frac{M d^{\frac{3}{2}}}{6} \int_{B(x, r)} ||x-z||^3 \mathrm dz \tag*{(using $||x||_1 \leq \sqrt d ||x||$)}\nonumber \\
& \leq  \frac{M d^{\frac{3}{2}}}{6} \int_{B(0, r)} ||z||^3 \mathrm dz \nonumber \\
\label{eq:werror} & = \frac{M d^{\frac{3}{2}}}{6} S_{d-1}\int_0^r a^{d+2} \mathrm da =   \frac{M d^{\frac{3}{2}}}{6 (d+3)} r^{d+3}.
\end{align} 
This implies that, for every $x\in S$, the $s$-th derivative $e_x^{(s)}(0)=0$, for all $1 \leq s \leq d+2$ and $\sup_{r \in I}|e^{(d+3)}(r)| < \infty$, uniformly in $x \in S$, for every compact interval $I \subset [0, \infty)$.

Now, let $r>0$ be such that $B(x, r) \subset S$ (which exists because $x$ is in the interior of $S$) . Note that, for all $x \in \R^d$ and $0 \leq t \leq r$, $$\int_{\partial B(x, t)}\langle x-z, \grad_x h(x)\rangle \mathrm dz=0,$$ by symmetry. Next, consider the spectral decomposition $\mathrm H_x h(x)=P_x^{\bot} \Lambda(x) P_x$, where $$\Lambda(x)=\diag(\lambda_1(x), \lambda_2(x), \ldots, \lambda_d(x))$$ is the diagonal matrix of eigenvalues of the Hessian matrix $\mathrm H_x h(x)$. Under orthogonal the transformation $z\mapsto P_x(z-x)$,  
\begin{align}
& \int_{\partial B(x, t)} (x-z)^\top \mathrm H_xh(x)(x-z) \mathrm dz \nonumber \\ 
=& \int_{\partial B(0, t)} \sum_{i=1}^d \lambda_i(x) z_i^2 \mathrm dz\nonumber\\
=&\frac{S_d}{d}\tr(\mathrm H_x h(x))t^{d+1} 
\tag*{(using $\int_{\partial B(0, t)} z_i^2 \mathrm d z= \frac{1}{d} \int_{\partial B(0, t)} ||z||^2 \mathrm d z=\frac{t^{d+1}}{d} S_{d-1})$} \nonumber \\
\label{dsecondterm} =&V_d \tr(\mathrm H_x h(x))t^{d+1},
\end{align}
since $V_d=\frac{S_d}{d}$. 

Therefore, for $r$ small enough~\eqref{d} and~\eqref{dsecondterm}, gives $H_x(t)=h(x) t^{d-1}S_d+\frac{1}{2}V_d \tr(\mathrm H_x h(x))t^{d+1}+\int_{\partial B(x, t)} R_x(z)\mathrm d z$, which implies 
$$v_{x} (r)=h(x) r^{d}V_d+\frac{1}{2 (d+2)}V_d \tr(\mathrm H_x h(x))r^{d+2}+e_x(r).$$
Then from the discussion following \eqref{eq:werror}, the result follows. 
\end{proof}

The following simple observation will also be useful in the proofs. To this end, note that $A \lesssim_\square B$ means $A \leq C(\square) B$, where $C:=C(\square) > 0$ is a constant that depends only on the subscripted quantities.

\begin{obs}\label{obsintconv}Let $A\subseteq \R^d$ be compact and convex with non-empty interior, and $z\in S^{d-1}$ be a fixed unit vector. Then for non-negative integers $r, s>0$,  and $N$ large enough,  
$$\int_{\R^d\backslash N^{\frac{1}{d}} A} ||t||^r |\langle t,  z \rangle|^s e^{-||t||^d}\mathrm dt \lesssim_d e^{-\frac{1}{2}N^\frac{1}{2}}.$$
\end{obs}

\begin{proof} Choose $N$ large enough, such that $B(0, \frac{1}{2}) \subset N^{\frac{1}{2d}}A$. Then $B(0, \frac{N^{\frac{1}{2d}}}{2}) \subset N^{\frac{1}{d}}A$, and  by the Cauchy-Schwarz inequality
\begin{align*}
\int_{\R^d\backslash N^{\frac{1}{d}} A} ||t||^r |\langle t,  z \rangle|^s e^{-||t||^d}\mathrm dt \leq &  \int_{\R^d\backslash B(0, \frac{N^{\frac{1}{2d}}}{2})} ||t||^{r+s+1} e^{-||t||^d}\mathrm dt \\
\lesssim_d &  \int_{\frac{N^{\frac{1}{2d}}}{2}}^\infty z^{r+s+d} e^{-z^d} \mathrm dz \\
\lesssim_d &  \int_{N^{\frac{1}{2}}}^\infty y^{\frac{r+s+1}{d}} e^{-y} \mathrm dy \\
\lesssim_d &  \int_{N^{\frac{1}{2}}}^\infty e^{-\frac{1}{2}y} \mathrm dy \lesssim_d e^{-\frac{1}{2}N^\frac{1}{2}},
\end{align*}
using $y\leq e^{\frac{d}{2(r+s+1)} y}$ for $y$ large enough.
\end{proof}

\subsection{Proof of Proposition~\ref{gradppn}: Limit of the gradient term} 
\label{gradientlimit}

Recall the definition of $\mu_N(\theta_1, \theta_2)$ from~\eqref{beta0}. Note that for $x, y$ in the interior of $S$,
\begin{align*}
\rho^{\theta_1, \theta_2}_{K}(x, y)&=\P((x, y)\in E(\cN_K(\cP_{N\phi_N}^{x, y}))) = e^{-\lambda_{\theta_1, \theta_2}(x, y)} \sum_{b=0}^{K-1} \frac{\lambda_{\theta_1, \theta_2}(x, y)^b}{b!}, 
\end{align*}
where $\lambda_{\theta_1, \theta_2}(x, y)=\int_{B_S(x, ||x-y||)}\{N_1f(z|\theta_1)+ N_2 f(z|\theta_2) \} \mathrm dz$. (This is the probability that there are at most $K-1$ points in the region $B_S(x, ||x-y||)$ in the Poisson process $\cP_{N\phi_N}$.) Then differentiating with respect to $\theta_2$ under the integral sign gives,
\begin{align}\label{eq:grad_edgeprob}
\grad \rho^{\theta_1, \theta_2}_{K}(x, y)&=-N_2\left(\int_{B_S(x, ||x-y||)} \grad_{\theta_2}  f(z|\theta_2)\mathrm d z\right) \left\{ \rho^{\theta_1, \theta_2}_{K}(x, y) -\rho^{\theta_1, \theta_2}_{K-1}(x, y) \right\}.
\end{align} 
Then by the chain rule and differentiating under the integral sign gives, 
\begin{align*}
\grad \mu_N(\theta_1,  \theta_2)=N^2\int_{S\times S}  \left(f(x|\theta_1)\grad_{\theta_2} f(y|\theta_2) \rho^{\theta_1, \theta_2}_{K}(x, y) +  f(x|\theta_1) f(y|\theta_2) \grad_{\theta_2} \rho^{\theta_1, \theta_2}_{K}(x, y) \right) \mathrm dx\mathrm dy.
\end{align*}
This and \eqref{eq:grad_edgeprob} implies that 
\begin{align}\label{grad1}
\varepsilon_N^\top \grad \mu_N (\theta_1,  \theta_1)=T_1-N_2 T_2, 
\end{align} 
where
\begin{align}
T_1:=&N^2\int_{S \times S}  f(x|\theta_1) \varepsilon_N^\top \grad_{\theta_1}  f(y|\theta_1)  \rho_K^{\theta_1, \theta_1}(x, y) \mathrm dx\mathrm dy\nonumber\\
T_2:=& N^2\int_{S \times S}  f(x|\theta_1) f(y|\theta_1) \left(\int_{B_S(x, ||x-y||)} \varepsilon_N^\top  \grad_{\theta_1}  f(z|\theta_1)\mathrm dz\right) 
 \left\{ \rho^{\theta_1, \theta_1}_{K}(x, y) -\rho^{\theta_1, \theta_1}_{K-1}(x, y) \right\} \mathrm dx\mathrm dy. \nonumber 
\end{align}

The exact asymptotics of $T_1$ and $T_2$ are obtained in the following two lemmas.

\begin{lem}\label{t11limit} Let $\varepsilon_N=\frac{h}{N^{\frac{1}{2}-\frac{2}{d}}}$, for some $h \in \R^p\backslash \{\bm 0\}$. Then 
$$\frac{T_1}{\sqrt N}\rightarrow  \frac{V_d^{K} C_{1, Kd+2} }{2(d+2) (K-1)!} \int_{S}  h^\top \grad_{\theta_1} \left(\frac{ \tr(\mathrm H_x f(x|\theta_1))}{f(x|\theta_1)}  \right) f^{\frac{d-2}{d}}(x|\theta_1) \mathrm d x.$$
\end{lem}

\begin{lem}\label{t12limit} Let $\varepsilon_N=\frac{h}{N^{\frac{1}{2}-\frac{2}{d}}}$, for some $h \in \R^p\backslash \{\bm 0\}$. Then 
$$\frac{N_2 T_2}{\sqrt N}=  \frac{q C_{K, 2}}{2d} \int_{S}  h^\top \grad_{\theta_1} \left(\frac{ \tr(\mathrm H_x f(x|\theta_1))}{f(x|\theta_1)}  \right) f^{\frac{d-2}{d}}(x|\theta_1) \mathrm d x.$$
\end{lem}

Proposition~\ref{gradppn} is a direct consequence of the above lemmas, \eqref{grad1}, and the following observation:

\begin{obs} $\frac{V_d^{K} C_{1, Kd+2} }{(d+2) (K-1)!}=\frac{C_{K, 2}}{d}.$
\end{obs}

\begin{proof}
By \eqref{cd}
\begin{align*} 
C_{K, 2} &= \left(\frac{1}{V_d}\right)^\frac{2}{d}  \sum_{b=0}^{K-1} \frac{1}{b!} \Gamma\left(\frac{2+bd+d}{d}\right) = \left(\frac{1}{V_d}\right)^\frac{2}{d}  \sum_{b=1}^{K} \frac{1}{(b-1)!} \Gamma\left(\frac{2+bd}{d}\right), 
\end{align*}
and  $\frac{V_d^{K} C_{1, Kd+2} }{(d+2) (K-1)!}=\left(\frac{1}{V_d}\right)^\frac{2}{d} \frac{1}{(d+2)(K-1)!} \Gamma\left(\frac{(K+1)d+2}{d}\right)$. Therefore, to prove the result it suffices to show that 
\begin{align}\label{eq:CK2}
\sum_{b=1}^{K} \frac{1}{(b-1)!} \Gamma\left(\frac{2+bd}{d}\right)=\frac{d}{(d+2)(K-1)!} \Gamma\left(\frac{(K+1)d+2}{d}\right). 
\end{align}
Note that the result holds for $K=1$. Assuming the result holds for all $K' \leq K$, we have 
\begin{align*}
\sum_{b=1}^{K+1} \frac{1}{(b-1)!} \Gamma\left(\frac{2+bd}{d}\right) & = \frac{d}{(d+2)(K-1)!} \Gamma\left(\frac{(K+1)d+2}{d}\right) +  \frac{1}{K!} \Gamma\left(\frac{2+(K+1)d}{d}\right) \nonumber \\
& =\frac{d}{(d+2)K!} \Gamma\left(\frac{(K+1)d+2}{d}\right) \left(\frac{Kd +(d+2)}{d} \right)  \nonumber \\ 
& = \frac{d}{(d+2)K!} \Gamma\left(\frac{(K+2)d+2}{d}\right), \tag*{(using $\Gamma(x+1)=x\Gamma(x)$)}
\end{align*}
as required. This shows that \eqref{eq:CK2} holds for all $K\geq 1$, completing the proof of the observation.
\end{proof}

It remains to prove the above lemmas, which is given below in Section \ref{sec:pft11limit} and Section \ref{sec:pft12limit}, respectively.

\subsubsection{\textbf{Proof of Lemma \ref{t12limit}}}
\label{sec:pft12limit}

It follows from \eqref{eq:deglength}, \eqref{eq:C} and the dominated convergence theorem that, 
\begin{align}\label{eq:length}
N^{\frac{r+d}{d}}\int_{S \times S} w(x) ||x-y||^{r} \rho_K^{\theta_1, \theta_1}(x, y) f(x|\theta_1) f(y|\theta_1) \mathrm d x \mathrm d y \rightarrow C_{K, r} \int_S w(x)f(x|\theta_1)^{1-\frac{r}{d}} \mathrm dx.
\end{align}
Now, the RHS above is zero if $w(\cdot)$ is such that  $\int_S w(x)f(x|\theta_1)^{1-\frac{r}{d}}=0$. This exactly what we encounter in the analysis of the term $T_2$. In this case, in order to determine the correct order and exact asymptotics of the LHS in \eqref{eq:length}, the higher orders terms  have to analyzed.  To this end, we have the following proposition, which gives the exact asymptotics of such `degenerate' functionals, a result which could be of independent interest.

\begin{ppn}\label{ppn:dw} Let $w: S \rightarrow \R$ be a continuously differentiable function such that $\int_S w(x) f(x|\theta_1)^{1-\frac{r}{d}}\mathrm dx=0$,  for some integer $r\geq 0$. Then,
\begin{align*}
N^{\frac{r+d+2}{d}} & \int_{S \times S} w(x) ||x-y||^{r} \rho_K^{\theta_1, \theta_1}(x, y) f(x|\theta_1) f(y|\theta_1) \mathrm d x \mathrm d y  \\
& \rightarrow  - \frac{r  C_{K, r+2} }{2d(d+2)} \int_S  w(x) \tr(\mathrm H_x f(x|\theta_1) )   f^{-\frac{r+2}{d}}(x|\theta_1) \mathrm d x. 
\end{align*}
\end{ppn}

The proof of Proposition \ref{ppn:dw} is given below. Here, we show how it can used to complete the proof of Lemma~\ref{t12limit}: Recall from~\eqref{grad1} that 
\begin{align}\label{t2}
T_2=T_{2}^{(K)}-T_2^{(K-1)},
\end{align}
where 
\begin{align}\label{t2K}
T_2^{(K)}= &  \frac{1}{N^{\frac{1}{2}-\frac{2}{d}}} \E \sum_{x, y \in \cP_{Nf(\cdot|\theta_1)}} v_{x, \theta_1}(||x-y||, h) \bm 1\{(x, y)\in E(\cN_K(\cP_{Nf(\cdot|\theta_1)}))\},
\end{align}
with $v_{x, \theta_1}(r, b)= \int_{B_S(x, r)} b^\top  \grad_{\theta_1}  f(z|\theta_1)\mathrm dz$, for $b\in \R^p, \theta_1\in \Theta$ and $r\geq 0$.

Let $v^{(s)}_{x, \theta_1}(r, b)=\frac{\partial^{s}}{\partial r^s}v_{x, \theta_1}(r, b)$. Then it is easy to see from Proposition \ref{dd+1}, that $v_{x, \theta_1}(0, b)=0$, $v^{(s)}_{x, \theta_1}(0, b)=0$, for all $1 \leq s \leq d-1$,  and $v^{(d+1)}_{x, \theta_1}(0, b)=0$. Therefore, by a Taylor series expansion of $v_{x, \theta_1}(||x-y||, h)$ around 0, $$T_2^{(K)}=T_{21}^{(K)}+T_{22}^{(K)}+T_{23}^{(K)},$$ 
where
\begin{align*}
T_{21}^{(K)} &= \frac{1}{d!} \cdot \frac{1}{N^{\frac{1}{2}-\frac{2}{d}}} \E  \sum_{x, y \in \cP_{Nf(\cdot|\theta_1)}} v^{(d)}_{x, \theta_1}(0, h) ||x-y||^{d} \bm 1\{(x, y)\in E(\cN_K(\cP_{Nf(\cdot|\theta_1)}))\}\nonumber\\ 
T_{22}^{(K)}&=\frac{1}{(d+2)!} \cdot \frac{1}{N^{\frac{1}{2}-\frac{2}{d}}} \E \sum_{x, y \in \cP_{Nf(\cdot|\theta_1)}} v^{(d+2)}_{x, \theta_1}(0, h)\cdot ||x-y||^{d+2} \bm 1\{(x, y)\in E(\cN_K(\cP_{Nf(\cdot|\theta_1)}))\} \nonumber \\
T_{23}^{(K)}&=\frac{1}{(d+3)!} \cdot \frac{1}{N^{\frac{1}{2}-\frac{2}{d}}} \E \sum_{x, y \in \cP_{Nf(\cdot|\theta_1)}} v^{(d+3)}_{x, \theta_1}(\zeta_{x, y}, h)\cdot ||x-y||^{d+3} \bm 1\{(x, y)\in E(\cN_K(\cP_{Nf(\cdot|\theta_1)}))\},
\end{align*}
for some $\zeta_{x, y} \in (0, ||x-y||)$.

Now, using $v^{(d)}_{x, \theta_1}(0, h)=d!V_d h^\top  \grad_{\theta_1}  f(x|\theta_1)$ (by Proposition \ref{dd+1}) and Proposition \ref{ppn:dw} (with $r=d$ and $w(x)=h^\top  \grad_{\theta_1}  f(x|\theta_1)$) gives 
\begin{align*}
N^{\frac{1}{2}} T_{21}^{(K)}&=V_d N^{\frac{2}{d}} \E  \sum_{x, y \in  \cP_{Nf(\cdot|\theta_1)}} h^\top  \grad_{\theta_1}  f(x|\theta_1)  ||x-y||^{d} \bm 1\{ (x, y)\in E(\cN_K(\cP_{Nf(\cdot|\theta_1)})) \} \nonumber \\ 
&= - \sum_{b=0}^{K-1} \frac{V_d^{b+1} C_{1, d+bd+2} }{2(d+2) b!}\int_S  h^\top  \grad_{\theta_1}  f(x|\theta_1) \tr(\mathrm H_x f(x|\theta_1) )   f^{-\frac{d+2}{d}}(x|\theta_1),
\end{align*} 
using $C_{K, d+2}=\sum_{b=0}^{K-1}\frac{V_d^b}{b!} C_{1, d+bd+2}$. Therefore, 
\begin{align}\label{t21}
N^{\frac{1}{2}} \left(T_{21}^{(K)}-T_{21}^{(K-1)}\right) \rightarrow - \frac{V_d^{K} C_{1, Kd+2} }{2(d+2) (K-1)!}\int_S  h^\top  \grad_{\theta_1}  f(x|\theta_1) \tr(\mathrm H_x f(x|\theta_1) )   f^{-\frac{d+2}{d}}(x|\theta_1). 
\end{align} 

Next, using and $v^{(d+2)}_{x, \theta_1}(0, h)=\frac{1}{2} (d+1)!V_d \tr(\mathrm H_x [h^\top  \grad_{\theta_1}  f(x|\theta_1)] )$ and~\eqref{eq:mb1}, it follows that,
\begin{align*} 
N^{\frac{1}{2}} T_{22} &=\frac{V_d}{2(d+2)} N^{\frac{2}{d}} \E \sum_{x, y \in \cP_{Nf(\cdot|\theta_1)}} \tr(\mathrm H_x [h^\top  \grad_{\theta_1}  f(x|\theta_1)] ) ||x-y||^{d+2} \bm 1\{(x, y)\in E(\cN_K(\cP_{Nf(\cdot|\theta_1)}))\} \\ 
& \rightarrow \frac{C_{K, d+2}V_d}{2(d+2)} \int_S \tr(\mathrm H_x [h^\top  \grad_{\theta_1}  f(x|\theta_1)] ) f(x|\theta_1)^{-\frac{2}{d}} \mathrm d x. 
\end{align*}
Therefore, 
\begin{align}\label{t22}
N^{\frac{1}{2}} \left(T_{22}^{(K)}-T_{22}^{(K-1)}\right) & \rightarrow \frac{V_d^{K} C_{1, Kd+2} }{2(d+2) (K-1)!} \int_S \tr(\mathrm H_x [h^\top  \grad_{\theta_1}  f(x|\theta_1)] ) f(x|\theta_1)^{-\frac{2}{d}} \mathrm d x. 
\end{align}
where the last step uses $C_{K, d+2}=\sum_{b=0}^{K-1} \frac{V_d^{b} C_{1, d+bd+2} }{b!}$.

Finally, note that by Assumption \ref{assumptionlocalpower}, $\sup_{x, z\in S} |v^{(d+3)}_{x, \theta_1}(z, h)|< \infty$, and using $$\E \sum_{x, y \in \cP_{Nf(\cdot|\theta_1)}}  ||x-y||^{d+3} \bm 1\{(x, y)\in E(\cN_K(\cP_{Nf(\cdot|\theta_1)}))\} =O(N^{-\frac{3}{d}}),$$ (by \eqref{eq:mb1}), it follows that $N^{\frac{1}{2}} T_{23}^{(K)}=O(N^{-\frac{1}{d}})$. Therefore, combining \eqref{t21} and \eqref{t22}, with \eqref{t2} gives 
\begin{align}
N^{\frac{1}{2}} T_2 \rightarrow \frac{V_d^{K} C_{1, Kd+2} }{2(d+2) (K-1)!} \int_{S} h^\top \grad_{\theta_1} \left(\frac{ \tr(\mathrm H_x f(x|\theta_1))}{f(x|\theta_1)}  \right) f^{\frac{d-2}{d}}(x|\theta_1) \mathrm d x, 
\end{align} 
using the the chain-rule of differentiation, 
\begin{align*}
&\int_{S}  \left(\frac{ \tr(\mathrm H_x h^\top \grad_{\theta_1}  f(x|\theta_1)) f(x|\theta_1)-h^\top \grad_{\theta_1}  f(x|\theta_1)  \tr(\mathrm H_x f(x|\theta_1))}{f^2(x|\theta_1)} \right)  f^{\frac{d-2}{d}}(x|\theta_1) \mathrm d x \nonumber   \\
&=\int_{S} h^\top \grad_{\theta_1} \left(\frac{ \tr(\mathrm H_x f(x|\theta_1))}{f(x|\theta_1)}  \right) f^{\frac{d-2}{d}}(x|\theta_1) \mathrm d x. 
\end{align*}
This completes the proof of Lemma~\ref{t12limit}. \\

\noindent\textit{Proof of Proposition}~\ref{ppn:dw}: \label{pfdiffxy} Define, for $x\in S$,  
\begin{align}\label{eq:Tx}
T(x):=\int_{S} ||x-y||^{r} \rho_K^{\theta_1, \theta_1}(x, y) f(y|\theta_1) \mathrm d y. 
\end{align}
Note that we need to find the limit of 
\begin{align}\label{eq:TxI}
N^{\frac{r+d+2}{d}} \int_{S \times S} w(x) ||x-y||^{r} \rho_K^{\theta_1, \theta_1}(x, y) f(x|\theta_1) f(y|\theta_1) \mathrm d x \mathrm d y = \int_S w(x) T(x) f(x|\theta_1) \mathrm d x. 
\end{align}
Rewrite, $T(x)=\sum_{b=0}^{K-1} \frac{N^b}{b!}  T_{r, b}(x)$, where 
\begin{align}\label{e1}
T_{r, b}(x) &:=  \int_{S} ||x-y||^{r} f(y|\theta_1) e^{-N \int_{B_S(x, ||x-y||)} f(z|\theta_1) \mathrm d z} \left( \int_{B_S(x, ||x-y||)} f(z|\theta_1) \mathrm d z   \right)^b \mathrm d y. 
\end{align}

\begin{lem}\label{lm:e2} Let $w(\cdot)$ be as in the statement of Proposition~\ref{ppn:dw}. Then 
\begin{align*}
 \limsup_{N \rightarrow \infty} N^{\frac{r+d+2}{d}} \int_{S} w(x)\left|T(x)-\sum_{b=0}^{K-1}\frac{N^{b}}{b!}  V_d^{b} f(x|\theta_1)^b \left\{T_{r+bd, 0} + b \eta_{\theta_1}(x) T_{r+bd+2, 0}\right\} \right| f(x|\theta_1) \mathrm d x=0,
\end{align*}
where $\eta_{\theta_1}(x)=\frac{\tr(\mathrm H_x f(x|\theta_1))}{2 (d+2) f(x|\theta_1)}$. 
\end{lem}

\begin{proof} Let $v_{x, f(\cdot|\theta_1)}(||x-y||)=\int_{B_S(x, ||x-y||)} f(z|\theta_1) \mathrm dz$. Then by Proposition \ref{dd+1}  
\begin{align*} 
v_{x, f(\cdot|\theta_1)}&(||x-y||) \\ 
& = V_d f(x|\theta_1) ||x-y||^d + V_d f(x|\theta_1) \eta_{\theta_1}(x) ||x-y||^{d+2} + \frac{v_{x, f(\cdot|\theta_1)}^{(d+3)}(\zeta_{x, y})}{(d+3)!}  ||x-y||^{d+3} , 
\end{align*}
where $\zeta_{x, y} \in (0, ||x-y||)$, for some $c \in (0, 1)$. By Proposition \ref{dd+1},  $\sup_{t, x}|v_{x, f(\cdot|\theta_1)}^{(d+3)}(t)|< \infty$ and by the multinomial theorem 
$$v_{x, f(\cdot|\theta_1)}(||x-y||)^b= V_d^b f(x|\theta_1)^b ||x-y||^{bd} \left\{1+ b \eta_{\theta_1}(x) ||x-y||^{2} \right\} + \sum_{s=3}^{3b}O(||x-y||^{bd+s}),$$ where the constants in the $O(\cdot)$ terms are independent of $x$ and $y$. 

Now, by \eqref{eq:deglength}, 
\begin{align*}
N^{\frac{r+(b+1)d+2}{d}} & \sup_{x \in S}\int_{S} ||x-y||^{r+b d+s} f(y|\theta_1) e^{-N \int_{B_S(x, ||x-y||)} f(z|\theta_1) \mathrm d z} \mathrm d y \nonumber \\ 
&=O(N^{-\frac{s-2}{d}})=o(1),
\end{align*} 
for all $s \in [3, 3b]$. Therefore, recalling the definition of $T_{r, b}(\cdot)$ from \eqref{e1}, whenever $b \geq 1$, 
$$N^{\frac{r+d+2}{d}} \int_{S} w(x)\left|T_{r, b}(x)-N^b V_d^b f(x|\theta_1)^b \left\{T_{r+bd, 0} + b \eta_{\theta_1}(x) T_{r+bd+2, 0}\right\}\right| f(x|\theta_1) \mathrm d x =o(1).$$ 
The lemma now follows by noting that $T(x)=\sum_{b=0}^{K-1} \frac{N^b}{b!}  T_{r, b}(x)$ (recall \eqref{eq:Tx} above). 
\end{proof}

The above lemma shows that to derive the limit of \eqref{eq:TxI} it suffices to derive the limit of $N^{\frac{R+d+2}{d}} T_{R, 0}(x)$, for $x \in S$ and $R \geq 0$. To this end, note that for $R \geq 0$, 
\begin{align}\label{eq:T}
T_{R, 0}(x)&:=T^{(1)}_{R , 0}(x)+ T^{(2)}_{R , 0}(x) + T^{(3)}_{R, 0}(x) + T^{(4)}_{R , 0}(x),
\end{align} 
where 
\begin{align*}
T^{(1)}_{R , 0}(x)&=f(x|\theta_1)  \int_{S} ||x-y||^{R} \rho_1^{\theta_1, \theta_1}(x, y) \mathrm d y, \\
T^{(2)}_{R , 0}(x)&=\int_{S} ||x-y||^{R} \langle y-x, \grad_{x} f(x|\theta_1)\rangle\rho_1^{\theta_1, \theta_1}(x, y) \mathrm d y, \\
T^{(3)}_{R , 0}(x)&=\frac{1}{2}\int_{S} ||x-y||^{R}  (y-x)^\top \mathrm H_x f(x|\theta_1) (y-x) \rho_1^{\theta_1, \theta_1}(x, y)  \mathrm d y, \\ 
T^{(4)}_{R , 0}(x)&=\frac{1}{6}\int_{S} ||x-y||^{R} D(x, y)\rho_1^{\theta_1, \theta_1}(x, y) ,
\end{align*}
where $D(x, y)=\sum_{i, j, k \in [d]} \frac{\partial^3}{\partial z_i\partial z_j \partial z_k}f(z|\theta_1)\big|_{z=\zeta_{x, y}}(x_i-y_i)(x_j-y_j)(x_k-y_k)$, where $\zeta_{x, y}=c x+(1-c)y$, for some $c \in (0, 1)$. 

The limit of $ T^{(j)}_{R, 0}(x)$, for $j \in \{1, 2, 3\}$ are computed in the following three lemmas.  

\begin{lem}\label{lm:T1} For any $R \geq 0$, 
\begin{align*}
N^{\frac{R+d+2}{d}} \left\{ T^{(1)}_{R, 0}(x)   -  \frac{N^\frac{2}{d} C_{1, R}}{f(x|\theta_1)^\frac{R}{d}} \right \} \rightarrow -\left(1+\frac{R}{d+2}\right) \frac{C_{1, R+2}}{2d} \tr(\mathrm H_x f(x|\theta_1) )   f^{-\frac{R+d+2}{d}}(x|\theta_1),
\end{align*}
uniformly over $x \in S$. 
\end{lem}

\begin{proof} To begin with define, 
\begin{align}\label{eq:T11R0}
T^{(11)}_{R, 0}&(x) \nonumber\\ 
:= &f(x|\theta_1) \int_{S} ||x-y||^{R}  e^{-N V_d ||x-y||^d f(x|\theta_1)} \mathrm d y \\
=&\frac{1}{N^{1+\frac{R}{d}}f(x|\theta_1)^{\frac{R}{d}}} \int_{(Nf(x|\theta_1))^{\frac{1}{d}}(S-x)} ||z||^R e^{-V_d ||z||^d} \mathrm dz \tag*{(substituting $y=x+(N f(x|\theta_1))^{-\frac{1}{d}}z$).}  \nonumber
\end{align}
This implies by Observation \ref{obsintconv}, 
\begin{align}\label{eq:s} 
 \sup_{x \in S} \left| N^{1+\frac{R}{d}}  T^{(11)}_{R, 0}(x) f(x|\theta_1) \mathrm d x- \frac{C_{1, R}}{f(x|\theta_1)^\frac{R}{d}}\right|=O(e^{-\frac{1}{4}N^{\frac{1}{2}}}).
\end{align}

Next, by \eqref{eq:T11R0}
\begin{align}\label{diffnn}
T^{(1)}_{R, 0}(x)-T^{(11)}_{R, 0}(x) = & f(x|\theta_1) \int_{S} ||x-y||^{R} e^{-N ||x-y||^d f(x|\theta_1) V_d} \left(e^{-N R_{x, y}}-1\right)\mathrm dy,
\end{align}
where 
\begin{align}\label{eq:RxyTR0} 
R_{x, y}:=\int_{B_S(x, ||x-y||)} f(z|\theta_1) \mathrm d z -||x-y||^d f(x|\theta_1) V_d=\frac{1}{(d+2)!}||x-y||^{d+2} v_{x, f(\cdot|\theta_1)}^{(d+2)}(\zeta_{x, y}),
\end{align}   
for some $\zeta_{x, y} \in (0, ||x-y||)$ (by Proposition \ref{dd+1}). Then using $|e^{-N R_{x, y}}-1+N R_{x, y}|\lesssim N^2 R_{x, y}^2$
\begin{align} 
& N^{\frac{R+d+2}{d}}  \left|T^{(1)}_{R, 0}(x)-T^{(11)}_{R, 0}(x)+ N f(x|\theta_1) \int_{S} ||x-y||^{R} R_{x, y} e^{-N ||x-y||^d f(x|\theta_1) V_d} \mathrm dy \right| \nonumber \\
&\lesssim_d N^{\frac{R+d+2}{d}} f(x|\theta_1)  \int_{S} ||x-y||^{R} |e^{-N R_{x, y}}-1+N R_{x, y}| e^{-N ||x-y||^d f(x|\theta_1) V_d} \mathrm dy \tag*{(by \eqref{diffnn})} \nonumber \\
& \lesssim_d N^{\frac{R+3d+2}{d}} f(x|\theta_1)  \int_{S} ||x-y||^{R} R_{x, y}^2 e^{-N ||x-y||^d f(x|\theta_1) V_d} \mathrm dy \nonumber \\
 & \lesssim_d N^{\frac{R+3d+2}{d}} f(x|\theta_1) \int_{S} ||x-y||^{R+2d+4} e^{-N ||x-y||^d f(x|\theta_1) V_d} \mathrm dy \tag*{(by \eqref{eq:RxyTR0} and $\sup_{t, x} |v_{x, f(\cdot|\theta_1)}^{(d+2)}(t)|< \infty$)}\nonumber \\
 & = \frac{1}{N^{\frac{2}{d}} f(x|\theta_1)^{\frac{R+2d+4}{d}}} \int_{(Nf(x|\theta_1))^{\frac{1}{d}}(S-x)} ||z||^{\frac{R+2d+4}{d}} e^{-||z||^d V_d} \mathrm dz \tag*{(substituting $y=x+(N f(x|\theta_1))^{-\frac{1}{d}}z$)}  \nonumber \\
\label{diffnnn} & = O\left(\frac{1}{N^{\frac{2}{d}}}\right), 
\end{align}
where the last step uses Observation \ref{obsintconv}. Note that, as before, the constant in the $O(\cdot)$ term does not depend on $x$ (by Assumption \ref{assumptionlocalpower}).

Therefore, it suffices to derive the limit of $N^{\frac{R+d+2}{d}} N f(x|\theta_1) \int_{S} ||x-y||^{R} R_{x, y} e^{-N ||x-y||^d f(x|\theta_1) V_d} \mathrm dy$. To this end, note that $$R_{x, y}= R_{x, y}^{(1)}+ R_{x, y}^{(2)},$$ where $R_{x, y}^{(1)}=\frac{1}{2(d+2)} ||x-y||^{d+2} V_d\tr(\mathrm H_x f(x|\theta_1))$ and $R_{x, y}^{(2)}=\frac{1}{(d+3)!}||x-y||^{d+3} v_{x, f(\cdot|\theta_1)}^{(d+3)}(\zeta'_{x, y})$, for some $\zeta'_{x, y} \in (0, ||x-y||)$ (by Proposition \ref{dd+1}).  Now, observe that 
\begin{align}\label{eq:R1xy}
&N^{\frac{R+2d+2}{d}}  f(x|\theta_1) \int_{S} ||x-y||^R R_{x, y}^{(1)} e^{-N ||x-y||^d f(x|\theta_1) V_d} \mathrm dy \nonumber \\
&= N^{\frac{R+2d+2}{d}} \frac{V_d}{2(d+2)}  f(x|\theta_1) \tr(\mathrm H_x f(x|\theta_1) ) \int_{S} ||x-y||^{R+d+2} e^{-N ||x-y||^d f(x|\theta_1) V_d} \mathrm dy \nonumber\\
 & \rightarrow \frac{V_dC_{1, R+d+2} }{2(d+2)} \tr(\mathrm H_x f(x|\theta_1) )   f^{-\frac{R+d+2}{d}}(x|\theta_1),
\end{align}
where the convergence in the last step is uniformly in $x \in S$. This follows by substituting $y=x+(N f(x|\theta_1))^{-\frac{1}{d}}z$ and then applying Observation \ref{obsintconv}. Similarly, using $\sup_{t, x} |v_{x, f(\cdot|\theta_1)}^{(d+3)}(t)|< \infty$, gives 
\begin{align}
&N^{\frac{R+2d+2}{d}}  f(x|\theta_1) \int_{S} ||x-y||^R |R_{x, y}^{(2)}| e^{-N ||x-y||^d f(x|\theta_1) V_d} \mathrm dy \nonumber \\
& \lesssim_d N^{\frac{R+2d+2}{d}} f(x|\theta_1) \int_{S} ||x-y||^{R+d+3}  e^{-N ||x-y||^d f(x|\theta_1) V_d} \mathrm dy \nonumber\\
& = \frac{1}{N^{\frac{1}{d}}f(x|\theta_1)^{\frac{R+d+3}{d}}} \int_{(Nf(x|\theta_1))^{\frac{1}{d}}(S-x)} ||z||^{R+d+3} e^{-V_d ||z||^d} \mathrm dz \tag*{(substituting $y=x+(N f(x|\theta_1))^{-\frac{1}{d}}z$)} \nonumber \\
\label{eq:R2xy} & = O\left( \frac{1}{N^{\frac{1}{d}}}\right), 
\end{align} 
where the constant in the $O(\cdot)$ term does not depend on $x$ (by Observation \ref{obsintconv} and Assumption \ref{assumptionlocalpower}). Therefore, 
\begin{align}
& N^{\frac{R+d+2}{d}}  (T^{(1)}_{R, 0}(x)-T^{(11)}_{R, 0}(x))  \nonumber \\
&=-N^{\frac{R+d+2}{d}} N  f(x|\theta_1) \int_{S} ||x-y||^{R} R_{x, y} e^{-N ||x-y||^d f(x|\theta_1) V_d} \mathrm dy +O\left(\frac{1}{N^{\frac{2}{d}}}\right) \tag*{(by \eqref{diffnnn})} \nonumber \\ 
&=-N^{\frac{R+2d+2}{d}}  f(x|\theta_1) \int_{S} ||x-y||^{R} (R_{x, y}^{(1)}+ R_{x, y}^{(2)}) e^{-N ||x-y||^d f(x|\theta_1) V_d} \mathrm dy +O\left(\frac{1}{N^{\frac{2}{d}}}\right) \nonumber \\ 
&=-N^{\frac{R+2d+2}{d}}  f(x|\theta_1) \int_{S} ||x-y||^{R} R_{x, y}^{(1)}  e^{-N ||x-y||^d f(x|\theta_1) V_d} \mathrm dy +O\left(\frac{1}{N^{\frac{1}{d}}}\right) \tag*{(by \eqref{eq:R2xy})} \nonumber \\ 
& \rightarrow  - \frac{V_dC_{1, R+d+2} }{2(d+2)} \tr(\mathrm H_x f(x|\theta_1) )   f^{-\frac{R+d+2}{d}}(x|\theta_1) \tag*{(by \eqref{eq:R1xy})} \nonumber \\ 
\label{eq:T1diff}  & = -\left(1+\frac{R}{d+2}\right) \frac{C_{1, R+2}}{2d} \tr(\mathrm H_x f(x|\theta_1) )   f^{-\frac{R+d+2}{d}}(x|\theta_1), 
\end{align}
where the last step uses $C_{1, R+d+2} = \frac{R+d+2}{d V_d} C_{1, R+2}$. Moreover, \eqref{eq:s} implies that $$N^{\frac{R+d+2}{d}}  T^{(11)}_{R, 0}(x)= \frac{N^\frac{2}{d} C_{1, R}}{f(x|\theta_1)^\frac{R}{d}} + O(N^{\frac{2}{d}}e^{-\frac{1}{4}N^{\frac{1}{2}}})=o(1).$$ This combined with  \eqref{eq:T1diff} above implies the lemma. 
\end{proof}

\begin{lem}\label{lm:T2} For any $R \geq 0$, 
\begin{align}\label{eq:T2}
N^{\frac{R+d+2}{d}}  & \sup_{x \in S}|T^{(2)}_{R, 0}(x)| \rightarrow    0.
\end{align}
\end{lem}

\begin{proof} For $x\in S$, substituting $y=x+(N f(x|\theta_1))^{-\frac{1}{d}}z$ gives 
\begin{align}\label{gradzero}
& N^{\frac{R+d+2}{d}} \left|\int_{S}   ||x-y||^{R} \langle y-x, \grad_{x} f(x|\theta_1)\rangle e^{-N||x-y||^d f(x|\theta_1) V_d} \mathrm d y \right|
\nonumber\\
= & \frac{N^{\frac{1}{d}}}{f(x|\theta_1)^{\frac{R+d+1}{d}}}\left|\int_{(Nf(x|\theta_1))^{\frac{1}{d}}(S-x)}  ||t||^R \langle t, \grad_{x} f(x|\theta_1)\rangle e^{-||t||^d  V_d} \mathrm d t \right|\nonumber\\
= & \frac{N^{\frac{1}{d}}}{f(x|\theta_1)^{\frac{R+d+1}{d}}}\left|\int_{\R^d\backslash(Nf(x|\theta_1))^{\frac{1}{d}}(S-x)}  ||t||^R \langle t, \grad_{x} f(x|\theta_1)\rangle e^{-||t||^d  V_d} \mathrm d t \right| \nonumber\\
\lesssim_d & \frac{N^{\frac{1}{d}}e^{-\frac{1}{2}N^\frac{1}{2}} }{f(x|\theta_1)^{\frac{R+d+1}{d}}}=O(e^{-\frac{1}{4}N^\frac{1}{2}}),
\end{align}
where the second equality uses $\int_{\R^d}  ||t||^R \langle t, \grad_{x} f(x|\theta_1)\rangle e^{-||t||^d  V_d} \mathrm d t=0$. 

Now, recalling the definition of $T^{(2)}_{R , 0}(x)$ and $R_{x, y}$ (from \eqref{eq:RxyTR0}) gives 
\begin{align}
& N^{\frac{R+d+2}{d}} |T_{R, 0}^{(2)}(x)| \nonumber\\
\leq & N^{\frac{R+d+2}{d}} \left| \int_{S} ||x-y||^{R} \langle y-x, \grad_{x} f(x|\theta_1)\rangle e^{-N ||x-y||^d f(x|\theta_1) V_d} \left(1-e^{-N R_{x, y}} \right)\mathrm dy \right|+O(e^{-\frac{1}{4}N^\frac{1}{2}})\tag*{(by \eqref{gradzero})}\nonumber\\
\leq & N^{\frac{R+2d+2}{d}} \int_{S} ||x-y||^{R+1} |R_{x, y}| ||\grad_{x} f(x|\theta_1)|| e^{-N ||x-y||^d f(x|\theta_1) V_d} \mathrm dy +O(e^{-\frac{1}{4}N^\frac{1}{2}}) \tag*{(using $1-e^{-x}\leq x$ and the Cauchy-Schwarz inequality)} \nonumber \\
\lesssim_d & N^{\frac{R+2d+2}{d}} \int_{S} ||x-y||^{R+d+3} e^{-N ||x-y||^d f(x|\theta_1) V_d} \mathrm dy + O(e^{-\frac{1}{4}N^\frac{1}{2}}) \tag*{(using $|R_{x, y}| \lesssim_d ||x-y||^{d+2}$)} \nonumber\\
\leq  & \frac{1}{N^{\frac{1}{d}} f(x|\theta_1)^\frac{R+2d+3}{d}}  \int_{(Nf(x|\theta_1))^{\frac{1}{d}}(S-x)} ||z||^{R+d+3} e^{-||z||^d V_d} \mathrm dz + O(e^{-\frac{1}{4}N^\frac{1}{2}}) \tag*{(substituting $y=x+(N f(x|\theta_1))^{-\frac{1}{d}}z$)}\nonumber\\
\label{diffn} =&O(N^{-\frac{1}{d}}).
\end{align}
Note that the third inequality uses $\sup_{x}||\grad_x f(x|\theta_1)||< \infty$ (Assumption \ref{assumptionlocalpower}(c)). 
\end{proof}

\begin{lem}\label{lm:T3} For any $R \geq 0$, 
\begin{align}\label{eq:T3}
N^{\frac{R+d+2}{d}}  & T^{(3)}_{R, 0}(x)  \rightarrow  \frac{C_{1, R+2}}{2d} \tr(\mathrm H_x f(x|\theta_1) ) f^{-\frac{R+d+2}{d}}(x|\theta_1),
\end{align}
uniformly over $x \in S$.
\end{lem}

\begin{proof} Recalling the definition of $T^{(3)}_{R, 0}(x)$ and using the expansion of $R_{x, y}$ as in \eqref{eq:RxyTR0} gives, 
\begin{align}\label{eq:xyVx3}
& N^{\frac{R+d+2}{d}} T^{(3)}_{R, 0}(x) \nonumber \\
&= \frac{N^{\frac{R+d+2}{d}}}{2} \int_{S} ||x-y||^{R}  (y-x)^\top \mathrm H_x f(x|\theta_1) (y-x) \rho_1^{\theta_1, \theta_1}(x, y)  \mathrm d y  \nonumber \\  
&=\frac{N^{\frac{R+d+2}{d}}}{2}  \int_{S} ||x-y||^{R}  (y-x)^\top \mathrm H_x f(x|\theta_1) (y-x) e^{-N ||x-y||^d f(x|\theta_1) V_d}  \mathrm d y  +o(1) \nonumber \\  
 & \rightarrow \frac{1}{2}  \left( \int_{\R^d} ||z||^{R} z^\top \mathrm H_x f(x|\theta_1) z e^{-V_d ||z||^d}\mathrm dz \right)  f^{-\frac{R+d+2}{d}}(x|\theta_1), 
\end{align}
where the last step follows by substituting $y=x+(N f(x|\theta_1))^{-\frac{1}{d}}z$ and applying Observation \ref{obsintconv}. Moreover, as in the proof of Lemma \ref{lm:T1}, the $o(1)$ term in the second step goes to zero uniformly over $x \in S$.

Now, consider the spectral decomposition $\mathrm H_x f(x)=P_x^{\bot} \Lambda(x) P_x$, where $$\Lambda(x)=\diag(\lambda_1(x), \lambda_2(x), \ldots, \lambda_d(x))$$ is the diagonal matrix of eigenvalues of the Hessian matrix $\mathrm H_x f(x)$. Under change of variable $t=P_x z$, \eqref{eq:xyVx3} becomes  
\begin{align*} 
N^{\frac{R+d+2}{d}}  T^{(3)}_{R, 0}(x)  & \rightarrow  \frac{C_{1, R+2}}{2d}  \tr(\mathrm H_x f(x|\theta_1) ) f^{-\frac{R+d+2}{d}}(x|\theta_1),  
\end{align*}
completing the proof of the lemma. 
\end{proof}

%
%

Now, using $\sup_{z}|\frac{\partial^3}{\partial z_i\partial z_j \partial z_k}f(z|\theta_1)|< \infty$ by Assumption \ref{assumptionlocalpower}, gives $|D(x, y)| \lesssim ||x-y||_1^3 \lesssim_d  ||x-y||^3$. This implies, 
$|T^{(4)}_{R , 0}(x)| \lesssim \int_{S} ||x-y||^{r+3} \rho_1^{\theta_1, \theta_1}(x, y)$, and using similar arguments as above it follows that $N^{\frac{R+d+2}{d}} |T^{(4)}_{R, 0}(x)| \rightarrow 0$, uniformly over $x \in S$. Therefore, using Lemma \ref{lm:T1}, \ref{lm:T2}, and \ref{lm:T3} with \eqref{eq:T}, and $C_{1, R+d+2}=\frac{R+d+2}{d V_d} C_{1, R+2}$ (recall \eqref{eq:C1}) gives 
\begin{align}\label{eq:Tn}
N^{\frac{R+d+2}{d}} T_{R, 0}(x)-  \frac{N^\frac{2}{d} C_{1, R}}{f(x|\theta_1)^\frac{R}{d}} \rightarrow  -\frac{R C_{1, R+2} }{2d(d+2)} \tr(\mathrm H_x f(x|\theta_1) )   f^{-\frac{R+d+2}{d}}(x|\theta_1),
\end{align}
uniformly over $x \in S$. 

Now, with $r \geq 0$ as in the statement of Proposition \ref{ppn:dw}, $\int_S w(x) f(x|\theta_1)^{1-\frac{r}{d}}=0$. Then by \eqref{eq:Tn} and the Dominated Convergence Theorem,  
\begin{align}\label{eq:T01}
N^{\frac{r+d+bd+2}{d}} & \int_S w(x)    T_{r+bd, 0}  f(x|\theta_1)^{b+1} \mathrm d x \nonumber \\
& \rightarrow -\frac{(r+bd) C_{1, r+bd+2} }{2d(d+2)} \int_S w(x) \tr(\mathrm H_x f(x|\theta_1) )   f^{-\frac{r+2}{d}}(x|\theta_1) \mathrm d x.
\end{align} 
Next, recalling $\eta_{\theta_1}(x)=\frac{\tr(\mathrm H_x f(x|\theta_1))}{2 (d+2) f(x|\theta_1)}$ and  applying \eqref{eq:mb1}, 
\begin{align}\label{eq:T02}
N^{\frac{r+d+bd+2}{d}} & \int_S w(x) \eta_{\theta_1}(x)   T_{r+bd+2, 0}  f(x|\theta_1)^{b+1} \mathrm d x \nonumber \\ 
& \rightarrow \frac{C_{1, r+bd+2} }{2(d+2)} \int_S w(x) \tr(\mathrm H_x f(x|\theta_1) )   f^{-\frac{r+2}{d}}(x|\theta_1) \mathrm d x.
\end{align}

Then, recalling \eqref{eq:Tx} and \eqref{e1} gives 
\begin{align}
N^{\frac{r+d+2}{d}} & \int_S w(x) T(x) f(x|\theta_1) \mathrm d x  \nonumber \\
&=\sum_{b=0}^{K-1} \frac{1}{b!}  N^{\frac{r+d+bd+2}{d}}\int_S w(x)  T_{r, b}(x) f(x|\theta_1) \mathrm d x \nonumber \\
&=\sum_{b=0}^{K-1} \frac{V_d^b}{b!} N^{\frac{r+d+bd+2}{d}} \int_S w(x)    \left\{T_{r+bd, 0} + b \eta_{\theta_1}(x) T_{r+bd+2, 0}\right\} f(x|\theta_1)^{b+1} \mathrm d x + o(1) \tag*{(by Lemma \ref{lm:e2})} \nonumber \\
&\rightarrow -\frac{r }{2d(d+2)} \sum_{b=0}^{K-1} \frac{V_d^b C_{1, r+bd+2} }{b!} \int_S  w(x) \tr(\mathrm H_x f(x|\theta_1) )   f^{-\frac{r+2}{d}}(x|\theta_1) \mathrm d x \tag*{(by \eqref{eq:T01} and \eqref{eq:T02})} \nonumber \\ 
&= -\frac{r  C_{K, r+2} }{2d(d+2)} \int_S  w(x) \tr(\mathrm H_x f(x|\theta_1) )   f^{-\frac{r+2}{d}}(x|\theta_1) \mathrm d x, \nonumber 
\end{align}
where the last step uses \eqref{cd}. This completes the proof of Proposition \ref{ppn:dw}. \qed

\subsubsection{\textbf{Proof of Lemma \ref{t11limit}}}
\label{sec:pft11limit}

To prove Lemma \ref{t11limit} we need another result about `degenerate' functionals similar to Proposition \ref{ppn:dw}. 

\begin{ppn}\label{t1ppn}
Let $w: S \rightarrow \R$ be a twice continuously differentiable function such that $\int_S w(y) \mathrm dy=0$. Then  
\begin{align*}
N^{\frac{d+2}{d}}& \int_{S\times S} w(y)  \rho^{\theta_1, \theta_1}_K(x, y) f(x|\theta_1)  \mathrm d x \mathrm d y \nonumber \\
&\rightarrow \frac{C_{K, 2}}{2d} \int_{S}  \left(\frac{ \tr(\mathrm H_x w(x)) f(x|\theta_1)-w(x)  \tr(\mathrm H_x f(x|\theta_1))}{f^2(x|\theta_1)}  \right) f^{\frac{d-2}{d}}(x|\theta_1) \mathrm d x,
\end{align*}
where $\mathrm H_x w(x) \in \R^{d\times d}$ is the Hessian matrix of $w$ at $x$.
\end{ppn}

The proof of Proposition \ref{t1ppn} is given below. Here, we show how it can used to complete the proof of Lemma~\ref{t11limit}: Recall from \eqref{gradppn} that
\begin{align*}
\frac{T_1}{\sqrt N}&=\frac{1}{N^{1-\frac{2}{d}}} \E  \sum_{x, y \in \cP_{Nf(\cdot|\theta_1)}} \frac{h^\top \grad_{\theta_1}  f(y|\theta_1)}{f(y|\theta_1)} \bm 1\{(x, y)\in E(\cN_K(\cP_{Nf(\cdot|\theta_1)}))\}  \\
&= N^{\frac{d+2}{d}}\int_{S \times S} h^\top \grad_{\theta_1}  f(y|\theta_1)   \rho_K^{\theta_1, \theta_1}(x, y)   f(x|\theta_1) \mathrm d x \mathrm d y \\ 
& \rightarrow \frac{C_{K, 2}}{2d} \int_{S}  \left(\frac{ \tr(\mathrm H_x h^\top \grad_{\theta_1}  f(x|\theta_1)) f(x|\theta_1)-h^\top \grad_{\theta_1}  f(x|\theta_1)  \tr(\mathrm H_x f(x|\theta_1))}{f^2(x|\theta_1)} \right)  f^{\frac{d-2}{d}}(x|\theta_1) \mathrm d x  \tag*{(by Proposition~\ref{t1ppn} with $w(y)=h^\top \grad_{\theta_1}  f(y|\theta_1)$)} \\
& = \frac{C_{K, 2}}{2d} \int_{S}  h^\top \grad_{\theta_1} \left(\frac{ \tr(\mathrm H_x f(x|\theta_1))}{f(x|\theta_1)}  \right) f^{\frac{d-2}{d}}(x|\theta_1) \mathrm d x,
\end{align*}
where the last step uses the product rule of differentiation. \\

\noindent\textit{Proof of Proposition}~\ref{t1ppn}: By Assumption~\ref{assumptionlocalpower}(b) all the points in $\cP_{Nf(\cdot|\theta_1)}$ are in the interior of the support $S$ of $f(\cdot|\theta_1)$ with probability 1. Now, by a Taylor-series expansion of $w$ around $x$ and arguments similar to the proof of Proposition \ref{ppn:dw} it can be shown that 
$$\int_{S\times S} w(y)  \rho^{\theta_1, \theta_1}_K(x, y)  f(x|\theta_1)  \mathrm d x \mathrm d y:=J_1+J_2+J_3+o(N^{-\frac{d+2}{d}}),$$
where 
\begin{align*}
J_1&:=\int_{S\times S} w(x) \rho^{\theta_1, \theta_1}_K(x, y) f(x|\theta_1) \mathrm d x \mathrm dy,  \\ 
J_2&:=\int_{S\times S}  (y-x)^\top \grad_x w(x) \rho^{\theta_1, \theta_1}_K(x, y) f(x|\theta_1) \mathrm d x \mathrm dy,  \\  
J_3&:=\frac{1}{2}\int_{S\times S}   (y-x)^\top \mathrm H_x w(x) (y-x)  \rho^{\theta_1, \theta_1}_K(x, y) f(x|\theta_1) \mathrm d x \mathrm dy. 
\end{align*}

Again, by arguments similar to Proposition \ref{ppn:dw},  it follows that 
\begin{align}\label{eq:wyJ1}
N^{\frac{d+2}{d}} J_1 \rightarrow -\frac{C_{K, 2}}{2d} \int_{S} w(x)  \tr(\mathrm H_x f(x|\theta_1) )   f^{-\frac{d+2}{d}}(x|\theta_1) \mathrm d x,
\end{align}
and $N^{\frac{d+2}{d}}J_2=o(1)$.
%
%
Finally, by \cite[Corollary 8.1]{yukichclt}, 
\begin{align}\label{eq:wyJ3}
N^{\frac{d+2}{d}} J_3&= N^{\frac{d+2}{d}} \cdot \frac{1}{2} \int_{S \times S} (y-x)^\top \mathrm H_x w(x) (y-x) \rho^{\theta_1, \theta_1}_K(x, y) f(x|\theta_1) \mathrm d x \mathrm d y \nonumber \\
& \rightarrow \frac{1}{2} \sum_{b=0}^{K-1} \frac{V_d^b}{b!}\int_S \left(\int_{\R^d} ||z||^{bd} z^\top \mathrm H_x w(x) z e^{-V_d ||z||^{d}}\mathrm dz \right) f^{-\frac{2}{d}}(x|\theta_1) \mathrm d x.
\end{align} 
Now, consider the spectral decomposition $\mathrm H_x w(x)=P_x^{\bot} \Lambda(x) P_x$, where $\Lambda(x)=\diag(\lambda_1(x), \lambda_2(x), \ldots, \lambda_d)$ is the diagonal matrix of eigenvalues of the Hessian matrix $\mathrm H_x w(x)$. Under change of variable $t=P_x z$, \eqref{eq:wyJ3} becomes  
\begin{align}
N^{\frac{d+2}{d}} J_3 & \rightarrow \frac{1}{2} \sum_{b=0}^{K-1} \frac{V_d^b}{b!} \int_S \left(\int_{\R^d} ||t||^{bd}e^{-V_d ||t||^{d}}\mathrm dt  \sum_{i=1}^d \lambda_i(x)t_i^2 \right) f^{-\frac{2}{d}}(x|\theta_1) \mathrm d x \nonumber \\ 
& = \frac{1}{2} \sum_{b=0}^{K-1} \frac{V_d^b}{b! d} \int_S \left(\int_{\R^d} ||t||^{bd+2}e^{-V_d ||t||^{d}}\mathrm dt  \right) \tr(\mathrm H_x w(x)) f^{-\frac{2}{d}}(x|\theta_1) \mathrm d x \nonumber \\ 
& = \frac{1}{2} \sum_{b=0}^{K-1} \frac{V_d^b C_{1, bd+2}}{b! d} \int_S \tr(\mathrm H_x w(x)) f^{-\frac{2}{d}}(x|\theta_1) \mathrm d x \nonumber \\ 
& = \frac{C_{K, 2}}{2 d} \int_S \tr(\mathrm H_x w(x)) f^{-\frac{2}{d}}(x|\theta_1) \mathrm d x, 
\end{align}
since $\tr(\mathrm H_x w(x))= \sum_{i=1}^d \lambda_i(x)$.

Finally, combining \eqref{eq:wyJ1} and \eqref{eq:wyJ3} gives 
\begin{align*}
N^{\frac{d+2}{d}}& \int_{S\times S} w(y) \rho_K^{\theta_1, \theta_1}(x, y) f(x|\theta_1)  \mathrm d x \mathrm d y \\
& \rightarrow \frac{C_{K, 2}}{2d} \int_{S}  \left(\frac{ \tr(\mathrm H_x w(x)) f(x|\theta_1)-w(x)  \tr(\mathrm H_x f(x|\theta_1))}{f^2(x|\theta_1)}  \right) f^{\frac{d-2}{d}}(x|\theta_1) \mathrm d x.
\end{align*}
This completes the proof of the result.

\subsection{Proof of Lemma~\ref{hessianppn}: Limit of the Hessian term}
\label{hessianlimit}

Recall the definition of $\mu_N(\theta_1, \theta_1)$ from~\eqref{beta0}. Differentiating with respect to $\theta_2$ twice under the integral signs gives,
\begin{align}\label{hessiant}
\mathrm H  \mu_N (\theta_1,  \theta_1) = T_{21}+T_{22}+T_{23},
\end{align}
where 
\begin{align*}
T_{21}&=N^2\int_{S\times S} f(x|\theta_1)  \mathrm H_{\theta_1} f(y|\theta_1)  \rho^{\theta_1, \theta_1}_{K}(x, y) \mathrm d x \mathrm d y, \nonumber\\
T_{22}&= 2 N^2 \int_{S\times S}  f(x|\theta_1)\grad_{\theta_1} f(y|\theta_1) \grad_{\theta_1} \rho^{\theta_1, \theta_1}_{K}(x, y) \mathrm d x \mathrm d y,  \nonumber\\
T_{23}&= N^2 \int_{S\times S} f(x|\theta_1) f(y|\theta_1)  \mathrm H_{\theta_1}  \rho^{\theta_1, \theta_1}_{K}(x, y)   \mathrm d x \mathrm d y. 
\end{align*}
The limits of these three terms are given in Lemma \ref{hessianlimitI}, \ref{hessianlimitII}, and \ref{hessianlimitIII}, respectively. Lemma~\ref{hessianppn} follows from these lemmas and noting that $C_{1, Kd}=\frac{1}{V_d^K}K!$ (recall \eqref{cd}).

%
%
%
%

\begin{lem}\label{hessianlimitI} Let $\varepsilon_N=h N^{-\frac{1}{4}}$, for some $h \in \R^p\backslash \{\bm 0\}$. Then $\frac{1}{\sqrt N}  \varepsilon_N^\top T_{21} \varepsilon_N  \rightarrow 0$. \end{lem}

\begin{proof} 
By Lemma~\ref{lem:hzz}, 
\begin{align*}
\frac{1}{\sqrt N} T_{21} =&\frac{1}{N}\E\sum_{x, y\in \cP_{Nf(\cdot|\theta_1)}} \frac{ h^\top \mathrm H_{\theta_1} f(y|\theta_1)  h}{f(y|\theta_1)}\bm 1\{(x, y)\in E(\cN_K(\cP_{Nf(\cdot|\theta_1)}))\}\\
\rightarrow & K \int_S h^\top \mathrm H_{\theta_1} f(x|\theta_1)h \mathrm dx=0,
\end{align*}
since 
\begin{align}
\int_S h^\top \mathrm H_{\theta_1}  f(x|\theta_1)h \mathrm dx=& \int _S \sum_{1\leq i, j \leq p } \frac{\partial^2}{\partial \theta_{1i} \theta_{1j}}  f(x|\theta_{1})h_i h_i\nonumber\\
=&\sum_{1\leq i, j \leq p }\frac{\partial^2}{\partial \theta_{1i} \partial \theta_{1j}} \int _S f(x|\theta_{1})h_i h_i=0.\nonumber
\end{align}
This completes the proof of the lemma. 
\end{proof}

\begin{lem}\label{hessianlimitII} Let $\varepsilon_N=h N^{-\frac{1}{4}}$, for some $h \in \R^p\backslash \{\bm 0\}$. Then $$\frac{1}{\sqrt N}  \varepsilon_N^\top T_{22} \varepsilon_N  \rightarrow -2 q \frac{V_d^{K} C_{1, Kd}}{(K-1)!} \E \left(\frac{h^\top\grad_{\theta_1} f(X|\theta_1)}{f(X|\theta_1)} \right)^2.$$  \end{lem}

\begin{proof} To begin with define 
\begin{align}\label{t22k}
T_{22}^{(K)}&= -2 N_2 N^2 \int_{S\times S}  f(x|\theta_1)\grad_{\theta_1} f(y|\theta_1) \left(\int_{B_S(x, ||x-y||)} \grad_{\theta_1}  f(z|\theta_1) \mathrm d z\right)\rho^{\theta_1, \theta_1}_{K}(x, y) \mathrm d x \mathrm d y, 
\end{align}
and note that $T_{22}=T_{22}^{(K)}-T_{22}^{(K-1)}$. Therefore, it suffices to derive the limit of 
$\frac{1}{\sqrt N} \cdot \varepsilon_N^\top T_{22}^{(K)} \varepsilon_N$. To this end,  observe that 
\begin{align}\label{eq:hf2}
&N^2\int_{S\times S} f(x|\theta_1)f(y|\theta_1)\frac{h^\top \grad_{\theta_1}  f(y|\theta_1) h^\top\grad_{\theta_1}  f(x|\theta_1)}{f(y|\theta_1)}||x-y||^d \rho_K^{\theta_1, \theta_1}(x, y) \mathrm d x \mathrm d y\nonumber\\
=& \E \sum_{(x, y)\in \cP_{Nf(\cdot|\theta_1)}}\frac{h^\top \grad_{\theta_1}  f(y|\theta_1) h^\top\grad_{\theta_1}  f(x|\theta_1)}{f(y|\theta_1)}||x-y||^d \bm 1\{(x, y)\in E(\cN_K(\cP_{Nf(\cdot|\theta_1)}))\}\nonumber\\
\rightarrow &C_{K, d}\int_{S} \frac{\left(h^\top \grad_{\theta_1}  f(x|\theta_1)\right)^2}{f(x|\theta_1)} \mathrm d x \nonumber\\
=&C_{K, d} \E \left(\frac{h^\top\grad_{\theta_1} f(X|\theta_1)}{f(X|\theta_1)} \right)^2.
\end{align}

Now, using (by Proposition \ref{dd+1}) 
\begin{align}\label{eq:vexpH} 
v_{x, \theta_1}(r, b):=  \int_{B_S(x, r)}  b^\top\grad_{\theta_2}  f(z|\theta_2)\mathrm d z=V_d r^d b^\top\grad_{\theta_2}  f(x|\theta_2)+O(r^{d+1}),
\end{align}
where the constant in the $O(\cdot)$ term does not depend on $x$, and \eqref{eq:hf2}, in \eqref{t22k}, we get 
$$\frac{1}{\sqrt N} \varepsilon_N^\top  T_{22}^{(K)} \varepsilon_N \rightarrow -2 q V_d C_{K, d} \E \left(\frac{h^\top\grad_{\theta_1} f(X|\theta_1)}{f(X|\theta_1)} \right)^2.$$
This implies $$\varepsilon_N^\top T_{22} \varepsilon_N =\varepsilon_N^\top T_{22}^{(K)} \varepsilon_N - \varepsilon_N^\top T_{22}^{(K-1)} \varepsilon_N =-2 q \frac{V_d^{K} C_{1, Kd}}{(K-1)!} \E \left(\frac{h^\top\grad_{\theta_1} f(X|\theta_1)}{f(X|\theta_1)} \right)^2,$$ (using $C_{K, d}=\sum_{b=0}^{K-1} \frac{V_d^{b} C_{1, (b+1)d}}{b!}$) proving the lemma.
\end{proof}

\begin{lem}\label{hessianlimitIII} Let $\varepsilon_N=h N^{-\frac{1}{4}}$, for some $h \in \R^p\backslash \{\bm 0\}$. Then $$\frac{1}{\sqrt N}  \varepsilon_N^\top T_{23} \varepsilon_N  \rightarrow \frac{2q^2 V_d^{K} C_{1, Kd}}{(K-1)!}  \E \left(\frac{h^\top\grad_{\theta_1} f(X|\theta_1)}{f(X|\theta_1)} \right)^2. $$ \end{lem}

\begin{proof}Define 
\begin{align}\label{t231}
T_{231}^{(K)}&= -N_2 N^2 \int_{S\times S}  f(x|\theta_1) f(y|\theta_1) \left(\int_{B_S(x, ||x-y||)} \mathrm H_{\theta_1}  f(z|\theta_1) \mathrm d z\right)\rho^{\theta_1, \theta_1}_{K}(x, y) \mathrm d x \mathrm d y, 
\end{align}
and 
\begin{align}\label{t231}
T_{232}^{(K)}&= N_2^2 N^2 \int_{S\times S}  f(x|\theta_1) f(y|\theta_1) \left(\int_{B_S(x, ||x-y||)} \grad_{\theta_1}  f(z|\theta_1) \mathrm d z\right)^2 \rho^{\theta_1, \theta_1}_{K}(x, y) \mathrm d x \mathrm d y. 
\end{align} 
Note that $T_{23}=T_{231}^{(K)}-T_{231}^{(K-1)}+T_{232}^{(K)}-T_{232}^{(K-2)}$. 

To begin with, consider 
\begin{align}
\frac{1}{\sqrt N} \varepsilon_N^\top T_{231}^{(K)}\varepsilon_N=& \frac{N_2}{N}  \E\sum_{x, y \in \cP_{N f(\cdot| \theta_1)}} \left(\int_{B_S(x, ||x-y||)} h^\top \mathrm H_{\theta_1}  f(z|\theta_1) h\mathrm d z\right) \bm 1\{(x, y)\in E(\cN_K(\cP_{Nf(\cdot| \theta_1)}))\}\nonumber\\
=& \frac{N_2}{N} \E\sum_{x, y \in \cP_{N f(\cdot| \theta_1)}}  h^\top \mathrm H_{\theta_1}  f(x|\theta_1) h ||x-y||^d \bm 1\{(x, y)\in E(\cN_K(\cP_{N f(\cdot| \theta_1)}))\}+o(1) \tag*{(using \eqref{eq:vexpH})}\nonumber\\
\label{eq:t231KI} \rightarrow &  q V_d C_{K, d} \int_S h^\top \mathrm H_{\theta_1} f(x|\theta_1)h \mathrm dx =0,
\end{align}
by \eqref{eq:mb1}.
 
Again, using \eqref{eq:vexpH} and Lemma \ref{eq:mb1} gives, 
\begin{align}
\frac{1}{\sqrt N} \varepsilon_N^\top T_{232}^{(K)}\varepsilon_N=&\left(\frac{N_2}{N}\right)^2 N \E\sum_{x, y \in \cP_{N f(\cdot| \theta_1)}}v_{x, \theta_1}^2(||x-y||, h)\bm 1\{(x, y)\in E(\cN_K(\cP_{N f(\cdot| \theta_1)}))\}\nonumber\\
=& \left(\frac{N_2}{N}\right)^2 N \E\sum_{(x, y)\in E(\cN_K(\cP_{N f(\cdot| \theta_1)}))} V_d^2 \left( h^\top\grad_{\theta_2}  f(x|\theta_2)\right)^2 ||x-y||^{2d}  +O(N^{-\frac{1}{d}})\nonumber\\
\rightarrow & q^2 V_d^2 C_{K, 2d} \int \frac{\left( h^\top\grad_{\theta_1}  f(x|\theta_1)\right)^2}{f(x|\theta_1)}\mathrm dx\nonumber\\ 
=& q^2 V_d^2 C_{K, 2d}  \E \left(\frac{h^\top\grad_{\theta_1} f(X|\theta_1)}{f(X|\theta_1)} \right)^2. \nonumber 
\end{align}
Now, using $C_{K, 2d}=\sum_{b=0}^{K-1} \frac{V_d^{b} C_{1, (b+2)d}}{b!}$, gives
\begin{align}\label{eq:t232diff}
\frac{1}{\sqrt N} \left(T_{232}^{(K)}-T_{232}^{(K-2)}\right)&= \frac{q^2 V_d^{K}}{(K-2)!} \left[\frac{V_d C_{1, (K+1)d}}{(K-1)}-C_{1, Kd} \right]  \E \left(\frac{h^\top\grad_{\theta_1} f(X|\theta_1)}{f(X|\theta_1)} \right)^2 \nonumber \\ 
&=\frac{2q^2 V_d^{K} C_{1, Kd}}{(K-1)!}  \E \left(\frac{h^\top\grad_{\theta_1} f(X|\theta_1)}{f(X|\theta_1)} \right)^2,
\end{align}
where the last step uses \eqref{eq:C1}.

The lemma follows from \eqref{eq:t231KI}, \eqref{eq:t232diff}, and noting that $T_{23}=T_{231}^{(K)}-T_{231}^{(K-1)}+T_{232}^{(K)}-T_{232}^{(K-2)}$. 
\end{proof}

\subsection{Proof of Lemma~\ref{remainderppn}}
\label{remainderlimit}

Recall the definition of the error term $\cR_N$ from \eqref{errorterm}:  
\begin{align*}
\cR_N=\frac{1}{6\sqrt N}\sum_{1 \leq a, b, c \leq p}\varepsilon_{N_a}\varepsilon_{N_b}\varepsilon_{N_c}\frac{\partial^3\mu_N(\theta_1, \theta)}{\partial \theta_{a}\partial \theta_{b}\partial \theta_{c}}\Big|_{\theta=c \theta_1+(1-c) \theta_2},
\end{align*} for some $c \in (0, 1)$, where $\theta=(\theta_{1}, \theta_{2}, \ldots, \theta_{p})'$ and $\varepsilon_N=(\varepsilon_{N_1}, \varepsilon_{N_2}, \ldots, \varepsilon_{N_p})'$. The third derivatives of $\mu_N(\theta_1, \theta)$ (with respect to $\theta$) can calculated by taking three derivatives under the integral sign in~\eqref{beta0}. This leads to many terms, all of which can be bounded in a similar manner. In the following, this is illustrated for one such term $T_0$: 
\begin{align}\label{ENI}
T_0 :=&\frac{N^2N_2^3}{\sqrt N}\int f(x|\theta_1) f(y|\theta)  \left( \int_{B_S(x, ||x-y||)}  \varepsilon_N^\top\grad_{\theta}  f(z|\theta)\mathrm d z\right)^3 \rho^{\theta_1, \theta}_{K}(x, y) \mathrm d x \mathrm d y.
\end{align}

Now, using $\varepsilon_N=hN^{-\delta_d}$, $K:=\sup_{z \in \R^d, \theta \in \R^p} ||\grad_{\theta}  f(z|\theta)||< \infty$ (by Assumption \eqref{assumptionlocalpower}), and the Cauchy-Schwarz inequality, 
\begin{align*}
|T_0| \leq & K^3 ||h||^3 \frac{N^2N_2^3}{N^{\frac{1}{2}+3 \delta_d}}\int_{S\times S} f(x|\theta_1) f(y|\theta)  ||x-y||^{3d} \rho^{\theta_1, \theta}_{K}(x, y) \mathrm d x \mathrm d y \tag*{(using $|\varepsilon_N^\top\grad_{\theta}  f(z|\theta))|\leq ||\varepsilon_N||\cdot ||\grad_{\theta}  f(z|\theta)||$}\\
\leq & K^3  ||h||^3 N^{\frac{9}{2}-3 \delta_d}  \int_{S\times S} f(x|\theta_1) f(y|\theta)  ||x-y||^{3d} \rho^{\theta_1, \theta}_{K}(x, y) \mathrm d x \mathrm d y \\
\leq & O( ||h||^3 N^{\frac{1}{2}-3 \delta_d}),
\end{align*}
by \eqref{eq:deglength}. 

The other terms in the third derivative can also be bounded similarly. This implies that $|\cR_N|=O(||h||^3 N^{\frac{1}{2}-3 \delta_d})$.

\section{Normal location} 
\label{sec:pfnlocation}

In this section we complete the calculations in Section \ref{sec:nlocation}. Recall that $A \in \R^d$ is a compact and convex set which symmetric around the origin $\bm 0 \in \R^d$, and, for $\theta \in \R^d$,  
$$\phi_A(x|\theta)=\frac{1}{Z_A(\theta)}e^{-\frac{1}{2} ||x-\theta||^2},$$
where $Z_A(\theta):=\int_A e^{-\frac{1}{2}||x-\theta||^2}\mathrm d x$, is the normalizing constant. We are considering the problem of testing \eqref{epsilonN} based on \eqref{rejregionKNN}, given i.i.d. samples $\sX_{N_1}$ and $\sY_{N_2}$ from $\phi_A(\cdot|\theta_1)$ and $\phi_A(\cdot|\theta_2)$, respectively. There are two cases depending whether the true $\theta_1$ is zero or non-zero.

\subsection{$\theta_1 \ne \textbf{0}$}

For dimension $d \leq 7$, by Theorem \ref{EFFSECOND}, depending on whether $||N^{\frac{1}{4}}\varepsilon_N|| \rightarrow 0$,  $||N^{\frac{1}{4}}\varepsilon_N|| \rightarrow h$, or $||N^{\frac{1}{4}}\varepsilon_N|| \rightarrow \infty$, the limiting power of the test is $0$, 
$$\Phi\left(z_\alpha+\frac{r^2 K}{2 \sigma_K}   \E_{X \sim \phi_A(\cdot|\theta_1)} \left[ \frac{h^\top  \grad_{\theta_1} \phi_A(X|\theta_1)}{\phi_A(X|\theta_1)} \right]^2 \right),$$ 
or $1$, respectively. 

Next, suppose $d\geq 9$. In this case, 
$$\frac{ \tr(\mathrm H_x \phi_A(x|\theta))}{\phi_A(x|\theta)}=||x-\theta||^2-d \quad \text{and} \quad h^\top \grad_{\theta_1} \left(\frac{ \tr(\mathrm H_x \phi_A(x|\theta_1))}{\phi_A(x|\theta_1)}\right)= 2 h^\top (\theta_1-x),$$ and  $\int_{S}  h^\top (\theta_1-x) \phi^{\frac{d-2}{d}}_A(x|\theta_1)\ne 0$, whenever $\theta_1 \ne \textbf{0}$. Therefore, by Theorem \ref{EFFSECOND}, the following cases arise: 

\begin{itemize}

\item[--] $||N^{\frac{1}{2}-\frac{2}{d}}\varepsilon_N||\rightarrow 0:$ The limiting power of the test~\eqref{rejregionKNN} is $\alpha$. 

\item[--]  $N^{\frac{1}{2}-\frac{2}{d}}\varepsilon_N \rightarrow h:$ The limiting power of the test~\eqref{rejregionKNN} is 
$$\Phi\left(z_\alpha -  \frac{rp C_{K, 2}}{2d\sigma_K}   \int_{S}  h^\top (\theta_1-x) \phi^{\frac{d-2}{d}}_A(x|\theta_1) \mathrm d x \right).$$

\item[--]  $||N^{\frac{1}{2}-\frac{2}{d}}\varepsilon_N||\rightarrow \infty$  such that $||N^{\frac{2}{d}}\varepsilon_N|| \rightarrow 0:$ Then depending on whether $\int_{S}  \varepsilon^\top (\theta_1-x) \phi^{\frac{d-2}{d}}_A(x|\theta_1) \mathrm d x$ is positive or negative, the limiting power of the test~\eqref{rejregionKNN} is $0$ or $1$, respectively.

\item[--] $||N^{\frac{2}{d}}\varepsilon_N|| \rightarrow \infty:$ The limiting power of the test~\eqref{rejregionKNN} is $1$. 

\end{itemize}

Finally, suppose dimension $d=8$. Again by Theorem \ref{EFFSECOND}, depending on whether $||N^{\frac{1}{4}}\varepsilon_N|| \rightarrow 0$ or $||N^{\frac{1}{4}}\varepsilon_N|| \rightarrow \infty$, the limiting power of the test is $0$ or $1$, respectively. At the threshold,  $||N^{\frac{1}{4}}\varepsilon_N|| \rightarrow h$, both the gradient and the Hessian term contribute, and the limiting power is 
$$\Phi\left(z_\alpha -  \frac{r p C_{K, 2}}{2d\sigma_K}   \int_{S}  h^\top (\theta_1-x) \phi^{\frac{d-2}{d}}_A(x|\theta_1) \mathrm d x +\frac{r^2 K}{2 \sigma_K}   \E_{X \sim \phi_A(\cdot|\theta_1)} \left[ \frac{h^\top  \grad_{\theta_1} \phi_A(X|\theta_1)}{\phi_A(X|\theta_1)} \right]^2 \right),$$ 
as predicted by \eqref{eq:effgradhessian}.

\subsection{$\theta_1= \textbf{0}$} In this case, $\int_{A}  h^\top (\theta_1-x) \phi^{\frac{d-2}{d}}_A(x|\bm 0)= 0$, and Theorem \ref{EFFSECOND} cannot be directly used to determine threshold for local power, for $d\geq 9$. However, when $\theta_1= \textbf{0}$, a direct calculation shows that the gradient term in \eqref{D} is exactly zero. To this end, define $\theta_A=\int_A x e^{-\frac{1}{2}||x-\theta||^2}\mathrm d x$. Then the gradient of the normalizing constant is $\grad_{\theta}Z_A(\theta)=\int_A (x- \theta) e^{-\frac{1}{2}||x-\theta||^2}\mathrm d x=\theta_A-\theta$. Therefore, 
\begin{align}\label{gradnl}
\grad_{\theta} \phi_A(x|\theta)=\left\{(x -\theta) -\frac{\theta_A-\theta}{Z_A(\theta)}\right\} \phi_A(x|\theta).
\end{align}
By symmetry of $A$, when $\theta=\bm 0$, then $\theta_A=\bm 0$, which implies $\grad_{\theta=\bm 0} \phi_A(x|\theta)=x \phi_A(x|\bm 0)$. In this case, the gradient term \eqref{grad1} is exactly zero, as shown in the following lemma:

\begin{lem}\label{lm:theta0} Let $\phi_A$ be as above. Then for any $\varepsilon \in \R^d$
\begin{enumerate}[(a)]
\item[(a)] $\int_{A\times A}  \phi_A(x|\bm 0) \varepsilon^\top\grad_{\theta=\bm 0} \phi_A(x|\theta) \rho_K^{\bm 0, \bm 0}(x, y)\mathrm dx\mathrm dy=0$.

\item[(b)] $\int_{A\times A}   \phi_A(x|\bm 0) \phi_A(y|\bm 0) \left(\int_{B_A(x, ||x-y||)} \varepsilon^\top  \grad_{\theta=\bm 0} \phi_A(z|\theta) \mathrm dz \right) \rho_K^{\bm 0, \bm 0}(x, y)\mathrm dx\mathrm dy \mathrm dx\mathrm dy=0$. 
\end{enumerate}

\end{lem}

\begin{proof} Note that 
\begin{align}\label{T1nl}
\int_{A\times A}  \phi_A(x|\bm 0) & \varepsilon^\top \grad_{\theta=\bm 0} \phi_A(y|\theta)  \rho_K^{\bm 0, \bm 0}(x, y)\mathrm dx\mathrm dy\nonumber\\
&=\frac{1}{N} \int_A  \varepsilon^\top y    \E(d^\downarrow (y, \cN_K(\cP^y_{N\phi_A(\cdot|\bm 0)}))) \mathrm dy=0,
\end{align}
since $A=-A$ and $\E(d^\downarrow (y, \cN_K(\cP^y_{N\phi_A(\cdot|\bm 0)})))=\E(d^\downarrow (-y, \cN_K(\cP^{-y}_{N\phi_A(\cdot|\bm 0)})))$, by symmetry.

Next, let $v_{x, \theta_1}(r, \varepsilon):=\int_{B_A(x, r)} \varepsilon^\top  \grad_{\theta=\bm 0} \phi_A(z|\theta) \mathrm dz$. By symmetry, for any $r>0$, $$v_{x, \theta_1}(r,  \varepsilon)=\int_{B_A(x, r)}  \varepsilon^\top  z \phi_A(z|\bm 0)\mathrm dz=-v_{-x, \theta_1}(r,  \varepsilon),$$ which implies (b). 
\end{proof}

The above lemma and Lemmas~\ref{hessianppn} and  \ref{remainderppn} implies that, for $\varepsilon_N=h N^{-\frac{1}{4}}$, $$\frac{1}{\sqrt N} \bigg\{\mu_N(\bm 0, \varepsilon_N) -\mu_N(\bm 0, \bm 0)\bigg\}  \rightarrow \frac{r^2 K}{2 \sigma_K}\E_{X \sim \phi_A(\cdot|\bm 0)}\left( h^\top X \right)^2,$$ 
since $\grad_{\theta=\bm 0} \log Z_A(\theta)=\frac{\grad_{\theta=\bm 0}Z_A(\theta)}{Z_A(\theta)}=0$. This implies that the power of the test \eqref{rejregionKNN} is 
$$\Phi\left(z_\alpha+\frac{r^2 K}{2 \sigma_K}   \E_{X \sim \phi_A(\cdot|\bm 0)}\left( h^\top X \right)^2 \right).$$
This also shows that the limiting power of the test is $\alpha$ or 1, depending on whether $||N^{\frac{1}{4}}\varepsilon_N|| \rightarrow 0$ or $||N^{\frac{1}{4}}\varepsilon_N|| \rightarrow \infty$,  respectively.

\section{Spherical Normal}
\label{sec:pfsnormal}

Here, we complete the calculations in Section \ref{sec:snormal}.  Let $M$ be a compact, convex subset of $\R^d$. Recall, for $\lambda >0$, the family of densities $\phi_M(\cdot|\lambda^2)$: 
$$\phi_M(x|\lambda^2)=\frac{1}{Z_M(\lambda^2)}e^{-\frac{1}{2 \lambda^2} ||x||^2}, \quad \text{for } x \in M$$
where $Z_M(\lambda^2):=\int_M e^{-\frac{1}{2\lambda^2}||x||^2}\mathrm d x$, is the normalizing constant. Consider the problem of testing \eqref{epsilonN} based on \eqref{rejregionKNN}, given i.i.d. samples $\sX_{N_1}$ and $\sY_{N_2}$ from $\phi_M(\cdot|\lambda_1^2)$ and $\phi_M(\cdot|\lambda_2^2)$.

For dimension $d \leq 8$, by Theorem \ref{EFFSECOND}, depending on whether $N^{\frac{1}{4}}\varepsilon_N \rightarrow 0$,  $N^{\frac{1}{4}}\varepsilon_N \rightarrow h$, or $N^{\frac{1}{4}}\varepsilon_N \rightarrow \infty$, the limiting power of the test is $0$,  \eqref{eq:effhessian} or \eqref{eq:effgradhessian} (depending on whether $d \leq 7$ or $d=8$) or $1$, respectively.

Next, suppose $d\geq 9$. In this case, $\grad_{\lambda_1} \frac{ \tr(\mathrm H_x \phi_M(x|\lambda^2))}{\phi_M(x|\lambda^2)}=-\frac{4 x^\top x-2\lambda_1 ^2 d}{\lambda_1^5}$,  
and 
\begin{align}
\int_{M} h \cdot \grad_{\lambda_1} \left(\frac{ \tr(\mathrm H_x \phi_M(x|\lambda^2_1))}{\phi_M(x|\lambda^2_1)}\right) \phi^{\frac{d-2}{d}}_M(x|\lambda^2_1) \mathrm d x & = -\frac{Z_M(\frac{d\lambda_1^2}{d-2} )}{Z_M(\lambda_1^2)^{\frac{d-2}{d}}}\frac{2 h}{\lambda_1^5} \int_{M} \left(2 x^\top x-\lambda_1^2 d\right) \phi_M\left(x \Big| \frac{d\lambda_1^2}{d-2}\right) \nonumber \\ 
& = -\frac{Z_M(\frac{d\lambda_1^2}{d-2} )}{Z_M(\lambda_1^2)^{\frac{d-2}{d}}}\frac{2 h}{\lambda_1^5} \left(2\E ||W||^2 -\lambda_1^2 d \right), 
\end{align}
where $W=(W_1, W_2, \ldots, W_d)'$ is sample from the density $\phi_M(\cdot|\frac{d\lambda_1^2}{d-2})$. 

Therefore, by Theorem \ref{EFFSECOND}, the following cases arise:

\begin{itemize}

\item[--] $N^{\frac{1}{2}-\frac{2}{d}}\varepsilon_N\rightarrow 0:$ The limiting power of the test~\eqref{rejregionKNN} is $\alpha$. 

\item[--]  $N^{\frac{1}{2}-\frac{2}{d}}\varepsilon_N \rightarrow h:$ The limiting power of the test~\eqref{rejregionKNN} is 
$$\Phi\left(z_\alpha+ \frac{rp C_{K, 2}}{2 d \sigma_K} \frac{Z_M(\frac{d\lambda_1^2}{d-2} )}{Z_M(\lambda_1^2)^{\frac{d-2}{d}}}\frac{h}{\lambda_1^5} \left(2\E || W|| -\lambda_1^2 d \right) \right).$$

\item[--]  $N^{\frac{1}{2}-\frac{2}{d}}\varepsilon_N\rightarrow \pm \infty$  such that $N^{\frac{2}{d}}\varepsilon_N \rightarrow 0:$ Then the limiting power of the test~\eqref{rejregionKNN} is $1$ or $0$, respectively. 

\item[--] $N^{\frac{2}{d}}\varepsilon_N \rightarrow \infty:$ The limiting power of the test~\eqref{rejregionKNN} is $1$. 

\end{itemize}

\section{The symmetrized $K$-NN test}
\label{sec:KNNsymmetry}

Here, we discuss how the formulas in Theorem \ref{EFFSECOND} can be modified for the test based on the symmetrized $K$-NN graph \eqref{Tknnsymmetry}. To begin with, note that, as in \eqref{beta0}, 
\begin{align}\label{beta0S}
\E_{H_1}(T_S(\sG(\cZ_N')))=&N_1N_2 \int_{S\times S}  \left(f(x|\theta_1)f(y|\theta_2) + f(x|\theta_2)f(y|\theta_1) \right)\rho^{\theta_1, \theta_2}_{K}(x, y) \mathrm dx\mathrm dy \nonumber \\ 
:=&\frac{N_1N_2}{N^2} \mu_N^{(S)}(\theta_1, \theta_2),
\end{align}
As before, by the chain rule and differentiating under the integral sign gives 
\begin{align*}
\grad \mu_N^{(S)}(\theta_1,  \theta_1)=N^2\int_{S\times S}  \left(f(x|\theta_1)\grad_{\theta_1} f(y|\theta_1) \rho^{\theta_1, \theta_1}_{K}(x, y) +  2 f(x|\theta_1) f(y|\theta_1) \grad_{\theta_1} \rho^{\theta_1, \theta_1}_{K}(x, y) \right) \mathrm dx\mathrm dy.
\end{align*} 
since $$\int_{S\times S} \grad_{\theta_1} f(x|\theta_1) f(y|\theta_1) \rho^{\theta_1, \theta_1}_{K}(x, y) \mathrm dx\mathrm dy =\int_S \E d^\uparrow(x, \cP_{Nf(\cdot|\theta_1)}^x) \grad_{\theta_1} f(x|\theta_1) \mathrm dx=O(e^{-\frac{N}{2}}),$$ 
because $\int_S \grad_{\theta_1} f(x|\theta_1) \mathrm dx=0$ and $| \E d^\uparrow(x, \cP_{Nf(\cdot|\theta_1)}^x)-K|\leq K \P(|\cP_{Nf(\cdot|\theta_1)}| \leq K-1) = K \sum_{j=0}^{K-1} \frac{N^j}{j!} e^{-N} \lesssim_K e^{-\frac{N}{2}}$ (since we define $\cN_K(S)$ to be empty, if $|S| \leq K$).

Then by Lemma \ref{t11limit} and Lemma \ref{t12limit}, with $\varepsilon_N=\frac{h}{N^{\frac{1}{2}-\frac{2}{d}}}$, for some $h \in \R^p\backslash \{\bm 0\}$, 
\begin{align}\label{gradlimitN}
\frac{\varepsilon_N^\top \grad  \mu_N^{(S)}(\theta_1, \theta_1)}{\sqrt N} \rightarrow    \frac{r(1-2q) C_{K, 2}}{4d}  \int_{S}  h^\top \grad_{\theta_1} \left(\frac{ \tr(\mathrm H_x f(x|\theta_1))}{f(x|\theta_1)}  \right) f^{\frac{d-2}{d}}(x|\theta_1) \mathrm d x.
\end{align}

Next, differentiating with respect to $\theta_2$ twice under the integral signs gives, $\mathrm H  \mu_N^{(S)}(\theta_1,  \theta_1) = T_{21}+\frac{3}{2}T_{22}+\frac{3}{2} T_{23},$ where  $T_{21}$, $T_{22}$, and $T_{23}$ are as defined in \eqref{hessiant}. Then by Lemmas  \ref{hessianlimitI}, \ref{hessianlimitII}, and \ref{hessianlimitIII}, with $\varepsilon_N=h N^{-\frac{1}{4}}$, for $h \in \R^p\backslash \{\bm 0\}$ 
\begin{align}\label{hessianlimitS}
\frac{\varepsilon_N^\top \mathrm H  \mu_N(\theta_1, \theta_1) \varepsilon_N}{\sqrt N} \rightarrow  -  \frac{3 r K}{2}  \cdot \E \left( \frac{h^\top\grad_{\theta_1} f(X|\theta_1)}{f(X|\theta_1)} \right)^2.
\end{align}

Therefore, the formula  in \eqref{eq:effhessian} changes to 
$$\Phi\left(z_\alpha+\frac{3r^2 K}{4 \sigma_{S, K}}  \E \left[ \frac{h^\top\grad_{\theta_1} f(X|\theta_1)}{f(X|\theta_1)} \right]^2\right),$$
and the formula in \eqref{eq:effgrad} becomes
$$\Phi\left(z_\alpha- \frac{r (1-2q) C_{K, 2}}{4d\sigma_{S, K}}   \int_{S}  h^\top \grad_{\theta_1} \left(\frac{ \tr(\mathrm H_x f(x|\theta_1))}{f(x|\theta_1)}  \right) f^{\frac{d-2}{d}}(x|\theta_1) \mathrm d x \right),$$
where $\sigma_{S, K}^2$ is the limiting variance of $\frac{1}{N}\Var(T_S(\cN_K(\cZ_N')))$ under the null (which can be obtained from Theorem \ref{TH:CLT_R} applied to the symmetrized $K$-NN graph). This shows that Theorem \ref{EFFSECOND} holds almost verbatim for the symmetrzied $K$-NN test with the formulas of the limiting power modified as above.

\section{Additional simulations}
\label{sec:simulations_normal}

In this section we provide additional simulations to illustrate the theoretical results derived above. 

\subsection{The Normal Location Problem}
\label{sec:normal_location}

\begin{figure*}[h]\vspace{-0.2in}
\centering
\begin{minipage}[l]{0.49\textwidth}
\centering
\includegraphics[width=2.05in]
    {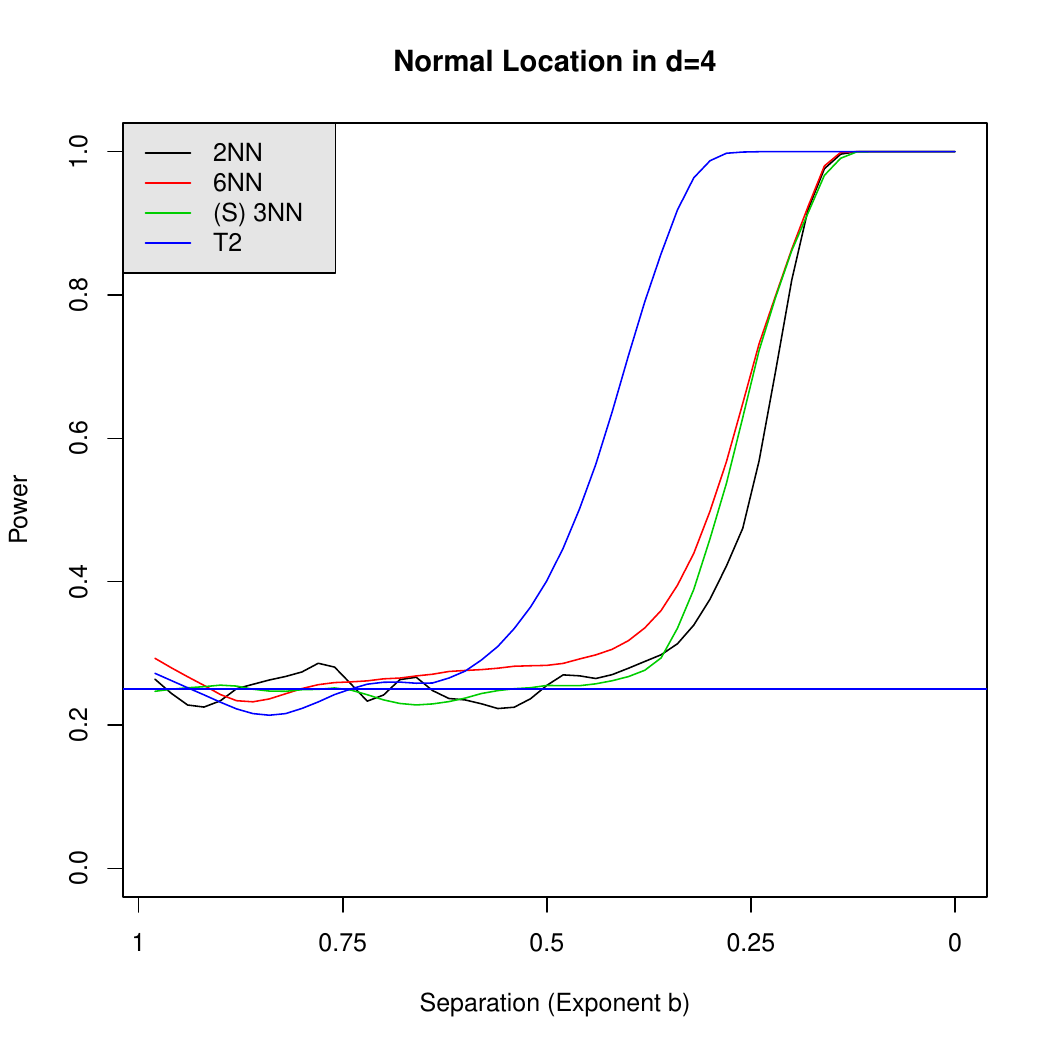}\\
\footnotesize{(a)}
\end{minipage}
\begin{minipage}[l]{0.49\textwidth}
\centering
\includegraphics[width=2.05in]
    {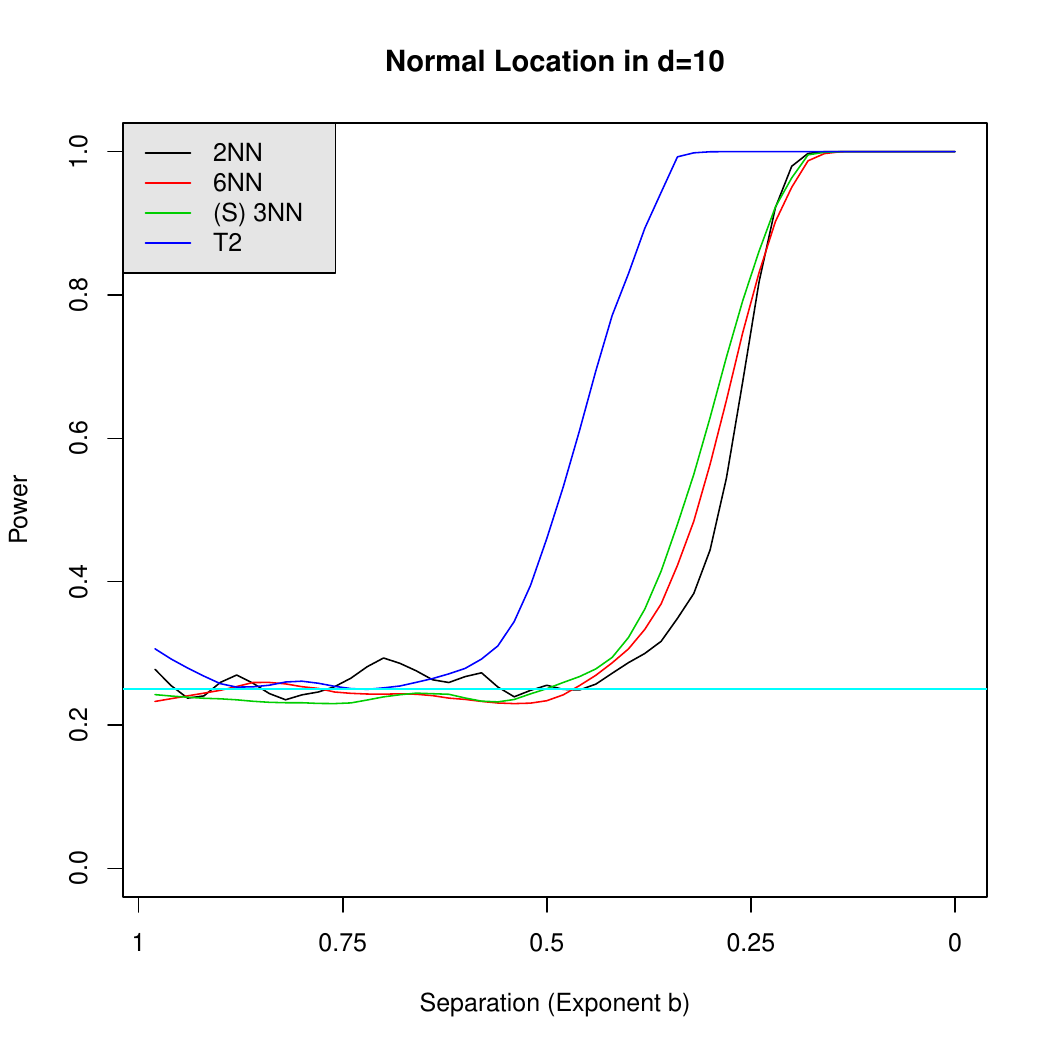}\\
\footnotesize{(b)}
\end{minipage}
\caption{\small{Empirical power in the normal location problem with $N_1=5000$ samples from $N(2 \cdot \bm 1, \mathrm I_d)$ and  $N_2=3000$  samples from $N( 2  \cdot \bm 1 + N^{-b} \cdot \bm h, \mathrm I_d)$, where $b$ varies over a grid of 100 values in $[0, 1]$,  $\bm h=-\bm 1$, and dimension (a) $d=4$ and (b) $d=10$.}}
\label{fig:nlocationII_N} 
\end{figure*}

Here, we provide additional simulations for testing local changes in the normal location family $\{N(\theta, \mathrm I_d): \theta \in \R^d\}$. Figure~\ref{fig:nlocationII_N} shows the empirical power (out of 100 repetitions) of the  different tests with $N_1=5000$ samples from $N(2 \cdot \bm 1, \mathrm I_d)$ and  $N_2=3000$  samples from $N( 2 \cdot \bm 1 + N^{-b} \cdot \bm h,  \mathrm I_d)$, where $b$ varies over a grid of 100 values in $[0, 1]$, $\bm h=-\bm 1$, and dimension (a)  $d=4$ and (b) $d=10$. (Here, $N=N_1+N_2=8000$.) The level of the tests are set to $\alpha=0.25$. Note that the power of the tests based on the $K$-NN graphs transitions from $\alpha$ to 1 around $b=0.25$, which corresponds to the rate $N^{-\frac{1}{4}}$,  as predicted by the calculations above. On the other hand, the power of the Hotelling's $T^2$ test transitions from $\alpha$ to 1 around $b=0.5$, which corresponds to the parametric rate of $N^{-\frac{1}{2}}$. Note that these plots are qualitatively the same as those in Figure~\ref{fig:nlocationII} in Section \ref{sec:nlocation}. This is because, for the normal location problem, the detection threshold does not depend on the dimension and the choice of the direction (recall discussion in Section \ref{sec:nlocation}).

\subsection{The Spherical Normal Problem}
\label{sec:normal_scale}

In this section additional simulations for testing local changes in the spherical normal scale family $\{N(0, \sigma^2 \mathrm I_d): \sigma>0\}$ are presented. Here, we zoom in to the different thresholds obtained in Theorem \ref{EFFSECOND} and compare the power of the different tests.

\begin{figure*}[h]
\centering
\begin{minipage}[l]{0.495\textwidth}
\centering
\includegraphics[width=2.05in]
    {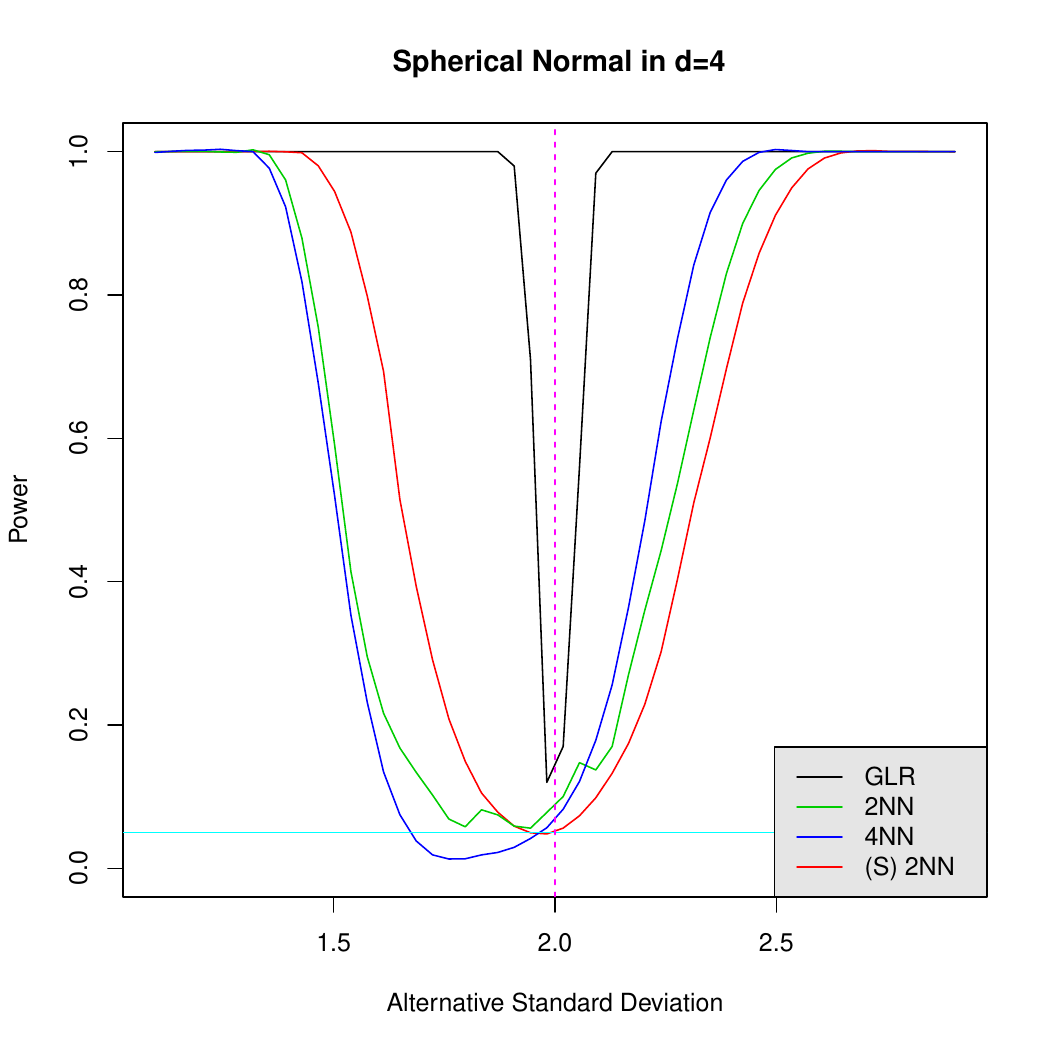}\\
\small{(a)}
\end{minipage}
\begin{minipage}[c]{0.495\textwidth}
\centering
\includegraphics[width=2.05in]
    {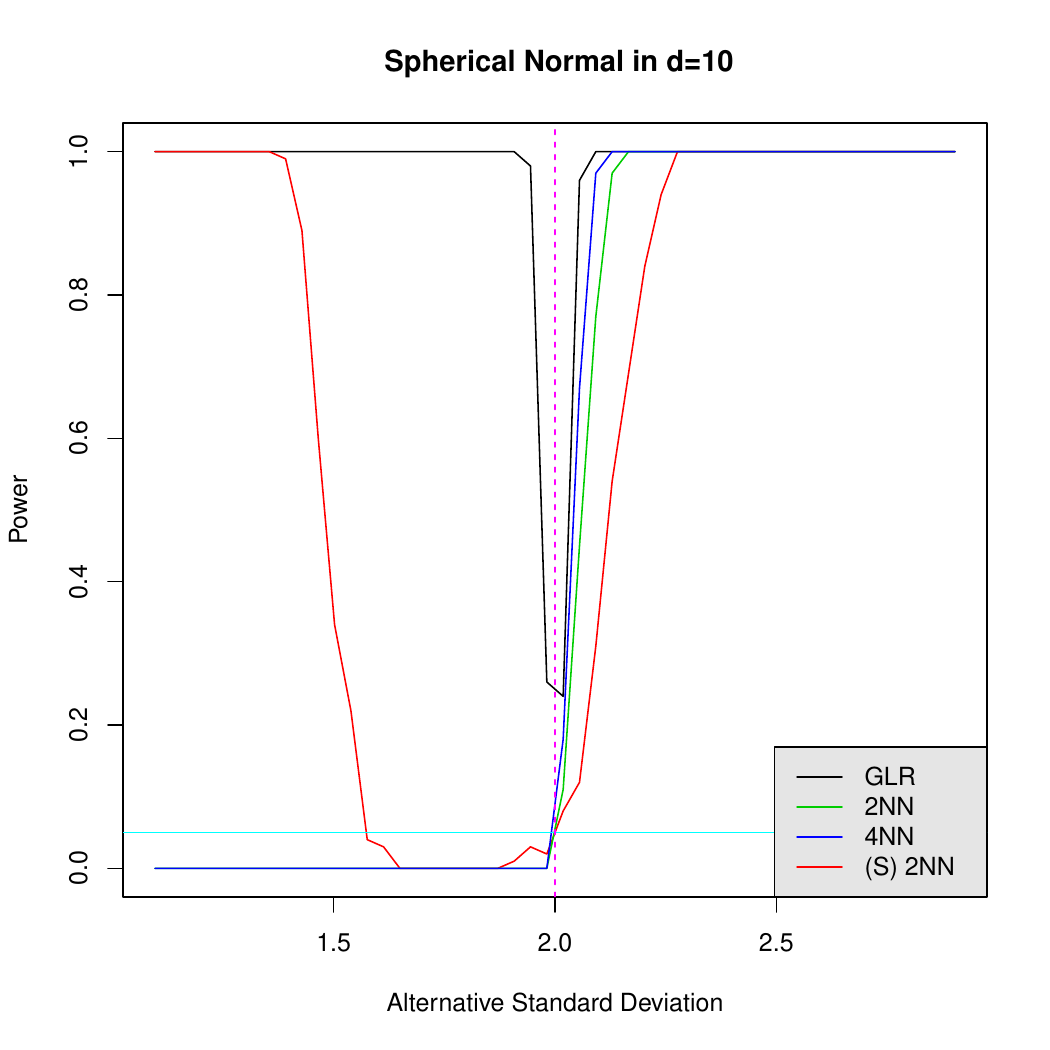}\\
\small{(b)}
\end{minipage}
\caption{\small{Empirical power for the spherical normal problem against $h N^{-\frac{1}{4}}$ alternatives, as a function of $h$, in dimension (a) $d=4$ and (b) $d=10$. The horizontal axis shows the scale parameter $\sigma$ under the alternative.}}
\label{fig:nscaleI}
\end{figure*}
%

Figure~\ref{fig:nscaleI} shows the empirical power (out of 100 repetitions) of the tests based on the 2-NN and 4-NN graphs, the test based on the symmetrized 2-NN graph, and the GLR test, with $N_1=10000$ samples from $N(0, 2^2 \cdot \mathrm I_d)$ and  $N_2=5000$  samples from $N(0, (2 + h N^{-\frac{1}{4}})^2 \mathrm I_d)$, over a grid of 50 values of $h$ in $[-10, 10]$ in (a) dimension 4 and (b) dimension 10.  (Here, $N=N_1+N_2=15000$.)  As before, in both cases, the GLR test has the highest power (going to 1). In dimension 4 (which is less than the critical dimension 8), the tests based on the NN graphs have non-trivial limiting local power as a function of $h$.  However, in dimension 10 (Figure \ref{fig:nscaleI}(b)), there is a drastic change in the shape of the limiting power curve:

\begin{description}

\item[$h>0$] In this case, the power of all the tests quickly increase to 1 with increasing $h$. This is expected because, recalling the discussion in Section \ref{sec:snormal}, for $d \geq 10$ and $h >0$, the detection threshold of the $K$-NN test is at $N^{-\frac{1}{2}+\frac{2}{d}}=N^{-\frac{3}{10}} \ll N^{-\frac{1}{4}}$, and the detection threshold of the GLR test is at $N^{-\frac{1}{2}} \ll N^{-\frac{1}{4}}$. 

\item[$h<0$] Here, there is a region where the $K$-NN based tests have zero power, illustrating the non-monotonicity of the power function and the asymptotic biasedness phenomenon discussed in Section \ref{sec:snormal}. This is because, in this case, the detection threshold is at $N^{-\frac{2}{d}}=N^{-\frac{1}{5}} \gg N^{-\frac{1}{4}}$. 

\end{description}

%

To observe the emergence of these phase transitions more clearly, we now zoom in at the different thresholds with larger sample sizes. Hereafter, we set the level of tests to $\alpha=0.15$.

\begin{figure*}[h]
\centering
\begin{minipage}[l]{0.48\textwidth}
\centering
\includegraphics[width=2.05in]
    {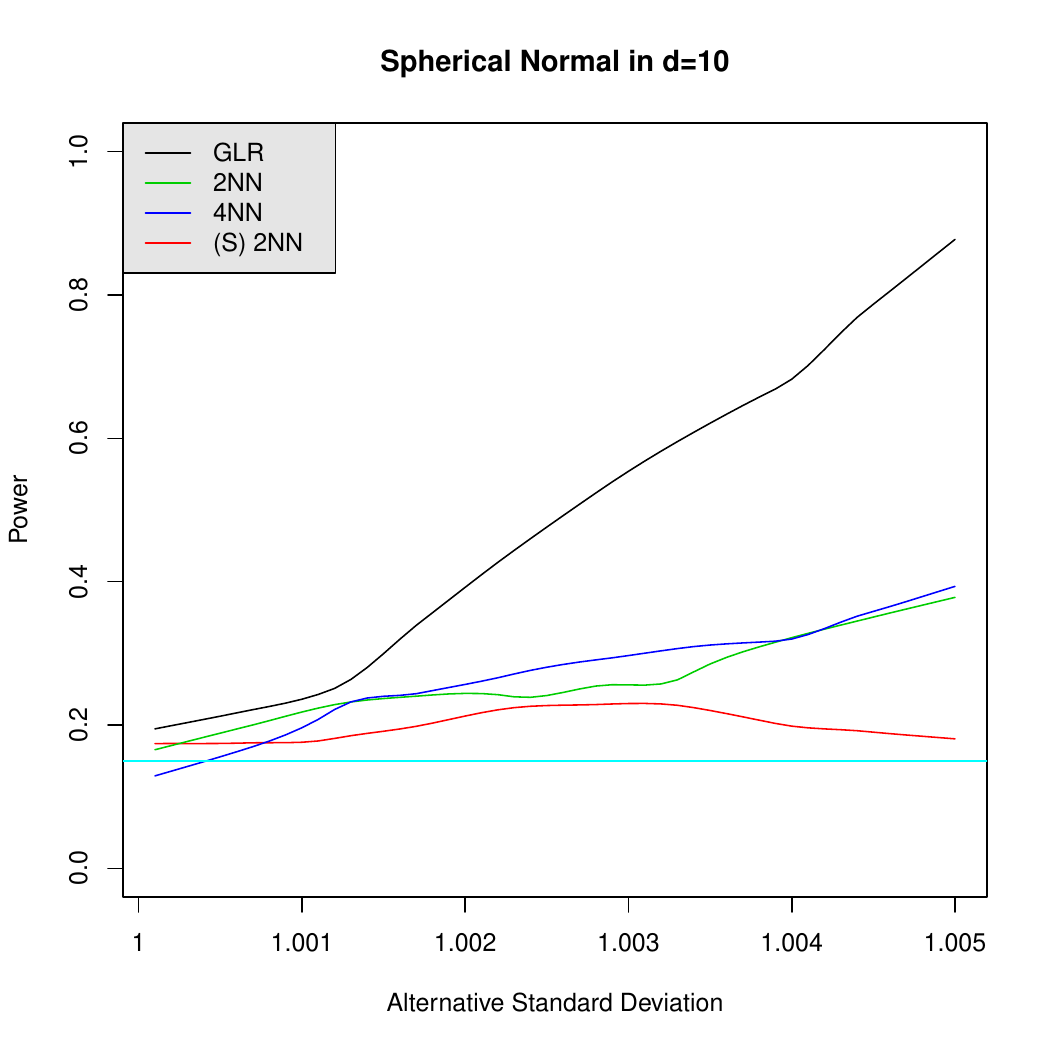}\\
\small{(a)}
\end{minipage}
\begin{minipage}[c]{0.48\textwidth}
\centering
\includegraphics[width=2.05in]
    {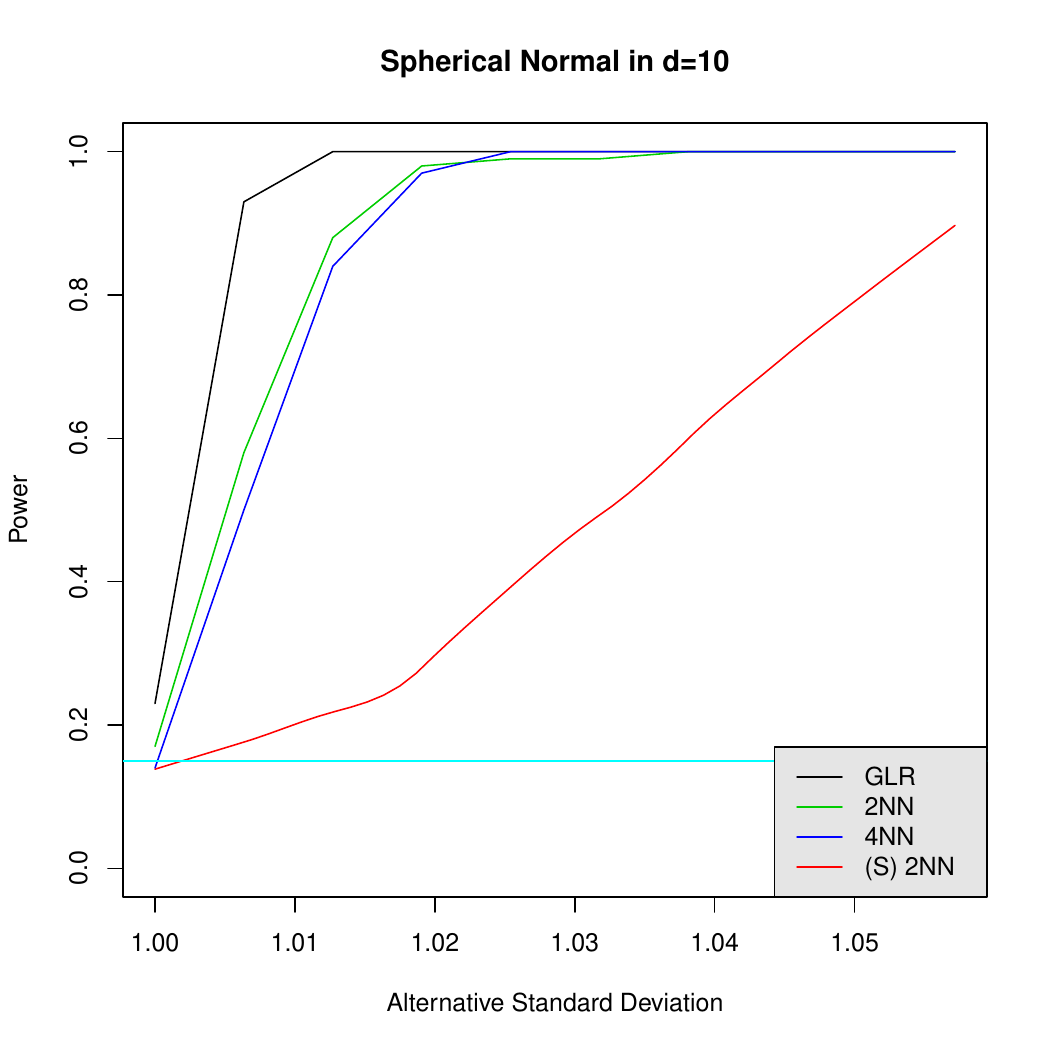}\\
\small{(b)}
\end{minipage}
\caption{\small{Empirical power for the spherical normal problem in dimension $d=10$  against (a) $h N^{-\frac{1}{2}}$ and (b) $h N^{-\frac{3}{10}}$ alternatives, as $h$ varies over a grid values in $[0, 2]$. The horizontal axis shows the scale parameter $\sigma$ under the alternative.}}
\label{fig:nscaleII_N}
\end{figure*}
%

\begin{description}

\item[$h>0$] Figure~\ref{fig:nscaleII_N} shows the empirical power (out of 100 repetitions) in dimension 10 of the different tests, with $N_1=80000$ samples from $N(0, \mathrm I_d)$ and  $N_2=60000$  samples from $N(0, (1 + h N^{-b})^2 \mathrm I_d)$, over a grid of 10 values of $h$ in $[0, 2]$. Figure \ref{fig:nscaleII_N}(a) shows the case $b=0.5$ (which corresponds to parametric rate of detection). Here, as excepted, only the parametric GLR test has power increasing  to 1. On the other hand, the power of the $K$-NN tests are significantly lower (closer to the level of the test). Note that the power of the 2-NN and the 4-NN tests do increase slightly with the number of neighbors $K$ and $h$, but this becomes less significant when the experiment is repeated with larger sample sizes (see the table in Figure \ref{fig:nscaleII_table}(a)).   

In Figure \ref{fig:nscaleII_N}(b) we choose $b=0.3$, thus zooming in at the predicted detection threshold of $N^{-0.3}$, for $h>0$ (recall the discussion from Section \ref{sec:snormal}). Here, the power of the  $K$-NN tests  quickly increase to 1.  The same experiment, with $b=0.5$ and $b=0.3$, is repeated with larger sample sizes $N_1=400000$ and $N_2=300000$, and the results are presented in the tables in Figure \ref{fig:nscaleII_table}, which further validate the theoretical findings.

\begin{figure*}[h]
\centering
\begin{minipage}[l]{0.48\textwidth}
\centering
\small
\begin{tabular}{c|cccc}
\hline
$h$ & 0 & 0.5 & 1 & 1.5 \\ 
\hline
\hline
2-NN& 0.12 & 0.12 & 0.12 & 0.16 \\
\hline
4-NN& 0.16 & 0.2 & 0.22 & 0.24  \\
\hline
2-(S)NN & 0.12 & 0.14 & 0.12 & 0.16 \\
\hline 
GLR & 0.13 & 0.21 & 0.49 & 0.61 \\
\hline
\end{tabular}\\
\small{(a) $\sigma=1 + \frac{h}{N^{0.5}}$ under the alternative}
\end{minipage}
\begin{minipage}[c]{0.48\textwidth}
\centering
\small 
\begin{tabular}{c|cccc}
\hline
$h$ & 0 & 0.5 & 1 & 1.5 \\ 
\hline
\hline
2-NN& 0.14 & 0.86 & 1 & 1 \\
\hline
4-NN& 0.08 & 1 & 1 & 1  \\
\hline
2-(S)NN & 0.24 & 0.26 & 0.44 & 0.8 \\
\hline 
GLR & 0.22 & 1 & 1 & 1 \\
\hline
\end{tabular} \\
\small{(b) $\sigma=1 + \frac{h}{N^{0.3}}$ under the alternative}
\end{minipage}
\caption{\small{Empirical power for the spherical normal problem in dimension $d=10$  for (a) $h N^{-\frac{1}{2}}$ and (b) $h N^{-\frac{3}{10}}$ alternatives, for varying $h$, with sample sizes $N_1=400000$ and $N_2=300000$.}}
\label{fig:nscaleII_table}
\end{figure*}
%

\normalsize

\item[$h<0$]  Figure~\ref{fig:nscaleI_N} shows the empirical power in dimension 10 of the different tests, with $N_1=80000$ samples from $N(0, \mathrm I_d)$ and  $N_2=60000$  samples from $N(0, (1 + h N^{-b})^2 \mathrm I_d)$, over a grid of 10 values of $h$ in $[0, 2]$.  Figure~\ref{fig:nscaleI_N}(a) zooms in at the parametric threshold ($b=0.5$), where only the GLR test has power increasing  to 1, while the power of the $K$-NN tests remain close to the level (decreasing slightly with increasing $h$). In Figure~\ref{fig:nscaleI_N}(b) we take $b=\frac{1}{2}-\frac{2}{d}=0.3$, where the $K$-NN tests have power decreasing from the level 0.15 down to zero, as predicted by Theorem 4.2. In Figure~\ref{fig:nscaleI_N}(c) we zoom in at $N^{-0.2}$, that is, $b=\frac{2}{d}=0.2$. This is the detection threshold for the $K$-NN test for $h<0$, which we observe in the plots where the power sharply transitions from 0 to 1. 
\end{description}

\begin{figure*}[h]
\centering
\begin{minipage}[l]{0.31\textwidth}
\centering
\includegraphics[width=2.05in]
    {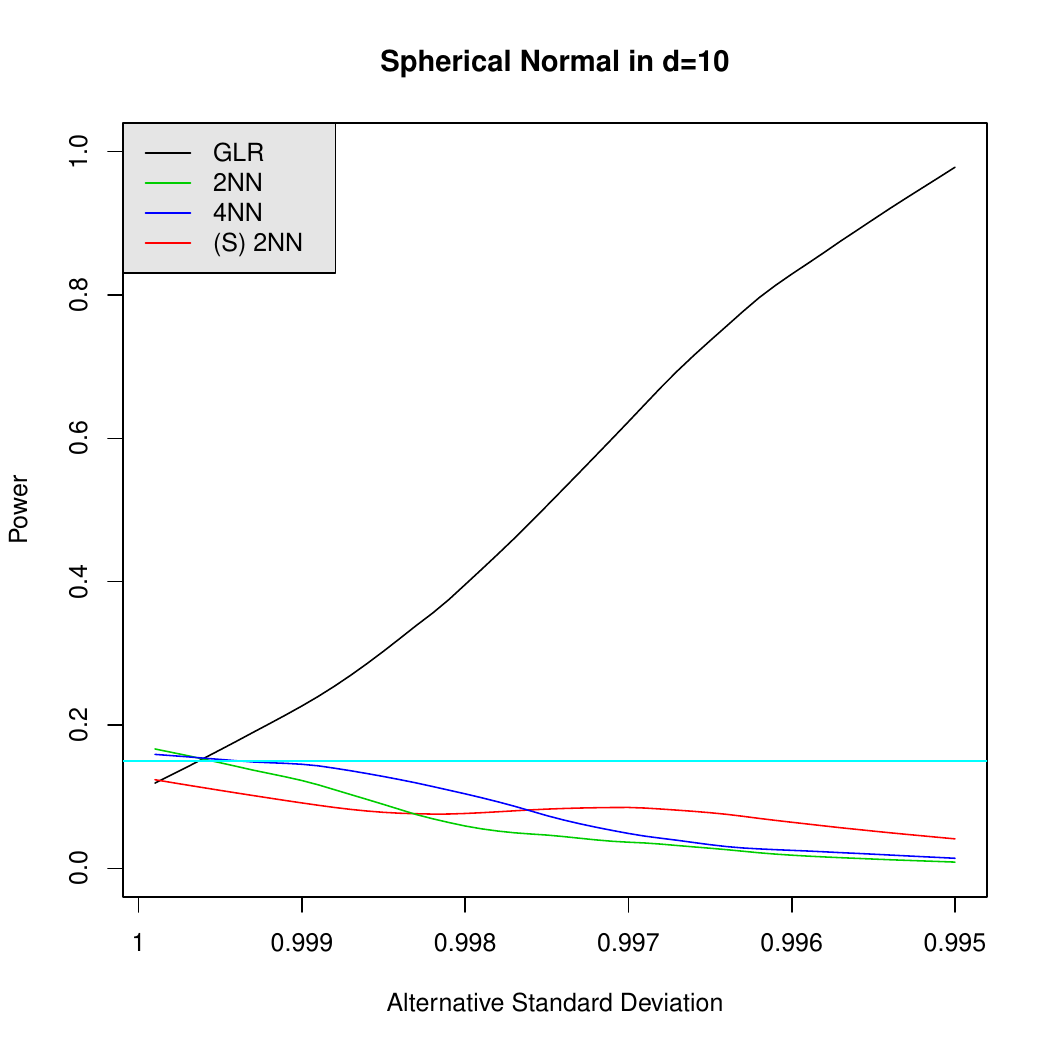}\\
\small{(a)}
\end{minipage}
\begin{minipage}[c]{0.31\textwidth}
\centering
\includegraphics[width=2.05in]
    {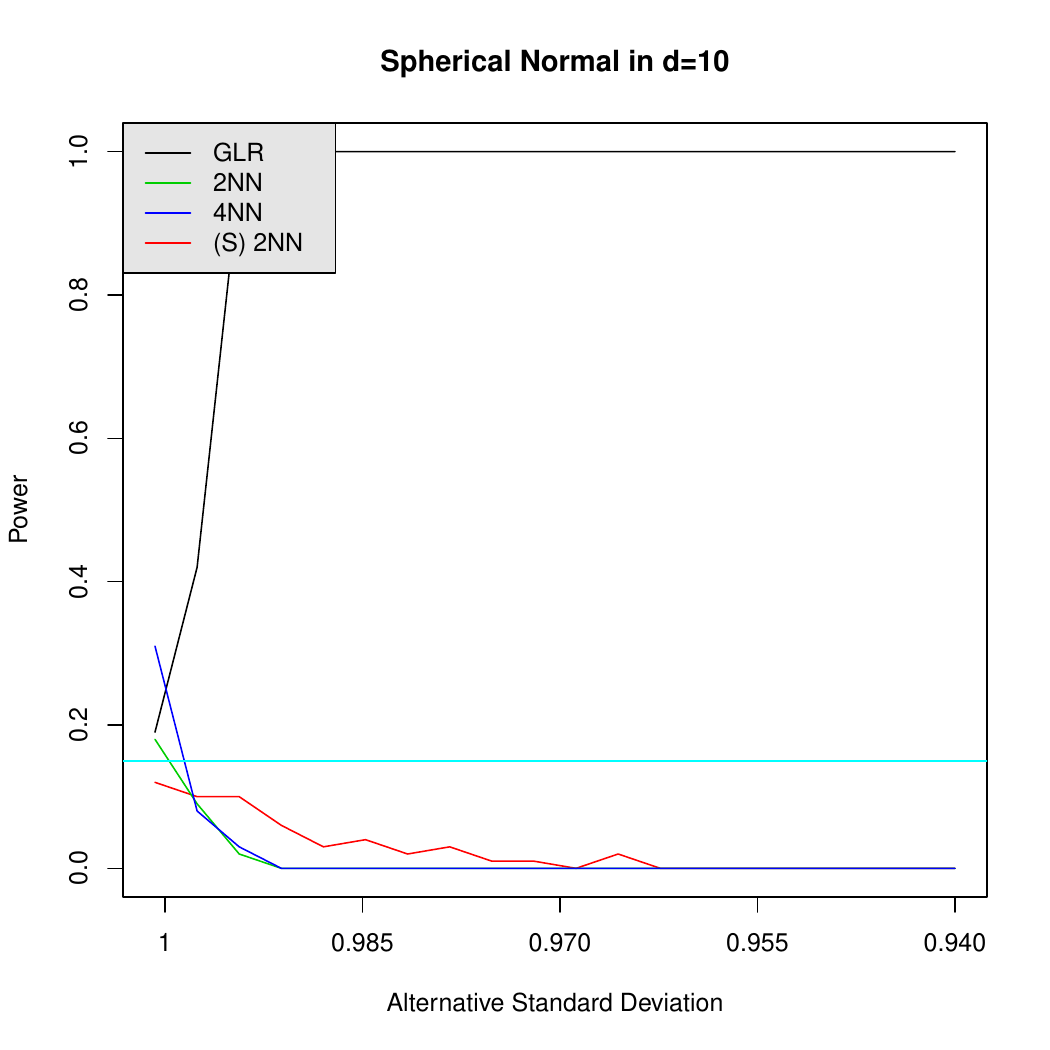}\\
\small{(b)}
\end{minipage}
\begin{minipage}[r]{0.31\textwidth}
\centering
\includegraphics[width=2.05in]
    {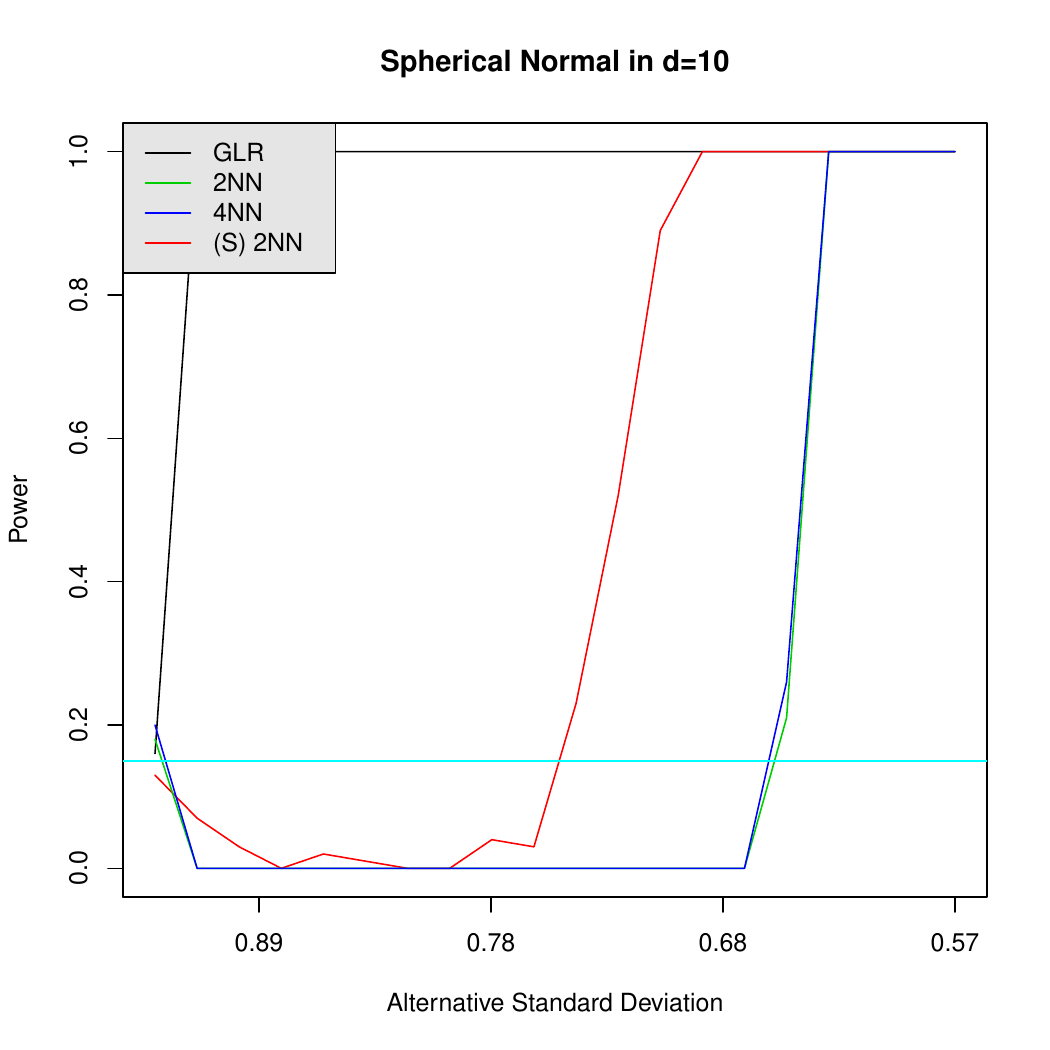}\\
\small{(c)}
\end{minipage}
\caption{\small{Empirical power for the spherical normal problem in dimension $d=10$  against (a) $h N^{-\frac{1}{2}}$, (b) $h N^{-\frac{3}{10}}$, and (c) $h N^{-\frac{1}{5}}$  alternatives, as a function of $h$. The horizontal axis shows the scale parameter $\sigma$ under the alternative.}}
\label{fig:nscaleI_N}
\end{figure*}

\end{document}